\documentclass[11pt]{amsart}
\usepackage{amsmath}
\usepackage{amssymb}
\usepackage{amsfonts}
\usepackage{amsthm}
\usepackage{enumerate}
\usepackage[mathscr]{eucal}
\usepackage{graphicx}
\usepackage[pdftex, pagebackref, bookmarksnumbered, bookmarksopen, colorlinks, citecolor=blue, linkcolor=blue]{hyperref}
\newtheorem{theorem}{Theorem}[section]
\newtheorem{lemma}[theorem]{Lemma}
\newtheorem{corollary}[theorem]{Corollary}

\newtheorem*{conj}{Conjecture}
\theoremstyle{definition}
\newtheorem{definition}[theorem]{Definition}
\newtheorem{prop}[theorem]{Proposition}
\newtheorem{proposition}[theorem]{Proposition}

\newtheorem{claim}[theorem]{Claim}

\theoremstyle{remark}
\newtheorem{remark}[theorem]{Remark}

\numberwithin{equation}{section}

\def\intslash{\rlap{\kern  .32em $\mspace {.5mu}\backslash$ }\int}
\def\qsl{{\rlap{\kern  .32em $\mspace {.5mu}\backslash$ }\int_{Q_x}}}
\renewcommand\Re{\operatorname{Re\,}}
\renewcommand\Im{\operatorname{Im\,}}

\newcommand\off{{\text{\rm off}}}

\newcommand{\R}{{\mathbb R}}
\newcommand\Z{{\mathbb Z}}
\newcommand\M{{\mathcal M}}

\newcommand\F{{\mathcal F}}
\newcommand\A{{\mathcal A}}
\newcommand\Q{{\mathcal Q}}
\newcommand{\mfp}{\mathfrak{p}}
\newcommand\PM{{\mathcal{P}(\mathcal{M})}}
\newcommand\PA{{\mathcal{P}(\mathcal{A})}}

\newcommand\ot{\overline{\otimes}}
\newcommand\wh{\widehat}

\renewcommand\Q{\mathcal Q}

\newcommand{\K}{\mathcal{K}}
\newcommand{\B}{{\mathcal B}}

\newcommand{\norm}[1]{{ \left|  #1 \right| }}
\newcommand{\Norm}[1]{{ \left\|  #1 \right\| }}

\newcommand\Id{\text{\sl Id}}

\newcommand\Rn{{\mathbb R^n}}

\newcommand\supp{{\text{\rm supp}}}

\newcommand{\inn}[2]{\langle #1,#2 \rangle}

\newcommand\zet{\zeta}

\newcommand\vphi{\varphi}

\newcommand{\Be}{\begin{equation}}
	\newcommand{\Ee}{\end{equation}}
\newcommand{\Bes}{\begin{equation*}}
	\newcommand{\Ees}{\end{equation*}}
\newcommand{\Bsp}{\begin{split}}
	\newcommand{\Esp}{\end{split}}
\newcommand{\Bm}{\begin{multline}}
	\newcommand{\Em}{\end{multline}}
\newcommand{\Bea}{\begin{eqnarray}}
	\newcommand{\Eea}{\end{eqnarray}}
\newcommand{\Beas}{\begin{eqnarray*}}
	\newcommand{\Eeas}{\end{eqnarray*}}
\newcommand{\Benu}{\begin{enumerate}}
	\newcommand{\Eenu}{\end{enumerate}}
\newcommand{\Bi}{\begin{itemize}}
	\newcommand{\Ei}{\end{itemize}}

\begin{document}
	
	\title[Noncommutative Maximal Averages]{Noncommutative Maximal Averages over Submanifolds and Variable Hypersurfaces}
	\author{Xudong Lai and Siyu Liu}
	
\address{Xudong Lai:
Institute for Advanced Study in Mathematics\\
Harbin Institute of Technology\\
Harbin
150001\\
China;
Zhengzhou Advanced Research Institute\\
Harbin Institute of Technology\\
Zhengzhou
450000\\
China}
\email{xudonglai@hit.edu.cn\ xudonglai@mail.bnu.edu.cn}
\thanks{This work is supported by National Natural Science Foundation of China (No. 12322107, No. 12271124 and No. W2441002) and Heilongjiang Provincial Natural Science Foundation of China (No. YQ2022A005).}
	
\address{Siyu Liu:
Institute for Advanced Study in Mathematics\\
Harbin Institute of Technology\\
Harbin
150001\\
China}
\email{syliu@stu.hit.edu.cn\ siyulliu@icloud.com}

	\date{\today}
	
	
	\begin{abstract}
		We develop a noncommutative theory of maximal averages over submanifolds. Maximal averaging along curves, surfaces and other lower-dimensional sets is one of the three pillars of classical harmonic analysis, alongside Calder\'on--Zygmund theory and oscillatory integrals, but its noncommutative counterpart has so far remained almost untouched. Working in the noncommutative $L^p$-spaces associated with semifinite von Neumann algebras, we prove local maximal $L^p$ inequalities for averages over a fixed smooth submanifold of finite type at the parameter origin, for every $1<p\leq\infty$, and global maximal $L^p$ inequalities for polynomial parametrizations over all positive scales, without any finite-type assumption and with constants depending only on the degree and the dimensions. We also prove local maximal $L^p$ inequalities for averages over variable hypersurfaces in $\mathbb{R}^n$, $n\geq3$, satisfying a uniform rotational curvature condition, in the range $p>n/(n-1)$. To the best of our knowledge, these are the first noncommutative maximal inequalities for averages over general lower-dimensional submanifolds and polynomial images.

The main novelty lies at the weak $(1,1)$ endpoint. The averaging kernels arising from submanifolds need not satisfy the pointwise smoothness or $L^2$-Dini type condition required by all existing noncommutative Calder\'on--Zygmund theories, so the available smooth-kernel endpoint machinery cannot be applied. We overcome this difficulty by establishing a weak type $(1,1)$ inequality for a regularized auxiliary family adapted to non-isotropic dilations, in which the Calder\'on--Zygmund decomposition based on Cuculescu projections is combined with integral regularity estimates, non-isotropic Littlewood--Paley decomposition and Fourier decay.

Using a Marcinkiewicz interpolation argument, we obtain the maximal $L^p$ inequality for the analytically regularized averages. Interpolation with Fourier-based $L^2$ bounds then recovers the maximal inequalities for the original averages. The variable-hypersurface result relies on a separate argument based on oscillatory $TT^*$ estimates and localization. As applications, we obtain noncommutative maximal ergodic inequalities for trace-preserving actions of $\mathbb{R}^n$ along the geometric families considered above.  Consequently, these maximal ergodic inequalities yield some bilaterally almost uniform convergence results for the corresponding normalized ergodic averages.
	\end{abstract}
	
	\maketitle
	\newpage
	\tableofcontents
	
\section{Introduction and the main results}
\label{sec:introduction}

\subsection{Background and motivation}
\label{subsec:background-motivation}

Noncommutative harmonic analysis has grown rapidly over the past two decades; see, for instance, \cite{Del26,JungeXu2007,Par26,PisierXu2003,Xu07}. Its starting point is a simple question: what becomes of the basic objects in Fourier analysis---maximal operators, singular integrals, square functions---when scalar functions are replaced by operators acting on a Hilbert space? In this setting the role of the underlying measure space is played by a von Neumann algebra $\M$ equipped with a normal semifinite faithful trace $\tau$, and the classical $L^p$-space is replaced by the noncommutative spaces $L^p(\M)$. This is far from being a mere generalization: the theory is deeply connected with quantum mechanics, operator algebras, noncommutative geometry, quantum probability and geometric group theory; see, e.g., 
\cite{ChenXuYin2013,CGPT23,Con94,Del26,GonzalezPerezJungeParcet2021,JungeMei2010,JungeMeiParcet2014,JMP18,Lai22,
Mei2007,MRX22,Par26,PisierXu1997,Var06,XXX16,XiongXuYin2018}.

Within this theory, maximal inequalities for families of averaging operators occupy a central position. They are the basic tool behind convergence results in noncommutative ergodic theory, where maximal boundedness is what allows one to pass from convergence on a dense class to bilaterally almost uniform convergence of the individual averages \cite{JungeXu2007,Yeadon1977}. But the very formulation of a maximal inequality is already a noncommutative difficulty: a family of positive operators which do not commute admits no literal pointwise supremum. The way around this was found in Pisier's vector-valued theory \cite{Pisier1998} and in Junge's extension of it to general von Neumann algebras \cite{Junge2002}: strong maximal estimates are formulated in the spaces $L^p(\M;\ell^\infty)$, where a family is controlled through a uniform majorant rather than through a supremum. On this foundation the theory was built up step by step. The noncommutative martingale inequalities of Pisier and Xu \cite{PisierXu1997} provided the probabilistic background. At the weak $(1,1)$ endpoint, Cuculescu \cite{Cuculescu1971} and Yeadon \cite{Yeadon1977} controlled martingale and ergodic averages by compression with a single common projection---the noncommutative substitute for the scalar level-set argument. Junge and Xu \cite{JungeXu2007} proved strong maximal ergodic inequalities and designed the interpolation scheme that links them to these endpoint estimates, and Mei's operator-valued Hardy and BMO theory \cite{Mei2007} produced Hardy--Littlewood maximal inequalities by relating spatial averages to martingales.

The same period saw the construction of a genuinely noncommutative Calder\'on--Zygmund theory, the framework in which the endpoint estimates behind all the maximal results just recalled are proved---and in which the present paper works as well. Parcet \cite{Parcet2009} first built a Calder\'on--Zygmund decomposition adapted to operator-valued functions, replacing scalar stopping regions by Cuculescu projections, developed the pseudo-localization arguments and deduced weak-type endpoint bounds for singular integrals whose kernels enjoy the usual pointwise smoothness. The price of noncommutativity is the appearance of off-diagonal terms: orthogonal projections need not commute with the operator sitting between them, and such terms have no classical analogue. Cadilhac, Conde-Alonso and Parcet \cite{CadilhacCondeAlonsoParcet2022} later proposed an alternative  noncommutative Calder\'on--Zygmund decomposition that removes the off-diagonal good functions altogether, at the expense of a more complicated bad function, and a series of subsequent works \cite{HongLaiXu2023,Lai2024,LaiXu2026,Lai26a} pushed the theory towards more general rougher kernels. In short, noncommutative Calder\'on--Zygmund theory is by now a powerful machine in this field.

On the other hand, recall that classical Euclidean harmonic analysis rests on three pillars, see e.g. \cite{Stein1993}. The first is the Calder\'on--Zygmund theory of singular integral operators. The second is the theory of oscillatory integral operators, which underlies such topics as Fourier restriction estimates, summability of Bochner--Riesz means and local smoothing estimates. The third is the theory of maximal averages over submanifolds (maximal Radon transforms) which govern  differentiation and ergodic averaging along curves, surfaces and other lower-dimensional submanifolds. In the noncommutative world much of the first pillar has been established, as recalled above, and the second has begun to attract attention in recent years (see e.g. \cite{Del26,Lai22}). The third pillar, by contrast, is almost entirely missing. To the best of our knowledge, the only existing results in this direction are the operator-valued Bourgain circular maximal theorem of Hong, the first author and Wang \cite{hong2025noncommutative} for $p>2$, obtained through a local smoothing estimate.
Maximal averages over general lower-dimensional submanifolds of finite-type, over polynomial images, and over variable families of hypersurfaces all lie beyond the reach of the current theory.  The purpose of this paper is to develop this missing theory systematically.

The natural frame for this program is the semi-commutative setting: the spatial variable remains Euclidean, so that the geometry of the averaging sets is untouched, while the functions take values in a von Neumann algebra. A natural question is then how the geometry of the averaging sets influences the corresponding noncommutative maximal estimates. This setting keeps the Euclidean geometry that feeds Fourier decay and curvature into the estimates, and for translation-invariant families the transference method connects the resulting bounds with maximal ergodic inequalities for trace-preserving actions of $\R^n$. Yet even though the geometric structure of the averages is unchanged, the passage from scalar to operator-valued functions involves essential difficulties. The obvious one, as we have seen, is how to deal with difficulties arising from the noncommutative formulation of the  maximal inequality itself.

\subsection{Classical averages and geometric setting}
\label{subsec:classical-averages}
To make our results precise, we first recall the classical averaging families.
For each positive integer $k$ and $r>0$, let
\[
 B_k(r):=\{t\in\R^k:|t|<r\}
\]
be the $k$-dimensional Euclidean ball centered at the origin with radius $r$.
For a smooth map $\gamma:\R^k\to\R^n$ with $\gamma(0)=0$, consider the normalized averages
\begin{equation}\label{e:30Mr}
 M_r(f)(x)=\frac1{|B_k(r)|}\int_{|t|<r}f(x-\gamma(t))dt.
\end{equation}
These translation-invariant operators average over pieces of a fixed parametrized submanifold, cut out by expanding parameter balls rather than by dilating a fixed sphere. The classical theory includes the parabola theorem of Nagel, Rivi\`ere and Wainger \cite{NagelRiviereWainger1976} and the theorem of Stein and Wainger \cite[Theorem 1]{SteinWainger1976} for curves whose derivatives at the origin span the ambient space. Averaging operators over submanifolds have subsequently been studied extensively; see, for instance, \cite{Bou86,CNS99,MSS92,STW04,Ste86,TV00}. More generally, a finite-type condition yields Fourier decay for the localized surface measure, while non-isotropic dilations reflect the different orders present in the parametrization.
These methods yield local maximal $L^p$ boundedness for averaging operators over smooth finite-type submanifolds for every \(1<p\leq\infty\). For polynomial parametrizations, global maximal bounds hold in the same range without any finite-type assumption on the image;
see \cite[Chapter XI, Section 2]{Stein1993} and \cite{SteinWainger1978}.

\medskip

A second classical class consists of hypersurface averaging operators, defined by
\[
 A_t(f)(x)=\int_{S_{x,t}}f(y)d\sigma_{x,t}(y).
\]
The underlying geometry of $S_{x,t}$ may depend on the base point, so that the operators are no longer translation invariant.
The relevant nondegeneracy is rotational curvature, which couples the surface and base-point variables, as studied by Phong, Stein and Sogge \cite{PhongStein1989,PhongStein1991,Sog91}. Stein's theorem gives local maximal bounds for uniformly nondegenerate scaled families in the range $p>n/(n-1)$, $n\geq3$ \cite[Chapter XI, Section 3, Theorem 2]{Stein1993}.
Thus finite-type governs averages over a fixed parametrized submanifold, whereas rotational curvature controls hypersurfaces varying with the base point.
Their operator-valued counterparts form the two main bodies of this paper: although the analytic mechanisms differ, both concern maximal averaging over lower-dimensional geometric families in the semi-commutative setting, and both lead to noncommutative ergodic consequences.

\subsection{Notation and main results}
\label{subsec:main-results}
We now fix the notation used throughout the paper and state our main results, which fall into three aspects: finite-type and polynomial maximal averages, maximal averages over variable hypersurfaces, and their ergodic consequences.

Let $\M$ be a von Neumann algebra equipped with a normal semifinite faithful trace $\tau$, and let $L^p(\M)$ denote the associated noncommutative $L^p$-spaces. We write $\mathcal P(\M)$ for the projections. The semi-commutative algebra and its trace are
\[
 \A=L^\infty(\R^n)\ot\M,\qquad
 \nu=\int_{\R^n}\otimes\tau.
\]
For $1\leq p<\infty$, $L^p(\A)$ is isometrically identified with $L^p(\R^n;L^p(\M))$. We use the notation
\[
 \Big\|\sup_{i\in I}a_i\Big\|_{L^p(\M)}
 :=\Norm{\{a_i\}_{i\in I}}_{L^p(\M;\ell^\infty(I))}.
\]
For a positive family this is the infimum of $\Norm{a}_{L^p(\M)}$ over all positive $a\in L^p(\M)$ satisfying $a_i\leq a$ for every $i$. In particular, the displayed supremum denotes a norm of a family, rather than an operator supremum.
We use the same convention for $\A$. The precise definitions are recalled in Appendix \ref{app:noncommutative-analysis}. For the detailed introduction of basic concepts in noncommutative analysis, we refer to Appendix \ref{app:noncommutative-analysis}.

\medskip
\noindent{\textbf{(A) Finite-type and polynomial maximal averages}.}
Suppose $n\geq2$ and $1\leq k\leq n-1$, and let $S$ be parametrized by a smooth map $\gamma:\R^k\to\R^n$. We use the following finite-type condition.
\begin{definition}[Finite-type condition]
 \label{def:finite-type def}
 The submanifold $S$ is \emph{of finite-type at $t_0$} if, for every unit vector $\eta\in\R^n$, there is a multi-index $\alpha$ with $|\alpha|\geq1$ such that
 \begin{equation*}
  \left.\partial_t^\alpha\big(\gamma(t)\cdot\eta\big)\right|_{t=t_0}\neq0.
 \end{equation*}
 The least integer $m$ for which one may always choose $1\leq|\alpha|\leq m$ is the \emph{type} at $t_0$. If $\gamma$ is of finite-type on a compact set $V$, its type on $V$ is the maximum of the types at points of $V$.
\end{definition}
\begin{remark}
The finite-type condition is equivalent to the condition that the vectors $\partial_t^\alpha\gamma(t_0)$, $1\leq|\alpha|\leq m$, span $\R^n$ for some finite $m$. This is the finite-order contact condition with affine hyperplanes used in \cite[Chapter VIII, Section 3.2]{Stein1993}. For example, the parabola $\gamma(t)=(t,t^2)$ has type two.
\end{remark}

Now we state our first result.

\begin{theorem}
 \label{thm:k-submanifold case}
Assume that $\gamma$ is of finite-type at $0$ and that $\gamma(0)=0$. Let $M_r$ be defined in \eqref{e:30Mr}. Then for each $1<p\leq\infty$, we have
 \[
  \Big\|\sup_{0<r<1}M_r(f)\Big\|_{L^p(\A)}
  \lesssim_{p,\gamma}\Norm{f}_{L^p(\A)}.
 \]
 That is, $\{M_r\}_{0<r<1}$ extends boundedly from $L^p(\A)$ to $L^p(\A;\ell^\infty((0,1)))$.
\end{theorem}

Notice that only the behavior near the parameter origin enters the finite-type assumption. Within $0<r<1$, these scales bounded away from zero are controlled by convolution with a fixed finite measure. For polynomial maps, the next result is global in scale and requires no finite-type assumption on the image. Let $P=(P_1,\dots,P_n):\R^k\to\R^n$, where each $P_i$ is a real polynomial of degree at most $d$, and set
\[
 M_r^P(f)(x)=\frac1{|B_k(r)|}\int_{|t|<r}f(x-P(t))dt.
\]
\begin{theorem}
 \label{thm:polynomial case}
Let $M_r^P$ be defined as above. For each $1<p\leq\infty$, we have
 \[
  \Big\|\sup_{r>0}M_r^P(f)\Big\|_{L^p(\A)}
  \lesssim_{p,d}\Norm{f}_{L^p(\A)}.
 \]
That is, the map $f\mapsto\{M_r^P(f)\}_{r>0}$ extends boundedly from $L^p(\A)$ to $L^p(\A;\ell^\infty(\R_+))$.
\end{theorem}
Here and below the dimensions are fixed. The constant in Theorem \ref{thm:polynomial case} may depend on the degree of the polynomial $P$.

\medskip
\noindent{\textbf{(B)  Variable hypersurfaces}.}
We next consider averages over hypersurfaces that depend on the base point.
Let $\Phi^t$ be a smooth family of defining functions, smooth up to $t=0$, and put
\[
 \Phi_t(x,y)=\Phi^t\Big(x,x+\frac{y-x}{t}\Big),\qquad
 S_{x,t}=\{y:\Phi_t(x,y)=0\}.
\]
With the smooth compactly supported cutoff and scaled measures (see a precise definition in (\ref{eq:A_t(f) on hypersurface})), define
\[
 A_t(f)(x)=\int_{S_{x,t}}f(y)d\sigma_{x,t}(y),\qquad 0<t\leq1.
\]
The uniform rotational curvature hypothesis is
\begin{equation}\label{e:30rotationcur}
 t^{2n}\operatorname{rotcurv}(\Phi_t)(x,y)\geq c>0
 \quad\text{when }\Phi_t(x,y)=0,\qquad 0<t\leq1,
\end{equation}
with $c$ independent of $t$; the determinant defining $\operatorname{rotcurv}$ and the full setup are given in Section \ref{sec:variable-hypersurfaces}. Our second result is stated as follows.
\begin{theorem}
 \label{thm:variable hypersurfaces intro}
 Suppose $n\geq3$ and the scaled family satisfies the uniform rotational curvature condition (\ref{e:30rotationcur}). Then, for $p>n/(n-1)$,
 \[
  \Big\|\sup_{0<t\leq1}A_t(f)\Big\|_{L^p(\A)}
  \lesssim_p\Norm{f}_{L^p(\A)}.
 \]
That is, the map $f\mapsto\{A_t(f)\}_{0<t\leq1}$ extends boundedly from $L^p(\A)$ to $L^p(\A;\ell^\infty((0,1]))$.
\end{theorem}
The full statement and proof of the above theorem appear in Theorem \ref{thm:A_t bounded for p>n/(n-1)}; the implicit constant may depend on the underlying geometric data of $\Phi^t$.

\medskip
\noindent{\textbf{(C) Ergodic applications and b.a.u. convergence.}}
Our third  result shows that the translation-invariant inequalities above have ergodic consequences. Let $s\mapsto T_s$ be a $w^*$-continuous action of $\R^n$ by trace-preserving $*$-automorphisms of $\M$. For $a\in L^p(\M)$, set
\[
 N_r(a)=\frac1{|B_k(r)|}\int_{|t|<r}T_{\gamma(t)}(a)dt,
 \qquad
 N_r^P(a)=\frac1{|B_k(r)|}\int_{|t|<r}T_{P(t)}(a)dt.
\]
Calder\'on transference gives the following local and global estimates, respectively. See Proposition \ref{prop:transference to ergodic} for more details.
\begin{theorem}
 \label{thm:sup N_r(a) Lp bounded}
 Under the assumptions on $\gamma$ in Theorem \ref{thm:k-submanifold case}, for $1<p\leq\infty$,
 \[
  \Big\|\sup_{0<r<1}N_r(a)\Big\|_{L^p(\M)}
  \lesssim_{p,\gamma}\Norm{a}_{L^p(\M)}.
 \]
\end{theorem}
\begin{theorem}
 \label{thm:sup N_r^P(a) Lp bounded}
 Under the assumptions on $P$ in Theorem \ref{thm:polynomial case}, for $1<p\leq\infty$,
 \[
  \Big\|\sup_{r>0}N_r^P(a)\Big\|_{L^p(\M)}
  \lesssim_{p,d}\Norm{a}_{L^p(\M)}.
 \]
 The constants are independent of the action.
\end{theorem}

For variable hypersurfaces, we work instead on the semi-commutative algebra $\mathcal A$, which retains the spatial variable.
With nonnegative averaging measures, define
\[
 \mathcal N_t(f)(x)=\int_{S_{x,t}}T_{y-x}(f(y))d\sigma_{x,t}(y).
\]
\begin{theorem}
 \label{thm:variable hypersurface ergodic maximal}
 Under the assumptions of Theorem \ref{thm:variable hypersurfaces intro}, suppose that the averaging measures are nonnegative. Then  for $n\geq3$ and $n/(n-1)<p<\infty$ and every $f\in L^p(\A)$,
 \[
  \Big\|\sup_{0<t<1}\mathcal N_t(f)\Big\|_{L^p(\A)}
  \leq C_p\Norm{f}_{L^p(\A)}.
 \]
 The constant is independent of the action.
\end{theorem}
The proof of this theorem utilizes the conjugacy relationship $\mathcal N_t=U^{-1}A_tU$, where $(Uf)(x)=T_x(f(x))$; see Section \ref{subsec:maximal-ergodic-theorem} for more details.

The maximal estimates also extend convergence from dense classes to all $L^p$ elements. The averages $N_r(a)$ converge bilaterally almost uniformly (b.a.u.) to $a$ as $r\to0^+$ for $1<p<\infty$; the same holds for $N_r^P(a)$ when $P(0)=0$.
For the variable averages, convergence to $f$ holds in the stated $p$ range when the nonnegative measures have mass one and are supported in $|y-x|\leq Rt$ for a fixed $R$; these are additional assumptions, which are stated in Corollary \ref{cor:variable ergodic bau at zero}. The definition of b.a.u. convergence is recalled in Appendix \ref{subsec:bau convergence}.

At large scales, the normalized polynomial averages converge b.a.u. along $r=2^m$, $m\to+\infty$, to the projection onto the fixed-point subspace, for $1<p<\infty$, provided $P$ is of finite-type at the origin. This additional hypothesis is not needed for the global maximal inequality. Convergence over the full parameter range $r\to\infty$ remains open.

\subsection{Proof ideas and noncommutative difficulties}
\label{subsec:proof-ideas}
Our proof of the maximal inequalities rests on a complex analytic interpolation argument, whose strategy we now sketch. We work with the analytic family of averaging operators
\[
 A^s_jf=f*\nu^s_j,
\]
where
\[
 \widehat{\nu^s}(\xi)=\wh{d\mu}(\xi)(1+|\xi|^2)^{s/2}
 \quad\text{and}\quad
 \wh{\nu^s_j}(\xi)=\wh{\nu^s}(2^{-j}\circ\xi).
\]
Here $2^{-j}\circ\xi$ is the corresponding non-isotropic dialtion. 
When $s=0$, the operator $A_j^0$ coincides exactly with the original averaging operator we wish to study. The strategy proceeds in four steps. First, we prove the maximal $L^2$ boundedness of the family $\{A_j^s\}_j$ in the range $\Re(s)<\delta$, where $\delta>0$ depends on the manifold that we average on; this is the Fourier-analytic side of the argument, where the decay of $\wh{d\mu}$ enters. Second, we show that $\{A_j^s\}_j$ is of weak type $(1,1)$ in the range $\Re(s)<0$; this is the endpoint step, based on the noncommutative Calder\'on--Zygmund decomposition, and it is where the main difficulties of the paper lie. Third, a noncommutative Marcinkiewicz interpolation between weak type $(1,1)$ and $L^\infty$ bounds yields the maximal $L^p$ boundedness of $\{A_j^s\}_j$ for $1<p<\infty$ and $\Re(s)<0$. Finally, complex analytic interpolation in the parameter $s$ allows us to pass to the critical line $\Re(s)=0$, and we conclude that the original maximal averages $\{A_j^0\}_j$ are bounded on $L^p$ for every $1<p<\infty$.

The $L^2$ boundedness of the averaging operator family $\{A_j^s\}_j$ follows from a square function estimate. The finite-type argument begins, as in the classical theory, with Fourier decay. If $\mu$ is the pushforward of $\eta(t)dt$ under $\gamma$, with $\eta$ a smooth cutoff supported where $\gamma$ has type at most $m$, then
\begin{equation}
 \label{eq:Fourer decay of d mu}
 |\widehat{d\mu}(\xi)|\leq A(1+|\xi|)^{-1/m};
\end{equation}
see \cite[Chapter VIII, Section 3]{Stein1993} and Lemma \ref{lem:decay of Fourier transform of measure} below. An analytic regularization and an operator-valued Plancherel identity give the $L^2$ side of the argument, following the classical square-function method.

For the weak type $(1,1)$ boundedness of the averaging operator family $\{A_j^s\}_j$, our proofs depend crucially on the noncommutative Calder\'on--Zygmund decomposition, and it is at this endpoint stage that the main new difficulties appear. The genuinely noncommutative step is a weak type $(1,1)$ estimate for the regularized auxiliary family, and this is where the noncommutative Calder\'on--Zygmund decomposition enters. Cuculescu projections and Parcet's decomposition replace the scalar stopping regions, but they introduce off-diagonal terms: orthogonal projections need not commute with the operator between them. The serious obstacle, however, is the kernel itself. The averaging kernels arising from submanifolds need not satisfy the pointwise Lipschitz smoothness or the $L^2$-Dini type conditions appearing in \cite{CadilhacCondeAlonsoParcet2022,HongLaiXu2023} that the existing noncommutative Calder\'on--Zygmund theorems require, so the smooth-kernel and maximal-truncation results in the literature (see, e.g., \cite{CadilhacCondeAlonsoParcet2022,HongLaiXu2023,LaiXu2026,Parcet2009}) cannot be applied directly to give the required bound.   A new obstacle  is therefore forced on us: we work instead with the integral regularity of the kernels, combining some localization,  non-isotropic Littlewood--Paley decomposition and the Fourier decay (\ref{eq:Fourer decay of d mu}) to control the off-diagonal terms of the good functions. Some of these ideas originate in the treatment of general averaging operators with rough kernels in \cite{Lai2024,Lai26a}.  For the off-diagonal terms of the bad functions, we apply the H\"ormander-type integral regularity of the kernel. We point out that the proof of the weak type $(1,1)$ estimate is one of the main technical contributions of the paper.

A word on the choice of noncommutative Calder\'on-Zygmund decomposition. Cadilhac, Conde-Alonso and Parcet \cite{CadilhacCondeAlonsoParcet2022} introduced an alternative decomposition that removes the off-diagonal good function at the expense of a more complicated bad function.
It simplifies the \(L^2\) treatment of the good function but changes the structure of the off-diagonal bad function.
For the regularized averaging family considered here, we adopt Parcet's decomposition because its cancellation and martingale-difference structure are better adapted to the nonsmooth nature of the kernels considered here.
A related choice of decomposition for rough kernels is discussed in \cite{Lai2024,Lai26a}.

Noncommutative interpolation then brings us back from the regularized family to the original maximal averages: the weak endpoint concerns the auxiliary regularized family, while the conclusions for the geometric averages are strong bounds for $p>1$. Polynomial maps are treated by lifting to a monomial map and then descending, and smooth finite-type maps by Taylor reduction, following \cite[Chapter XI, Section 2]{Stein1993}.

For variable hypersurfaces the mechanism is different: the argument is based on rotational curvature and on oscillatory $L^2$ estimates obtained by the $TT^*$ method.
An operator-valued Sobolev-type inequality provides maximal control in the averaging parameter.
Near $t=0$, localization in the base point and frequency decomposition keep the estimates uniform across scales. Interpolation in the noncommutative maximal spaces then yields the range $p>n/(n-1)$ for $n\geq3$, without using scalar measurable selection.

\subsection{Organization of the paper}
\label{subsec:paper-organization}
The paper is organized as follows.

Sections \ref{sec:parabola-model}--\ref{sec:lp-boundedness} develop the parabola model, which is the core of the paper: all the noncommutative difficulties already appear in this simplest curved geometry, and the arguments here serve as the template for everything that follows.

Section \ref{sec:parabola-model} introduces the analytically regularized averaging family and states three propositions ($L^2$, weak type $(1,1)$ and $L^p$ boundedness), whose proofs constitute the main content of Sections \ref{sec:l2-boundedness}--\ref{sec:lp-boundedness}.

Section \ref{sec:l2-boundedness} proves the $L^2$ estimates for the regularized family, by some Fourier-analytic properties inherited from the finite-type geometry, the operator-valued Plancherel theorem and the square-function method.

Section \ref{sec:weak-11-boundedness} is the technical heart of the paper: it establishes the weak type $(1,1)$ endpoint for the regularized averaging family by means of the noncommutative Calder\'on--Zygmund decomposition, with the off-diagonal good and bad terms controlled through Littlewood--Paley decomposition, the Fourier decay and integral regularity rather than pointwise kernel smoothness.

Section \ref{sec:lp-boundedness} completes the interpolation argument and recovers strong $L^p$ maximal bounds for the original averages for all $1<p\leq\infty$.

Section \ref{sec:submanifolds} passes from the parabola to general geometries: polynomial maps are first lifted to a monomial map and then treated by descent, and smooth finite-type submanifolds are reduced to the polynomial model by Taylor expansion; this yields Theorems \ref{thm:k-submanifold case} and \ref{thm:polynomial case}.

Section \ref{sec:variable-hypersurfaces} turns to averages over variable hypersurfaces and proves Theorem \ref{thm:variable hypersurfaces intro} under the uniform rotational curvature hypothesis, via oscillatory $TT^*$ estimates and an operator-valued Sobolev inequality.

Section \ref{sec:ergodic-theory} develops the ergodic applications: Calder\'on transference for the translation-invariant families, the conjugacy argument for the variable averages, and the resulting bilaterally almost uniform convergence results, both at small scales and along lacunary large scales.

Appendix \ref{app:noncommutative-analysis} recalls the definitions of the noncommutative $L^p$-spaces and maximal spaces, together with the notion of b.a.u.\ convergence.

Appendix \ref{app:Bessel kernels} collects the properties of Bessel potential multipliers and kernels used in Sections \ref{sec:l2-boundedness}--\ref{sec:submanifolds}.

\medskip

{\itshape Notation.} The letters $C$ and $A$ denote positive constants that may change from line to line. We write $X\lesssim Y$ if $X\leq CY$, and $X\lesssim_\alpha Y$ when the constant is allowed to depend on $\alpha$; fixed dimensions and geometric data are suppressed when no confusion can arise. We write $X\approx Y$ if both $X\lesssim Y$ and $Y\lesssim X$, and $X(r)\sim Y(r)$ when their ratio tends to one in the limit under discussion. Our Fourier transform convention is
\[
 \hat f(\xi)=\int_{\R^d}f(x)e^{-2\pi i\inn{x}{\xi}}dx,
 \qquad
 \check f(x)=\int_{\R^d}f(\xi)e^{2\pi i\inn{\xi}{x}}d\xi.
\]
We also write $\F[f]=\hat f$ and $\F^{-1}[f]=\check f$. For $s\in\R$, $\lfloor s\rfloor$ is the greatest integer not exceeding $s$, and $\chi_E$ denotes the indicator of a set $E$.

\section{A model case: the parabola}
\label{sec:parabola-model}
\par Before proving Theorem \ref{thm:k-submanifold case}, we consider a model example. The simplest case of this problem concerns a smooth curve $\gamma:\R^1\to\R^2$. We consider the averages along $\gamma$
\begin{equation}
	\label{eq:average on parabola}
	M_rf(x)=\frac1{2r}\int_{-r}^r f(x-\gamma(t))dt, \quad r>0.
\end{equation}
Here $|B_1(r)|=2r$, so this agrees with the normalization in Section \ref{sec:introduction}. Let $\gamma=(t,t^2)$ be the parabola in $\R^2$. The parabola is a submanifold of type $2$ at each point. Averages along $(t,t^2)$ provide the simplest model but capture the main features of finite-type conditions for more general curves. 
In this section, we set $\A=\M\ot L^{\infty}(\R^2)$. Our main result is stated as follows.
\begin{theorem}
	Let $\gamma$ be the parabola above and let $M_r(f)$ be given by (\ref{eq:average on parabola}). Then for each $1<p\leq\infty$, we have
	\[ \Big\|{\sup_{r>0}M_rf}\Big\|_{L^p(\A)}\lesssim_p \Norm{f}_{L^p(\A)}. \]
	In other words, the operator $f\to \{M_rf\}_{r>0}$ extends to a bounded operator from $L^p(\A)$ to $L^p(\A;\ell^{\infty})$.
\end{theorem}

\par We first replace the characteristic function cutoff by a smooth cutoff. Let $\eta:\R\to\R$ be a nonnegative smooth function supported on $[-2,2]$ and equal to $1$ on $[-1,1]$. For $j\in\Z$, define
\begin{equation}
	\label{eq:average on parabola eta}
	A_j(f)(x)=2^j\int_{\R}f(x-\gamma(t))\eta(2^jt)dt.
\end{equation}
For positive $f$ and each $r>0$, $M_r(f)$ is bounded by a constant multiple of $A_j(f)$ for some $j\in\Z$, so it suffices to bound $\Norm{\sup_{j\in\Z}A_jf}_p$. More precisely, we prove the following theorem.
\begin{theorem}
	\label{thm:sup A_jf bound}
	Let $\gamma$ be the parabola above and let $A_j(f)$ be given by (\ref{eq:average on parabola eta}). Then for each $1<p\leq\infty$, we have
	\[ \Big\|{\sup_{j\in\Z}A_jf}\Big\|_{L^p(\A)}\lesssim_p \Norm{f}_{L^p(\A)}. \]
\end{theorem}
\par We begin with the $L^2$ estimate. In the commutative setting, a square function argument and Plancherel's identity yield Theorem \ref{thm:sup A_jf bound} for $p=2$ (see \cite[Chapter XI]{Stein1993}). The same method applies in the noncommutative setting. In fact, it yields a stronger $L^2$ result. Note that the operator $A_j$ can be written as a convolution operator, i.e.
\[ A_jf=f*d\mu_j, \]
where $d\mu_j$ is a Borel measure on $\R^2$ given by
\[ \int_{\R^2}gd\mu_j=\int_{\R}g(\gamma(2^{-j}t))\eta(t)dt. \]
Equivalently, $\wh{A_jf}(\xi)=\wh{f}(\xi)\wh{d\mu_j}(\xi)$, and each $\wh{d\mu_j}$ is a smooth function satisfying
\[ \wh{d\mu_j}(\xi_1,\xi_2)=\wh{d\mu_0}(2^{-j}\xi_1,2^{-2j}\xi_2), \]
where $\xi=(\xi_1,\xi_2)\in\R^2$. For convenience, write $d\mu:=d\mu_0$, and use the notation
\[ \delta\circ x=(\delta x_1, \delta^2 x_2) \quad \text{ for }  \delta>0, \ x=(x_1,x_2)\in\R^2 \]
to denote the one-parameter non-isotropic dilation. We then have
\[ \wh{d\mu_j}(\xi)=\wh{d\mu}(2^{-j}\circ\xi). \]
As stated in (\ref{eq:Fourer decay of d mu}) (for details, see Lemma \ref{lem:decay of Fourier transform of measure}), the measure $d\mu$ satisfies the Fourier decay estimate
\begin{equation*}
	|\widehat{d\mu}(\xi)|\lesssim (1+|\xi|)^{-\frac12}.
\end{equation*}
\par For each complex number $s$, we denote by $\nu^s$ the distribution on $\R^2$ with Fourier transform
\[ \widehat{\nu^s}(\xi)=\wh{d\mu}(\xi)(1+|\xi|^2)^{s/2}. \]
Similarly, for each $j\in\Z$, let $\nu^s_j$ be the distribution determined by
\[ \wh{\nu^s_j}=\wh{\nu^s}(2^{-j}\circ\xi), \]
and then define
\[ A^s_jf=f*\nu^s_j. \]
Note that $\nu^0_j=d\mu_j$ and the operators $A^0_j$ coincide with $A_j$.
\par The same $L^2$ argument applies to the more general kernels $\nu_j^s$ $(\Re(s)<1/2)$.
\begin{prop}
	\label{prop:A_j^s L2 bound}
	For every $s\in\mathbb{C}$ with $\Re(s)<1/2$ we have
	\[ \Big\|{\sup_{j\in\Z}A_j^sf}\Big\|_{L^2(\A)}\lesssim_s \Norm{f}_{L^2(\A)}. \]
\end{prop}
\par To prove Theorem \ref{thm:sup A_jf bound}, we also establish the following $L^p$ bound when $\Re(s)<0$.
\begin{prop}
	\label{prop:A_j^s Lp bound}
	For every $s\in\mathbb{C}$ with $\Re(s)<0$ and $1<p\leq\infty$, we have
	\[ \Big\|{\sup_{j\in\Z}A_j^sf}\Big\|_{L^p(\A)}\lesssim_{s,p} \Norm{f}_{L^p(\A)}. \]
\end{prop}
\par To prove Proposition \ref{prop:A_j^s Lp bound}, we need a noncommutative analogue of a weak type $(1,1)$ estimate for $\{A_j^s\}_{j\in\Z}$ for real values of $s$.
More precisely, for a family $\{x_j\}_{j\in I}\subset L^{1,\infty}(\M)$, define the $\Lambda^{1,\infty}(\M;\ell^\infty(I))$ quasi-norm by
\begin{equation*}
	\Norm{\{x_j\}_{j\in I}}_{\Lambda^{1,\infty}(\M;\ell^\infty(I))}=\sup_{\lambda>0}\lambda\inf_{e\in \mathcal{P}(\M)}\left\{\tau(e^\perp):\Norm{ex_je}_\infty<\lambda \text{ for all } j\in I\right\}.
\end{equation*}
This is the noncommutative analogue of the $L^{1,\infty}$-norm for maximal functions. For more details, see Appendix \ref{subsec:nc-vector-valued-spaces}.
\begin{prop}
	\label{prop:A_j^s weak L1 bound}
	For each $s\in\mathbb{C}$ with $\Re(s)<0$, we have
	\begin{equation*}
		\big\|\{A_j^sf\}_{j\in\Z}\big\|_{\Lambda^{1,\infty}(\A;\ell^\infty(\Z))}\lesssim_{s} \Norm{f}_{L^1(\A)}.
	\end{equation*}
	As a consequence, the operator $f\mapsto \{A_j^sf\}_{j\in\Z}$ extends to a bounded operator from $L^1(\A)$ to $\Lambda^{1,\infty}(\A;\ell^{\infty}(\Z))$.
\end{prop}

\par The structure of the proof of Theorem \ref{thm:sup A_jf bound} is as follows. In Sections \ref{sec:l2-boundedness}--\ref{sec:lp-boundedness}, we give the proofs of Propositions \ref{prop:A_j^s L2 bound}, \ref{prop:A_j^s Lp bound} and \ref{prop:A_j^s weak L1 bound}.
Section \ref{sec:l2-boundedness} is devoted to the proof of Proposition \ref{prop:A_j^s L2 bound}. The proof of Proposition \ref{prop:A_j^s weak L1 bound} is highly technical and lengthy; it is presented in Section \ref{sec:weak-11-boundedness}. 
In Section \ref{sec:lp-boundedness}, we use the weak type $(1,1)$ bound in Proposition \ref{prop:A_j^s weak L1 bound}, together with the trivial $(\infty,\infty)$ estimate, to obtain the $L^p$-boundedness (i.e. Proposition \ref{prop:A_j^s Lp bound}) by the noncommutative Marcinkiewicz interpolation arguments. This is carried out in several steps since the operator $A_j^s$ may not be positive. Finally, once we have proved  Proposition \ref{prop:A_j^s Lp bound} and Proposition \ref{prop:A_j^s L2 bound}, noncommutative complex interpolation  yields Theorem \ref{thm:sup A_jf bound} (see Section \ref{sec:lp-boundedness} for details).

\section{$L^2$-boundedness}
\label{sec:l2-boundedness}
\par In this section, we prove Proposition \ref{prop:A_j^s L2 bound} by a square function argument.

We first introduce some basic notation. Let $\mathcal S_\M^+$ denote the set of bounded positive operators with $\tau$-finite support, and let $\mathcal S_\M$ be their linear span. We write $\M_+$ and $L^p_+(\M)$ for the positive cones, and use the same convention for $\A$. To prove the boundedness of a linear operator or a family of averaging operators, it suffices initially, by decomposing a function into a linear combination of four positive functions, to consider positive functions in the usual dense class
\[
 \A_{c,+}=\{f:\R^n\to\mathcal S_\M:
 f\in\A_+,\ \overset{\longrightarrow}{\supp}f\text{ is compact}\},
\]
where $\overset{\longrightarrow}{\supp}f$ denotes spatial support. This argument will be frequently used in subsequent proofs.

\subsection{Noncommutative Hardy-Littlewood maximal functions}
\label{subsec:hardy-littlewood-maximal}
\par The quasi-norm $\rho$ on $\R^2$ adapted to this problem is
\[ \rho(x)=\max\{|x_1|,|x_2|^{1/2}\}, \quad x=(x_1,x_2)\in\R^2. \]
Let $|x|_\infty:=\max\{|x_1|,|x_2|\}$. We have $|x|_\infty=1$ if and only if $\rho(x)=1$, and
\begin{equation}
	\label{eq:rho(x) and |x|_infty}
	\begin{split}
		|x|_\infty\leq \rho(x)\leq |x|_{\infty}^{1/2}, \quad &\text{ when } |x|_\infty\leq 1, \\
		|x|_{\infty}^{1/2}\leq \rho(x)\leq |x|_\infty, \quad &\text{ when } |x|_\infty \geq 1.
	\end{split}
\end{equation}
Note that $\rho$ satisfies the triangle inequality:
\begin{equation}
	\label{eq:triangle ineq for quasinorm}
	\rho(x+y)\leq \rho(x)+\rho(y).
\end{equation}
This follows from $|x_1+y_1|\leq |x_1|+|y_1|$ and $|x_2+y_2|^{1/2}\leq |x_2|^{1/2}+|y_2|^{1/2}$. The quasi-norm $\rho$ is homogeneous with respect to the non-isotropic dilations:
\[ \rho(\delta\circ x)=\delta\rho(x), \quad \delta>0, \ x\in\R^2. \]
We define the associated metric on $\R^2$ (also denoted by $\rho$) by
\[ \rho(x,y):=\rho(x-y)=\max\{|x_1-y_1|,|x_2-y_2|^{1/2}\}, \quad x,y\in\R^2. \]
The corresponding balls on $\R^2$ are given by
\[ B(x,r)=\{y\in\R^2: \rho(x,y)<r\}=\{y\in\R^2: |y_1-x_1|<r, |y_2-x_2|<r^2\}. \]
If $x=0$ we write $B(0,r)=B(r)$.
Let $\tilde{A}_{r}(f)$ be the non-isotropic Hardy-Littlewood averaging operator
\begin{equation}
	\label{eq:average nonisotropic}
	\tilde{A}_r(f)(x)=\frac1{|B(r)|} \int_{B(r)}f(x-y)dy, \quad r>0.
\end{equation}
\begin{theorem}
	\label{thm: H-L maximal nonisotropic}
	For all $\lambda>0$ and $f\in L^1(\A)$, there exists a projection $e\in \PA$ such that
	\begin{equation*}
		\|e\tilde{A}_r(f)e\|\leq \lambda\quad \text{ for all } r>0 \quad \text{ and } \quad \nu(e^\perp)\lesssim \lambda^{-1}\Norm{f}_{L^1(\A)}.
	\end{equation*}
	In other words, the non-isotropic maximal operator $f\mapsto \{\tilde{A}_r(f)\}_{r>0}$ is $L^1(\A)\to \Lambda^{1,\infty}(\A;\ell^{\infty}(\R_+))$ bounded.
\end{theorem}
\begin{proof}
The space $(\R^2,\rho,dx)$ is a doubling metric measure space with Radon measure $dx$: its balls satisfy $|B(x,r)|=4r^3$, so they have positive finite measure and the doubling constant is $8$ since $|B(x,2r)|=8|B(x,r)|$. Applying the result in \cite[Theorem 4.1]{HongLiaoWang2021} to the positive function $f$ gives the required common projection and weak type $(1,1)$ estimate.
The same result also gives strong type $(p,p)$ boundedness of $\{\tilde A_r\}_{r>0}$ for every $1<p<\infty$.
\end{proof}

\medskip

\par For later use, we also recall the $L^p$-boundedness of noncommutative strong maximal functions, which is a consequence of the $L^p$-boundedness of one-dimensional maximal functions in \cite[Chapter 3]{Mei2007}.
For two nondegenerate intervals $I$, $J$ in $\R$ which contain the origin, we define
\begin{equation*}
	A_{I,J}(f)(x)=\frac1{|I\times J|}\int_{I\times J}f(x-y)dy.
\end{equation*}
\begin{theorem}
	\label{thm:strong maximal boundedness}
	For $1<p<\infty$, we have
	\begin{equation*}
		\Big\|\sup_{0\in I, J}A_{I,J}(f)\Big\|_{L^p(\A)}\lesssim_p \Norm{f}_{L^p(\A)}.
	\end{equation*}
	In other words, the strong maximal averaging operator $f\mapsto \{A_{I,J}(f)\}_{0\in I,J}$ is bounded from $L^p(\A)$ to $L^p(\A;\ell^\infty)$.
\end{theorem}

\subsection{$L^2$-boundedness of $\{A_j^s\}_{j\in\Z}$}
\label{subsec:l2-analytic-family}
\par
\par Recall that $A_j(f)$ is defined in (\ref{eq:average on parabola eta}), with the cutoff $\eta$ specified there. Choose a real-valued $\psi \in C_c^\infty(\R^2)$ satisfying
\[ \int_{\R^2} \psi(x_1,x_2)dx_1dx_2=\int_\R \eta(t)dt. \]
This implies $\widehat{d\mu}(0)=\widehat{\nu^s}(0)=\widehat{\psi}(0)$. Then for $j\in\Z$ define
\begin{equation}
	\label{eq:B_j def}
	B_j(f)(x)=\int_{\R^2}f(x_1-2^{-j}y_1,x_2-2^{-2j}y_2)\psi(y_1,y_2)dy_1dy_2.
\end{equation}
Note that $B_j$ is a convolution operator, $B_j(f)=f*\psi_j$, where
\begin{equation}
	\label{eq:psi_j def}
	\psi_j(x_1,x_2)=2^{3j}\psi(2^jx_1,2^{2j}x_2).
\end{equation}
\par Consider the Bessel potential multiplier involved in the kernel of $A_j^s$:
\begin{equation*}
	F_s(\xi)=(1+|\xi|^2)^{\frac s2}, \quad s\in\mathbb{C}.
\end{equation*}
Let $G_s=\F^{-1}[F_s]$ on $\R^n$. For our later use also in Section \ref{sec:submanifolds}, we record the required kernel properties on general $\R^n$; here the parabola case corresponds only to the result of $n=2$ below.
\begin{lemma}
	\label{lem:Bessel kernel properties}
	 For every real $s$, $G_s$ is a real radial tempered distribution. If $\Re(s)<0$, it is represented by a radial $L^1$-function, smooth away from the origin. For real $s<0$, this function is positive and radially decreasing, has integral $1$, and  its first derivatives decay exponentially at infinity. Moreover, if $s=\sigma+it$ with $\sigma<0$, then
	\begin{equation}
		\label{eq:G_s domination}
		|G_s(x)|\leq \frac{\Gamma(-\sigma/2)}{|\Gamma(-s/2)|}G_\sigma(x), \qquad x\neq0.
	\end{equation}
\end{lemma}
For a proof of Lemma \ref{lem:Bessel kernel properties}, as well as the integral representation and precise asymptotics of $G_s(x)$, see Appendix \ref{app:Bessel kernels}.

\medskip

We first give the proof of Proposition \ref{prop:A_j^s L2 bound} for real $s<1/2$.
\begin{proof}[Proof of Proposition \ref{prop:A_j^s L2 bound} for real $s<1/2$.] By the definitions of $A_j^s$ and $B_j$, we have
\begin{equation*}
	A_j^sf(x)=f*\nu_j^s(x),  \quad \widehat{A_j^sf}(\xi)=\widehat{f}(\xi)\widehat{\nu^s}(2^{-j}\circ\xi),
\end{equation*}
\begin{equation*}
	B_jf(x)=f*\psi_j(x), \quad \widehat{B_jf}(\xi)=\widehat{f}(\xi)\widehat{\psi}(2^{-j}\circ\xi).
\end{equation*}
By the triangle inequality, we have
\begin{equation*}
	\Big\|{\sup_j A_j^s(f)}\Big\|_{L^2(\A)}\leq \Big\|{\sup_j B_j(f)}\Big\|_{L^2(\A)}+\Big\|{\sup_j \big(A_j^s(f)-B_j(f)\big)}\Big\|_{L^2(\A)}.
\end{equation*}
Define
\begin{equation}
	\label{eq:S(f) def}
	S(f)(x)=\bigg(\sum_{j\in\Z}|A_j^sf(x)-B_jf(x)|^2\bigg)^{1/2}
\end{equation}
as the square function. Then
\begin{equation*}
	|A_j^s(f)-B_j(f)|^2\leq S(f)^2.
\end{equation*}
For real $s$, the multiplier $\widehat{\nu^s}-\widehat{\psi}$ is Hermitian symmetric: its value at $-\xi$ is the complex conjugate of its value at $\xi$, since $d\mu$ and $\psi$ are real and $F_s$ is real and even. The same holds after dilation, so $A_j^s-B_j$ has a real tempered distribution kernel and preserves self-adjointness. Thus, for $f\in\A_{c,+}$, we have
\begin{equation*}
	-S(f) \leq -|A_j^s(f)-B_j(f)| \leq A_j^s(f)-B_j(f) \leq |A_j^s(f)-B_j(f)| \leq S(f).
\end{equation*}
By the definition of the maximal norm,
\[ \Big\|{\sup_j \big(A_j^s(f)-B_j(f)\big)}\Big\|_{L^2(\A)}\leq\Norm{S(f)}_{L^2(\A)}, \]
and therefore
\begin{equation}
	\label{eq:sup A_j<sup B_j +S(f)}
	\Big\|{\sup_j A_j^s(f)}\Big\|_{L^2(\A)}\leq \Big\|{\sup_j B_j(f)}\Big\|_{L^2(\A)}+\Norm{S(f)}_{L^2(\A)}.
\end{equation}
\par Now by (\ref{eq:sup A_j<sup B_j +S(f)}), it suffices to prove
\begin{equation*}
	\Big\|{\sup_j B_j(f)}\Big\|_{L^2(\A)}\lesssim\Norm{f}_{L^2(\A)} \quad \text{and} \quad \Norm{S(f)}_{L^2(\A)}\lesssim \Norm{f}_{L^2(\A)}.
\end{equation*}
The first estimate follows from the strong type $(2,2)$ estimate in Theorem \ref{thm: H-L maximal nonisotropic}. For the $L^2$-estimate of the square function $S(f)$, we use the operator-valued Plancherel identity (see Appendix \ref{subsec:nc-lp-spaces}):
\begin{equation*}
	\begin{split}
		\Norm{S(f)}_{L^2(\A)}^2 &= \sum_{j\in\Z}\big\|{A_j^sf-B_jf}\big\|_{L^2(\A)}^2
		=\sum_{j\in\Z}\big\|{\widehat{A_j^sf}-\widehat{B_jf}}\big\|_{L^2(\A)}^2 \\
		&=\sum_{j\in\Z}\int_{\R^2}\big\|{\widehat{A_j^sf}(\xi)-\widehat{B_jf}(\xi)}\big\|_{L^2(\M)}^2 d\xi \\
		&=\sum_{j\in\Z}\int_{\R^2}\big\|{\hat{f}(\xi)}\big\|_{L^2(\M)}^2\big|{\widehat{\nu^s}(2^{-j}\circ\xi)-\widehat{\psi}(2^{-j}\circ\xi)}\big|^2d\xi \\
		&\lesssim\Norm{f}_{L^2(\A)}^2.
	\end{split}
\end{equation*}
The last step follows from the following estimate: for each real $s<1/2$ we have
\begin{equation}
	\label{eq:sum nu-psi estimate}
	\sum_{j\in\Z}\big|{\widehat{\nu^s}(2^{-j}\circ\xi)-\widehat{\psi}(2^{-j}\circ\xi)}\big|^2\leq \text{Const}, \quad \text{uniformly in } \xi.
\end{equation}
The constant above depends only on $s$. The proof of the estimate (\ref{eq:sum nu-psi estimate}) depends only on the decay of $|\widehat{\nu^s}(\xi)-\widehat{\psi}(\xi)|$ as $\rho(\xi)\to0$ and $\rho(\xi)\to\infty$. More precisely,
\begin{equation}
	\label{eq:decay of FT of nu^s-psi}
	\begin{cases}
		|\widehat{\nu^s}(\xi)-\widehat{\psi}(\xi)|\leq A_s\rho(\xi),  &\text{ if } \rho(\xi)\leq1; \\
		|\widehat{\nu^s}(\xi)-\widehat{\psi}(\xi)|\leq A_s\rho(\xi)^{-1/2+s},  &\text{ if } \rho(\xi)\geq1.
	\end{cases}
\end{equation}
The estimates (\ref{eq:decay of FT of nu^s-psi}) can be found in Stein \cite[Chapter XI, Section 2.2.1]{Stein1993}.
For completeness, we give a proof to see how to use the finite type condition in our arguments. By the definition of $\psi$ and $\nu^s$, we see that $\widehat{\nu^s}(0)=\widehat{\psi}(0)$. Note that $d\mu$ and $\psi$ have compact support and their Fourier transforms are both smooth. It follows that $\widehat{\nu^s}(\xi)-\widehat{\psi}(\xi)=\widehat{d\mu}(\xi)(1+|\xi|^2)^{\frac s2}-\widehat{\psi}(\xi)$ is also smooth. As a result, by (\ref{eq:rho(x) and |x|_infty})
\begin{equation*}
	|\widehat{\nu^s}(\xi)-\widehat{\psi}(\xi)|\leq A|\xi|\leq A'\rho(\xi), \quad \text{ for } \rho(\xi)\leq 1.
\end{equation*}
This gives the first part of (\ref{eq:decay of FT of nu^s-psi}). For the $\rho(\xi)\geq1$ part, we need the following result, which is stated for a general $k$-dimensional submanifold in $\R^n$.
\begin{lemma}
	\label{lem:decay of Fourier transform of measure}
	Suppose $S$ is a smooth $k$-dimensional manifold in $\R^n$ of finite type. Let the measure $d\mu$ be defined in Section \ref{sec:parabola-model}. Then there exists a constant $A>0$ such that
	\begin{equation*}
		|\widehat{d\mu}(\xi)|\leq A (1+|\xi|)^{-1/m},
	\end{equation*}
	where $m$ is the type of $S$ inside the support of $\eta$.
\end{lemma}
Lemma \ref{lem:decay of Fourier transform of measure} can be found in \cite[Chapter VIII, Section 3.2]{Stein1993}.
In the present setting, $S$ is given by the parabola $\gamma(t)=(t,t^2)$ in $\R^2$, which has type $2$. Lemma \ref{lem:decay of Fourier transform of measure} implies the decay
\begin{equation*}
	|\widehat{\nu^s}(\xi)|=|\widehat{d\mu}(\xi)|(1+|\xi|^2)^{\frac s2}\leq A|\xi|^{-\frac12}(1+|\xi|^2)^{\frac s2} \leq A'\rho(\xi)^{-\frac12+s}, \quad \text{ if } \rho(\xi)\geq1.
\end{equation*}
Since $\widehat{\psi}$ is a Schwartz function and rapidly decreasing at $\infty$, we obtain the second part of (\ref{eq:decay of FT of nu^s-psi}).
Thus we have proven Proposition \ref{prop:A_j^s L2 bound} for real $s<1/2$.
\end{proof}
\medskip

\par We now treat complex parameters $s$ with $\Re(s)<1/2$.

\begin{proof}[Proof of Proposition \ref{prop:A_j^s L2 bound}]
Denote by $K^s$ the scalar distribution kernel $\nu^s-\psi$ that we discussed above. Let
\begin{equation*}
	K_j^s=\nu_j^s-\psi_j, \quad \text{ and then } \quad \widehat{K_j^s}(\xi)=\widehat{\nu^s}(2^{-j}\circ\xi)-\widehat{\psi}(2^{-j}\circ\xi).
\end{equation*}
We have already proved the case of real $s$.
For complex $s=\sigma+i\rho$, with $\sigma<1/2$, the estimates of the kernel $\widehat{K^s}$ used in the real case remain valid after replacing $s$ by $\sigma=\Re(s)$ on the right-hand side:
\begin{equation*}
	|\widehat{K^s}(\xi)|\leq A_s\rho(\xi), \quad \text{ and } \quad  |\widehat{K^s}(\xi)|\leq A_s\rho(\xi)^{-1/2+\sigma},
\end{equation*}
for $\rho(\xi)\leq1$ and $\rho(\xi)\geq1$, respectively. Therefore, as in the real case, we have
\begin{equation*}
	\sum_{j\in\Z}|\widehat{K^s}(2^{-j}\circ\xi)|^2\leq C(s), \quad \text{ uniformly in } \xi.
\end{equation*}
Thus the square function estimate remains valid for complex $s$:
\begin{equation*}
	\sum_{j\in\Z}\big\|A_j^sf-B_jf\big\|_{L^2(\A)}^2= \int_{\R^2}\sum_{j\in\Z}|\widehat{K^s}(2^{-j}\circ\xi)|^2\|\hat{f}(\xi)\|_{L^2(\M)}^2d\xi\leq C(s)\Norm{f}_{L^2(\A)}^2.
\end{equation*}
Split $K_j^s$ into its real and imaginary parts:
\begin{equation*}
	K_j^{s}=K_j^{s,(1)}+iK_j^{s,(2)},
\end{equation*}
where $K_j^{s,(1)}$ and $K_j^{s,(2)}$ are real scalar distributions. Define
\begin{equation*}
	T_j^{s,(1)}(f)=f*K_j^{s,(1)}, \quad T_j^{s,(2)}(f)=f*K_j^{s,(2)}.
\end{equation*}
Then $T_j^{s,(1)}$ and $T_j^{s,(2)}$ preserve self-adjointness. The corresponding Fourier multipliers are
\begin{equation*}
	\widehat{K_j^{s,(1)}}(\xi)=\frac12\bigg[\widehat{K_j^{s}}(\xi)+\overline{\widehat{K_j^{s}}(-\xi)}\bigg], \quad \widehat{K_j^{s,(2)}}(\xi)=\frac1{2i}\bigg[\widehat{K_j^{s}}(\xi)-\overline{\widehat{K_j^{s}}(-\xi)}\bigg].
\end{equation*}
Since
\begin{equation*}
	\Big|\widehat{K_j^{s,(1)}}(\xi)\Big|+\Big|\widehat{K_j^{s,(2)}}(\xi)\Big|\lesssim |\widehat{K_j^{s}}(\xi)|+|\widehat{K_j^s}(-\xi)|,
\end{equation*}
and since $\rho(-\xi)=\rho(\xi)$, the same square-sum estimate holds:
\begin{equation*}
	\sum_{j\in\Z}\Big|\widehat{K_j^{s,(1)}}(\xi)\Big|^2\lesssim_s 1, \quad \sum_{j\in\Z}\Big|\widehat{K_j^{s,(2)}}(\xi)\Big|^2\lesssim_s 1.
\end{equation*}
Therefore the vector-valued Plancherel identity gives
\begin{equation*}
	\bigg\| \Big(\sum_{j\in\Z}\big|T_j^{s,(1)}(f)\big|^2\Big)^{1/2} \bigg\|_{L^2(\A)}\lesssim_s \Norm{f}_{L^2(\A)},
\end{equation*}
and similarly
\begin{equation*}
	\bigg\| \Big(\sum_{j\in\Z}\big|T_j^{s,(2)}(f)\big|^2\Big)^{1/2} \bigg\|_{L^2(\A)}\lesssim_s \Norm{f}_{L^2(\A)}.
\end{equation*}
\par Now assume $f\in\A_{c,+}$. Then $T_j^{s,(1)}(f)$ is also self-adjoint. We have
\begin{equation*}
	-\big|T_j^{s,(1)}(f)\big|\leq T_j^{s,(1)}(f) \leq \big|T_j^{s,(1)}(f)\big|,
\end{equation*}
and moreover
\begin{equation*}
	\big|T_j^{s,(1)}(f)\big|\leq \Big(\sum_{k\in\Z}\big|T_k^{s,(1)}(f)\big|^2\Big)^{\frac12}.
\end{equation*}
This yields
\begin{equation*}
	\Norm{\big\{T_j^{s,(1)}(f)\big\}_j}_{L^2(\A;\ell^\infty)}\leq \bigg\| \Big(\sum_{j\in\Z}\big|T_j^{s,(1)}(f)\big|^2\Big)^{1/2} \bigg\|_{L^2(\A)}\lesssim_s \Norm{f}_{L^2(\A)}.
\end{equation*}
A similar argument gives
\begin{equation*}
	\Norm{\big\{T_j^{s,(2)}(f)\big\}_j}_{L^2(\A;\ell^\infty)} \lesssim_s \Norm{f}_{L^2(\A)}.
\end{equation*}
Recall that $A_j^s(f)-B_j(f)=f*K_j^s=T_j^{s,(1)}(f)+iT_j^{s,(2)}(f)$, hence
\begin{equation*}
	\begin{split}
		\Big\|\sup_{j\in\Z} \left(A_j^s(f)-B_j(f)\right)\Big\|_{L^2(\A)}
		&= \Norm{\{A_j^s(f)-B_j(f)\}}_{L^2(\A;\ell^\infty)} \\
		&\leq \Norm{\{T_j^{s,(1)}(f)\}}_{L^2(\A;\ell^\infty)}+ \Norm{\{T_j^{s,(2)}(f)\}}_{L^2(\A;\ell^\infty)} \\
		&\lesssim_s \Norm{f}_{L^2(\A)}.
	\end{split}
\end{equation*}
For general $f$, write $f$ as a linear combination of four positive elements. Applying the estimate to each summand gives the same bound, up to an inessential constant.
\end{proof}

\begin{remark}
	\label{rem:growth of C(s)}
	We have shown that for every $s\in\mathbb{C}$ with $\Re(s)<1/2$, there exists a constant $C(s)>0$ depending only on $s$ such that
	\begin{equation}
		\label{eq:C(s) inequality}
		\Big\|{\sup_{j\in\Z} A_j^s(f)}\Big\|_{L^2(\A)}\leq C(s) \Norm{f}_{L^2(\A)}.
	\end{equation}
	If $s=\sigma+it$ is a complex number, we consider the growth of $t\mapsto C(\sigma+it)$ as $|t|\to\infty$.
	In fact, the key is to estimate
	\begin{equation*}
		\sum_{j\in\Z}|\widehat{K^s}(2^{-j}\circ\xi)|^2.
	\end{equation*}
	We write (recall that $\widehat{d\mu}(0)=\widehat{\psi}(0)$)
	\begin{equation*}
		\begin{split}
			\widehat{K^s}(\xi)&=\widehat{\nu^s}(\xi)-\widehat{\psi}(\xi)=\widehat{d\mu}(\xi)(1+|\xi|^2)^{\frac s2}-\widehat{\psi}(\xi) \\
			&=\big(\widehat{d\mu}(\xi)-\widehat{d\mu}(0)\big)+\big(\widehat{\psi}(0)-\widehat{\psi}(\xi)\big)+\widehat{d\mu}(\xi)\big((1+|\xi|^2)^{\frac s2}-1\big).
		\end{split}
	\end{equation*}
	For small $\rho(\xi)$, the first two terms in the second line are $O(\rho(\xi))$, and the third term can be bounded by
	\begin{equation*}
		\big|{(1+|\xi|^2)^{\frac s2}-1}\big|\lesssim |s||\xi|^2\lesssim_\sigma (1+|t|)\rho(\xi)^2 \lesssim_\sigma (1+|t|)\rho(\xi).
	\end{equation*}
	Thus
	\begin{equation*}
		|\widehat{K^{\sigma+it}}(\xi)|\lesssim_\sigma (1+|t|)\rho(\xi), \quad \rho(\xi)\leq1.
	\end{equation*}
	For large $\rho(\xi)$, the imaginary part does not affect the estimate, because
	\begin{equation*}
		\norm{\widehat{\nu^s}(\xi)}=\big|{\widehat{d\mu}(\xi)(1+|\xi|^2)^{\frac s2}}\big|\lesssim |\xi|^{-\frac12+\sigma}.
	\end{equation*}
	Since $\rho(\xi)\lesssim|\xi|$ for $\rho(\xi)\geq1$, and $\sigma<1/2$, we have
    \begin{equation*}
		|\widehat{K^{\sigma+it}}(\xi)|\lesssim \rho(\xi)^{-\frac12+\sigma},\quad \rho(\xi)\geq1.
	\end{equation*}
	This estimate has no growth in $t$.
	Thus the constant in (\ref{eq:C(s) inequality}) satisfies the estimate
	\begin{equation*}
		C(\sigma+it)\lesssim_\sigma 1+|t|.
	\end{equation*}
\end{remark}

\medskip

\section{Weak $(1,1)$ boundedness of $\{A_j^s\}_{j\in\Z}$}
\label{sec:weak-11-boundedness}
\par In this section, we prove Proposition \ref{prop:A_j^s weak L1 bound}, which gives the weak type $(1,1)$ estimate for the maximal operator associated with the family $\{A_j^s\}_{j\in\Z}$. Let us first introduce the noncommutative Calder\'on-Zygmund decomposition.
\subsection{Noncommutative Calder\'on-Zygmund decomposition}
\label{subsec:cz-decomposition}
\par We first introduce notation for the \emph{non-isotropic} dyadic cubes on $\R^2$. For $k\in\Z$, let $\Q_k$ be the set of non-isotropic dyadic cubes of the form
\begin{equation*}
	\Q_k:=\left\{ [a2^{-k},(a+1)2^{-k})\times[b2^{-2k},(b+1)2^{-2k}) : a,b\in\Z \right\}.
\end{equation*}
Let $\Q:=\bigcup_{k\in\Z}\Q_k$. For any rectangle $I:=[x_1,x_2]\times[y_1,y_2]$ in $\R^2$, we call $I$ a non-isotropic cube with respect to the metric $\rho$ if $|x_2-x_1|=|y_2-y_1|^{1/2}$. In this case, we define
\begin{equation*}
	\ell(I):=\rho\left((x_1,y_1),(x_2,y_2)\right)=|x_2-x_1|=|y_2-y_1|^{1/2}.
\end{equation*}
to be the non-isotropic ``side length" of $I$. Then for any $Q\in \Q_k$, $Q$ is a (non-isotropic) cube with side length $\ell(Q)=2^{-k}$. Let $|Q|$ be the (Lebesgue) volume of $Q$. For any $Q\in \Q$, we denote by $c_Q$ its center, and let $sQ$ be the (non-isotropic) cube with the same center as $Q$ such that $\ell(sQ)=s\ell(Q)$ \footnote{Strictly speaking, since the dilation is non-isotropic, the notation should be $s\circ Q$ instead of $sQ$. We use the latter notation for simplicity.}.
For $y\in\R^2$ and $k\in\Z$, there is essentially a unique $Q(y)\in\Q_k$ such that $y\in Q(y)$; we denote by $\bar{y}^k$ the center of $Q(y)$. If $f:\R^2\to\M$ is integrable on $Q$, we define its average by $f_Q:=|Q|^{-1}\int_Qf(x)dx$.
\par For $k\in\Z$, let $\sigma_k$ be the $k$-th non-isotropic dyadic $\sigma$-algebra, i.e. the $\sigma$-algebra generated by the non-isotropic dyadic cubes in $\Q_k$. Let $E_k$ be the conditional expectation associated with the non-isotropic dyadic filtration $\sigma_k$ on $\R^2$. For a locally integrable function $f$, set $f_k=E_k(f)$ for all $k\in\Z$. Then each $f_k$ is constant on every $Q\in\Q_k$. Therefore, by the definition of $E_k$,
\begin{equation*}
	f_k(x)=E_kf(x)=\sum_{Q\in\Q_k}f_Q\chi_Q(x).
\end{equation*}

\par Let $\A_k:=\{E_k(f):f\in\A\}$. Then $\{\A_k\}_{k\in\Z}$ is an increasing sequence of von Neumann subalgebras of $\A$ such that $\bigcup_{k\in\Z}\A_k$ is weak*-dense in $\A$, hence $\{\A_k\}_{k\in\Z}$ is a filtration. For $f\in L^p(\A)$, the sequence $\{f_k\}_k$ is bounded in $L^p(\A)$, and is therefore an $L^p$-martingale with respect to this filtration.
For simplicity, we denote by $df_k$ the martingale difference $f_k-f_{k-1}$.

\par Now we fix $f\in\A_{c,+}$. Then the sequence $\{f_k\}_{k\in\Z}$ is a positive dyadic martingale in $L^1(\A)$. Applying Cuculescu's construction (see \cite{Cuculescu1971,Parcet2009}) to this non-isotropic dyadic martingale, we obtain the following result.

\begin{lemma}
	Fix $f\in \A_{c,+}$ and $\lambda>0$. There exists a decreasing sequence of projections $\{q_k\}_{k\in\Z}\subset\PA$ satisfying the following properties:
	\begin{enumerate}
		\item $q_k$ commutes with $q_{k-1}f_kq_{k-1}$ for $\forall\ k\in\Z$.
		\item $q_k$ belongs to $\A_k$ for all $k\in\Z$ and $q_kf_kq_k\leq\lambda q_k$.
		\item Set $q=\bigwedge_{k\in\Z}q_k$. We have the inequality
		\[ \nu(q^{\perp})=\nu(1_\A-q)\leq \lambda^{-1}\Norm{f}_{L^1(\A)}. \]
		\item The sequence of projections $q_k$ is given by
		\begin{equation*}
			q_k=
			\begin{cases}
				1_{\A},  & \text{if } k<m, \\
				\chi_{[0,\lambda]}(f_k), & \text{if } k=m, \\
				q_{k-1}\chi_{[0,\lambda]}(q_{k-1}f_kq_{k-1}), & \text{if } k>m,
			\end{cases}
		\end{equation*}
		for some integer $m$.
	\end{enumerate}
\end{lemma}
\par Since $q_k$ is a projection in $\A_k$, we can view $q_k$ as a $\PM$-valued function on $\R^2$. In fact, $q_k(x)$ is constant on each cube $Q\in\Q_k$, so we can write
\[ q_k(x)=\sum_{Q\in\Q_k}\xi_Q\chi_Q(x), \]
where $\xi_Q$ is a projection in $\M$ that satisfies the following conditions:
\begin{enumerate}
	\item The projections $\xi_Q$ satisfy
	\begin{equation*}
		\xi_Q=
		\begin{cases}
			1_\M, &\text{if } k<m, \\
			\chi_{[0,\lambda]}(f_Q), & \text{if } k=m, \\
			\xi_{\widehat{Q}}\chi_{[0,\lambda]}(\xi_{\widehat{Q}}f_Q\xi_{\widehat{Q}}), & \text{if } k>m,
		\end{cases}
	\end{equation*}
	where $\widehat{Q}$ is the dyadic parent of $Q$.
	\item $\xi_Q\in\PM$ and $\xi_Q\leq\xi_{\widehat{Q}}$.
	\item $\xi_Q$ commutes with $\xi_{\widehat{Q}}f_Q\xi_{\widehat{Q}}$.
	\item $\xi_Qf_Q\xi_Q\leq\lambda\xi_Q$.
\end{enumerate}
\par Now define the projections $p_k=q_{k-1}-q_k$, $k\in\Z$. Then all the $p_k$s are projections in $\A$. By the above explicit expressions, we see that
\[ p_k=\sum_{Q\in\Q_k}(\xi_{\widehat{Q}}-\xi_{Q})\chi_Q(x):=\sum_{Q\in\Q_k}\pi_Q\chi_Q(x) \]
where $\pi_Q=\xi_{\widehat{Q}}-\xi_Q$. The $p_k$ are pairwise disjoint and $\sum_{k\in\Z}p_k=1_\A-q$.
\par For convenience, we write $\widehat{\Z}=\Z\cup\{\infty\}$ and $p_\infty=q$. Then we have $\sum_{k\in\widehat{\Z}}p_k=1_\A$. Recall that $f_k=E_k(f)$ for $k\in\Z$. For $k=\infty$, define $f_{\infty}=f$.

\par We decompose $f$ into its good and bad functions as follows:
\begin{equation*}
	f=g+b, \quad \text{ where } g=\sum_{i,j\in\widehat{\Z}}p_if_{i\vee j}p_j, \quad b=\sum_{i,j\in\widehat{\Z}}p_i(f-f_{i\vee j})p_j
\end{equation*}
where $i\vee j=\max\{i,j\}$. If $i$ or $j$ is infinity, $i\vee j$ denotes $\infty$.
\par We further decompose $g$ into the diagonal terms and the off-diagonal terms:
\begin{equation*}
	g=g_d+g_\off,\quad \text{ where } \quad  g_d=\sum_{k\in\widehat{\Z}}p_kf_kp_k=qfq+\sum_{k\in\Z}p_kf_kp_k,
\end{equation*}
\begin{equation*}
	g_{\off}=\sum_{i\neq j\in\widehat{\Z}}p_if_{i\vee j}p_j=\sum_{i\neq j\in\Z}p_if_{i\vee j}p_j+qf(1_\A-q)+(1_\A-q)fq.
\end{equation*}
The bad function $b$ can also be decomposed into diagonal and off-diagonal terms. However, we treat these two parts together.

\par As in the previous section, we temporarily assume $s\in\R$. For $s<0$, Lemma \ref{lem:Bessel kernel properties} shows that $G_s=\F^{-1}[(1+|\cdot|^2)^{s/2}]$ is a positive $L^1$-function. Therefore $f\in\A_{c,+}$ implies $A_j^sf\geq0$.
\par In the rest of this section, we fix $s<0$. For simplicity, we write
\begin{equation*}
	T_j:=A_j^s-B_j, \quad \text{ with kernel } K_j(x):=\nu_j^s(x)-\psi_j(x).
\end{equation*}
(Recall that the operator $B_j$ and its kernel $\psi_j$ are defined in (\ref{eq:B_j def}) and (\ref{eq:psi_j def}), respectively.) Set
\[ K(x):=K_0(x), \quad \nu^s(x):=\nu_0^s(x) \text{ and } \psi(x)=\psi_0(x). \]
\par The family $B_j(f)$ is controlled by the operator-valued Hardy-Littlewood maximal function and therefore satisfies a weak type $(1,1)$ inequality. It suffices to prove Proposition \ref{prop:A_j^s weak L1 bound} with $A_j^s$ replaced by $T_j=A_j^s-B_j$, that is, we need to show that the operator $f\mapsto\{T_j(f)\}_{j\in\Z}$ is bounded from $L^1(\A)$ to $\Lambda^{1,\infty}(\A;\ell^{\infty})$.
\par In the following, we list some properties of the kernel $K(x)$ that will be used later.
\begin{lemma}
	\label{prop:K(x) kernel preoperties}
	Suppose $s<0$ is fixed. The function $K(x)$ satisfies
	\begin{enumerate}
		\item\label{item:kernel-local-integrability} $K(x)$ is a locally integrable function on $\R^2$ that is smooth for sufficiently large $x$, say $\rho(x)\geq 10$.
		\item\label{item:kernel-translation} There exists $\varepsilon>0$ such that for all $y$ with $\rho(y)\leq 1$,
		\begin{equation*}
			\int_{\R^2}\norm{K(x-y)-K(x)} dx\lesssim \rho(y)^{\varepsilon}.
		\end{equation*}
		\item\label{item:kernel-decay} For arbitrary $M>0$, we have
		\begin{equation*}
			\max\{|K(x)|,\rho(\nabla K(x))\}\lesssim_M \rho(x)^{-M}, \quad \text{ for } \rho(x) \text{ sufficiently large. }
		\end{equation*}
		\item\label{item:kernel-fourier-decay} Its Fourier transform $\widehat{K}(\xi)$ is an $L^\infty$ function satisfying
		\begin{equation*}
			|\widehat{K}(\xi)|\lesssim
			\begin{cases}
				\rho(\xi), &\text{ for } |\xi| \text{ near } 0, \\
				\rho(\xi)^{-\delta}, &\text{ for } |\xi| \text{ near } \infty,
			\end{cases}
		\end{equation*}
		where $\delta=1/2$.
	\end{enumerate}
\end{lemma}
\begin{proof}
	Recall that $K=G_s*d\mu-\psi$. By Appendix \ref{app:Bessel kernels}, the function $G_s\in L^1(\R^2)$ and is smooth away from $0$. Since $d\mu$ is finite and compactly supported, $G_s*d\mu$ is integrable and smooth off $\supp(d\mu)$. This proves (\ref{item:kernel-local-integrability}).
	For (\ref{item:kernel-translation}), choose $0<\varepsilon<\min\{1,-s\}$. The translation estimate (\ref{eq:Bessel translation estimate}) and the smoothness of $\psi$ give that, when $\rho(y)\leq1$,
	\[
	 \Norm{K(\cdot-y)-K}_{L^1(\R^2)}\lesssim_s |y|^\varepsilon+|y|\lesssim_s\rho(y)^\varepsilon.
	\]
	By (\ref{eq:Bessel kernel decay}) and the compact supports of $\mu$ and $\psi$, for every $L>0$,
	\[
	 |K(x)|+|\nabla K(x)|\lesssim_{s,L}(1+|x|)^{-L}, \qquad |x|\text{ sufficiently large}.
	\]
	Taking $L=2M$ and using (\ref{eq:rho(x) and |x|_infty}) proves (\ref{item:kernel-decay}). Finally, (\ref{item:kernel-fourier-decay}) follows from (\ref{eq:decay of FT of nu^s-psi}), since $s<0$.
	One should note that the implicit constants in (\ref{item:kernel-translation}), (\ref{item:kernel-decay}) and (\ref{item:kernel-fourier-decay}) depend only on the negative number $s$.
\end{proof}

\par In the rest of this section, we prove Proposition \ref{prop:A_j^s weak L1 bound}, the noncommutative maximal weak type $(1,1)$ estimate for the operator $f\mapsto \{A_j^s(f)\}_{j\in\Z}$.
First note that it suffices to consider the case $s$ is a real number such that $s<0$. For general complex $s$, we may split the kernels into real parts and imaginary parts and treat them separately, as in the previous section.

\par More explicitly, our aim is the following result:
\begin{prop}
	{\itshape
	Fix $s<0$. For each $\lambda>0$, $f\in\A_{c,+}$, there exists a projection $e\in\mathcal{P}(\A)$ such that
	\begin{enumerate}
		\item $\Norm{eT_j(f)e}_\infty\leq\lambda$ for all $j\in\Z$;
		\item $\nu(e^{\perp})\lesssim\lambda^{-1}\Norm{f}_{L^1(\A)}$.
	\end{enumerate}
	}
\end{prop}

To achieve this goal, we decompose $f=g+b$. If we can find projections $e_g$ and $e_b$ that satisfy
\begin{equation}
	\label{eq:good}
	\|e_gT_j(g)e_g\|_{\infty}\leq \lambda/2 \ \text{ for all } j, \quad \text{ and } \ \nu(e_g^{\perp})\lesssim\lambda^{-1}\Norm{f}_{L^1(\A)},
\end{equation}
and
\begin{equation}
	\label{eq:bad}
	\|e_bT_j(b)e_b\|_\infty\leq \lambda/2 \ \text{ for all } j, \quad \text{ and } \ \nu(e_b^{\perp})\lesssim\lambda^{-1}\Norm{f}_{L^1(\A)}.
\end{equation}
Then the projection $e=e_g\wedge e_b$ satisfies $ee_g=ee_b=e$, hence
\begin{equation*}
	\begin{split}
		\|eT_j(f)e\|_\infty \leq \Norm{eT_j(g)e}+\Norm{eT_j(b)e} \leq \Norm{e_gT_j(g)e_g}+\Norm{e_bT_j(b)e_b} \leq \lambda/2+ \lambda/2=\lambda,
	\end{split}
\end{equation*}
and
\begin{equation*}
	\nu(e^\perp) \leq \nu(e_g^\perp)+\nu(e_b^\perp)\lesssim\lambda^{-1}\Norm{f}_{L^1(\A)}.
\end{equation*}
In the following two subsections, we will prove (\ref{eq:good}) and (\ref{eq:bad}), respectively.
\par To proceed further, we need the following lemma:
\begin{lemma}
	\label{lem:zeta projection}
	There exists a projection $\zeta\in \PA$ that satisfies the following conditions:
	\begin{enumerate}
		\item $\nu(\zeta^{\perp})\lesssim\lambda^{-1}\Norm{f}_{L^1(\A)}$,
		\item\label{item:zeta-local-order} If $Q_0\in\Q$ and $x\in100Q_0$, then $\zeta(x)\leq 1_{\M}-\xi_{\widehat{Q_0}}+\xi_{Q_0}$ which further implies $\zeta(x)\leq\xi_{Q_0}$.
	\end{enumerate}
\end{lemma}
The proof is a direct adaptation of the argument in \cite{Parcet2009} (see also \cite{Lai2024}). The constant $100$ in Lemma \ref{lem:zeta projection} (\ref{item:zeta-local-order}) is immaterial and may be adjusted.

\subsection{Estimate for the bad function}
\label{subsec:bad-function}
\par
We will prove the following:
\begin{lemma}
	\label{lem:e_b' projection}
	There exists a projection $e_b'\in\PA$ such that
	\begin{equation*}
		\|e_b'\zeta T_j(b)\zeta e_b'\|_\infty \leq \lambda/2 \ \text{ for all } j, \quad  \text{ and }  \nu(1_\A-e_b')\lesssim\lambda^{-1}\Norm{f}_{L^1(\A)}.
	\end{equation*}
\end{lemma}
Lemma \ref{lem:e_b' projection} implies (\ref{eq:bad}). Indeed, assuming Lemma \ref{lem:e_b' projection}, set $e_b=e_b'\wedge\zeta$. Then $e_b$ satisfies
\begin{equation*}
	\|e_bT_j(b)e_b\|_\infty\leq \|e_b'\zeta T_j(b)\zeta e_b'\|_\infty \leq \lambda/2
\end{equation*}
and
\begin{equation*}
	\nu(1_\A-e_b)\leq \nu(1_\A-e_b')+\nu(1_\A-\zeta) \lesssim \lambda^{-1}\Norm{f}_{L^1(\A)}.
\end{equation*}
Therefore the projection $e_b$ satisfies (\ref{eq:bad}).
\par To prove Lemma \ref{lem:e_b' projection}, we decompose the bad function $b$ as follows
\begin{equation*}
	\begin{split}
		b\phantom{:}=\sum_{k\in\Z}p_k(f-f_k)p_k+\sum_{s\geq1}\sum_{k\in\Z}\big[p_k(f-f_{k+s})p_{k+s}+p_{k+s}(f-f_{k+s})p_k\big] =\sum_{s\geq0}\sum_{k\in\Z}b_{k,s},
	\end{split}
\end{equation*}
where
\begin{equation}
	\label{eq:b_{k,s} def}
	b_{k,0}=p_k(f-f_k)p_k, \quad b_{k,s}=p_k(f-f_{k+s})p_{k+s}+p_{k+s}(f-f_{k+s})p_k \ \ (s\geq 1).
\end{equation}
The functions $b_{k,s}$ satisfy the following properties, which will be used later.
\begin{lemma}
	\label{lem:b_{k,s} properties}
	Let $b_{k,s}$ be defined as in (\ref{eq:b_{k,s} def}). For each $s\geq 0$, we have the following properties:
	\begin{enumerate}
		\item The $L^1$-estimate $\sum_{k\in\Z}\Norm{b_{k,s}}_{L^1(\A)}\lesssim\Norm{f}_{L^1(\A)}$ holds.
		\item For all $k\in\Z$ and $Q\in\Q_{k+s}$, the cancellation property $\int_Q b_{k,s}(x)dx=0$ holds.
	\end{enumerate}
\end{lemma}
For a proof of this lemma, see \cite[Appendix B]{Parcet2009}.

\begin{lemma}
	\label{lem:zeta T_jb zeta L^1 norm estimate}
	The following estimate holds:
	\[ \sum_{j\in\Z}\Norm{\zeta T_j(b)\zeta}_{L^1(\A)}\lesssim\Norm{f}_{L^1(\A)}. \]
\end{lemma}
\begin{proof}[Proof of Lemma \ref{lem:e_b' projection} assuming Lemma \ref{lem:zeta T_jb zeta L^1 norm estimate}]
	Assume Lemma \ref{lem:zeta T_jb zeta L^1 norm estimate} and set
	\[ e_{b,j}':=1_{[0,\lambda/2]}(|\zeta T_j(b)\zeta|), \quad  e_b':=\bigwedge_{j\in\Z}e_{b,j}'. \]
	Since $\zeta T_j(b)\zeta$ is self-adjoint, we have
	\begin{equation*}
		\norm{e_{b,j}'\zeta T_j(b)\zeta e_{b,j}'}\leq e_{b,j}'|\zeta T_j(b)\zeta| e_{b,j}'\leq \lambda/2.
	\end{equation*}
	Moreover, Chebyshev's inequality implies
	\begin{equation*}
		\nu(1_\A-e_{b,j}')\leq \frac2{\lambda}\Norm{\zeta T_j(b)\zeta}_{L^1(\A)}
	\end{equation*}
	These properties imply
	\begin{equation*}
		\begin{split}
			\Norm{e_b'\zeta T_j(b)\zeta e_b'}_\infty \leq \Norm{e_{b,j}'\zeta T_j(b)\zeta e_{b,j}'}_\infty \leq  \lambda/2 \quad  \text{ for all } j\in\Z,
		\end{split}
	\end{equation*}
	and
	\begin{equation*}
		\nu(1_\A-e_b')\leq \sum_{j\in\Z}\nu(1_\A-e_{b,j}')\lesssim\lambda^{-1}\sum_{j\in\Z}\Norm{\zeta T_j(b)\zeta}_{L^1(\A)} \lesssim \lambda^{-1}\Norm{f}_{L^1(\A)},
	\end{equation*}
	where the last step is guaranteed by Lemma \ref{lem:zeta T_jb zeta L^1 norm estimate}.
\end{proof}
\subsubsection{Proof of Lemma \ref{lem:zeta T_jb zeta L^1 norm estimate}}
\label{subsubsec:bad-l1-estimate} In what follows, we prove Lemma \ref{lem:zeta T_jb zeta L^1 norm estimate}.
Since we have
\begin{equation*}
	T_j(b)=\sum_{s\geq0}\sum_{k\in\Z} T_j(b_{k,s}),
\end{equation*}
by the triangle inequality,
\[ \sum_{j\in\Z}\Norm{\zeta T_j(b)\zeta}_{L^1(\A)}\leq \sum_{s\geq0}\sum_{j\in\Z}\sum_{k\in\Z}\Norm{\zeta T_j(b_{k,s})\zeta}_{L^1(\A)} \]
and it suffices to estimate the right-hand side. Note that, by the cancellation property of $b_{k,s}$, we get
\begin{equation*}
	\begin{split}
		\zet(x)T_j(b_{k,s})(x)\zet(x)&=\int_{\R^2}\left(K_j(x-y)-K_j(x-\bar{y}^{k+s})\right)\zet(x)b_{k,s}(y)\zet(x) dy \\
		&=\sum_{Q\in\Q_{k+s}}\int_Q\left(K_j(x-y)-K_j(x-c_Q)\right)\zet(x)b_{k,s}(y)\zet(x) dy,
	\end{split}
\end{equation*}
where $\bar{y}^{k+s}$ is the center of dyadic cube $Q$ containing $y$ and $c_Q$ is the center of $Q$. Let $\widehat{Q}^s$ be the ancestor of $Q$ lying $s$ generations above it. We have the following result.
\begin{lemma}
	\label{claim:zeta b_ks zeta=0}
	Let $Q\in\Q_{k+s}$. For each $y\in Q$ and $x\in100\widehat{Q}^s$, we get
	\begin{equation*}
		\zet(x)b_{k,s}(y)\zet(x)=0.
	\end{equation*}
\end{lemma}
\begin{proof}
	For $s\geq1$, the definition of $b_{k,s}$ gives
	\[ b_{k,s}=p_kfp_{k+s}+p_{k+s}fp_k-p_kf_{k+s}p_{k+s}-p_{k+s}f_{k+s}p_k. \]
	We need only to consider the four terms separately. For example, consider the first term
	\[ \zet(x)(p_{k}fp_{k+s})(y)\zet(x)=\zet(x)\pi_{\widehat{Q}^s}f(y)\pi_Q\zet(x), \]
	where $Q\in\Q_{k+s}$ is the one that contains $y$. By Lemma \ref{lem:zeta projection}, if $x\in 100\widehat{Q}^s$, then $\zet(x)\leq\xi_{\widehat{Q}^s}$ and therefore
	\[ \zet(x)\pi_{\widehat{Q}^s}f(y)\pi_Q\zet(x)=\zet(x)\xi_{\widehat{Q}^s}\pi_{\widehat{Q}^s}f(y)\pi_Q\zet(x)=0, \]
	since $\xi_{\widehat{Q}^s}\pi_{\widehat{Q}^s}=\xi_{\widehat{Q}^s}(\xi_{\widehat{Q}^{s+1}}-\xi_{\widehat{Q}^s})=0$. The other three terms are handled similarly. For $s=0$, we have $b_{k,0}=p_k(f-f_k)p_k$, and $\zet(x)\pi_Q=0$ gives the same conclusion.
\end{proof}
By Lemma \ref{claim:zeta b_ks zeta=0}, we obtain
\begin{multline*}
	\zet(x)T_j(b_{k,s})(x)\zet(x)= \\
	\sum_{Q\in\Q_{k+s}}\chi_{(100\widehat{Q}^s)^c}(x) \int_Q\left(K_j(x-y)-K_j(x-c_Q)\right)\zet(x)b_{k,s}(y)\zet(x) dy.
\end{multline*}
Therefore for any $s\geq0,j\in\Z,k\in\Z$, we have
\begin{equation*}
	\begin{split}
		&\Norm{\zeta T_j(b_{k,s})\zeta}_{L^1(\A)} \\
		\leq&  \int_{\R^2}\bigg\|{\sum_{Q\in\Q_{k+s}}\chi_{(100\widehat{Q}^s)^c}(x)\int_Q\left(K_j(x-y)-K_j(x-c_Q)\right)\zet(x)b_{k,s}(y)\zet(x) dy}\bigg\|_{L^1(\M)}dx \\
		\leq&  \int_{\R^2}\sum_{Q\in\Q_{k+s}}\chi_{(100\widehat{Q}^s)^c}(x)\int_Q\norm{K_j(x-y)-K_j(x-c_Q)}\Norm{b_{k,s}(y)}_{L^1(\M)}dydx \\
		\leq&  \sum_{Q\in\Q_{k+s}}\int_{x:\rho(x,c_Q)\geq 10\cdot 2^{-k}}\int_Q\norm{K_j(x-y)-K_j(x-c_Q)}\Norm{b_{k,s}(y)}_{L^1(\M)}dydx,
	\end{split}
\end{equation*}
where the last step is obtained by interchanging the integral $\int_{\R^2}$ and the sum $\sum_{Q\in\Q_{k+s}}$, and using the fact that $x\in (100\widehat{Q}^s)^c$ implies $\rho(x,c_Q)\geq 10\cdot 2^{-k}$. Indeed, the definition of the dilated cube gives $\rho(x,c_{\widehat{Q}^s})\geq50\cdot2^{-k}$ for $x\in(100\widehat{Q}^s)^c$, while $\rho(w,c_{\widehat{Q}^s})\leq2^{-k}$ for $w\in\widehat{Q}^s$. Thus the triangle inequality (\ref{eq:triangle ineq for quasinorm}) gives
\[ \rho(x,w)\geq \rho(x,c_{\widehat{Q}^s})-\rho(w,c_{\widehat{Q}^s})\geq49\cdot2^{-k}>10\cdot2^{-k}. \]
Since the point $c_Q$ belongs to $\widehat{Q}^s$, we deduce that $\rho(x,c_Q)\geq 10\cdot 2^{-k}$.
Next we note that
\begin{equation*}
	\begin{split}
		&\sum_{Q\in\Q_{k+s}}\int_{x:\rho(x,c_Q)\geq 10\cdot 2^{-k}}\int_Q\norm{K_j(x-y)-K_j(x-c_Q)}\Norm{b_{k,s}(y)}_{L^1(\M)}dydx \\
		=& \sum_{Q\in\Q_{k+s}}\int_Q\int_{x:\rho(x,c_Q)\geq 10\cdot 2^{-k}}\norm{K_j(x-y)-K_j(x-c_Q)}dx\ \Norm{b_{k,s}(y)}_{L^1(\M)}dy \\
		=& \int_{\R^2}\int_{x:\rho(x,\bar{y}^{k+s})\geq 10\cdot 2^{-k}}\norm{K_j(x-y)-K_j(x-\bar{y}^{k+s})}dx\ \Norm{b_{k,s}(y)}_{L^1(\M)}dy.
	\end{split}
\end{equation*}
We therefore deduce that
\begin{multline}
	\label{eq:zeta b_(s,k) zeta estimate}
	\sum_{s\geq0}\sum_{j\in\Z}\sum_{k\in\Z}\Norm{\zeta T_j(b_{k,s})\zeta}_{L^1(\A)} \leq  \sum_{s\geq0}\sum_{k\in\Z}\int_{\R^2} \Norm{b_{k,s}(y)}_{L^1(\M)} \\
	\Big(\sum_{j\in\Z}\int_{x:\rho(x,\bar{y}^{k+s})\geq 10\cdot 2^{-k}}\big|{K_j(x-y)-K_j(x-\bar{y}^{k+s})}\big|dx\Big) \ dy.
\end{multline}
We now estimate the sum over $j$ in parentheses on the last line. A change of variables implies
\begin{equation*}
	\begin{split}
		&\sum_{j\in\Z}\int_{x:\rho(x,\bar{y}^{k+s})\geq 10\cdot 2^{-k}}\big|{K_j(x-y)-K_j(x-\bar{y}^{k+s})}\big|dx \\
		\leq& \sum_{j\in\Z}\int_{x:\rho(x,\bar{y}^{k+s})\geq 10\cdot 2^{s}\rho(y,\bar{y}^{k+s})}\norm{K_j\left((x-\bar{y}^{k+s})-(y-\bar{y}^{k+s})\right)-K_j(x-\bar{y}^{k+s})}dx \\
		=& \sum_{j\in\Z}\int_{\rho(x)\geq 10\cdot 2^s\rho(y)}\norm{K_j(x-y)-K_j(x)}dx \\
		=&\sum_{j\in\Z}\int_{\rho(x)\geq 10\cdot \rho(y)}\norm{K_j(2^s\circ x-y)-K_j(2^s\circ x)}d(2^s\circ x) \\
		=&\sum_{j\in\Z}\int_{\rho(x)\geq 10\rho(y)}|K_{j+s}(x-2^{-s}\circ y)-K_{j+s}(x)|dx \\
		=&\sum_{j\in\Z}\int_{\rho(x)\geq 10\rho(y)}|K_j(x-2^{-s}\circ y)-K_j(x)|dx.
	\end{split}
\end{equation*}
\begin{claim}
	\label{claim:sum int |K_j(x-y)-K_j(x)| estimate}
	There exists a constant $\alpha>0$ such that
	\begin{equation*}
		\sup_{y\in\R^2}\sum_{j\in\Z}\int_{\rho(x)\geq 10\rho(y)}|K_j(x-2^{-s}\circ y)-K_j(x)|dx\lesssim 2^{-\alpha s}
	\end{equation*}
	holds for all $s\geq0$.
\end{claim}

\par Once Claim \ref{claim:sum int |K_j(x-y)-K_j(x)| estimate} has been proved, using (\ref{eq:zeta b_(s,k) zeta estimate}), together with the $L^1$-estimate of $b_{k,s}$ (see Lemma \ref{lem:b_{k,s} properties}), we obtain
\begin{equation*}
	\sum_{s\geq0}\sum_{j\in\Z}\sum_{k\in\Z}\Norm{\zeta T_j(b_{k,s})\zeta}_{L^1(\A)}\lesssim \sum_{s\geq0}\sum_{k\in\Z}2^{-\alpha s}\Norm{b_{k,s}}_{L^1(\A)}\lesssim\Norm{f}_{L^1(\A)}.
\end{equation*}
This completes the proof of Lemma \ref{lem:zeta T_jb zeta L^1 norm estimate}, and hence that of Lemma \ref{lem:e_b' projection}.

\subsubsection{Proof of Claim \ref{claim:sum int |K_j(x-y)-K_j(x)| estimate}}
\label{subsubsec:kernel-difference-claim}
	We decompose the sum over $j\in\Z$ into two parts
	\begin{equation*}
		\begin{split}
			\sum_{j\in\Z}\int_{\rho(x)\geq 10\rho(y)}|K_j(x-2^{-s}\circ y)-K_j(x)|dx&=\sum_{j:2^j\rho(y)<1}\cdots+\sum_{j:2^j\rho(y)\geq1}\cdots \\
			:&=(I)+(II).
		\end{split}
	\end{equation*}
	For the first part, we have
	\begin{equation*}
		\begin{split}
			(I)&\leq \sum_{j:2^j\rho(y)<1} \int_{\R^2}|K_j(x-2^{-s}\circ y)-K_j(x)|dx \\
			&=\sum_{j:2^j\rho(y)<1} \int_{\R^2}2^{3j}|K(2^j\circ x-2^{j-s}\circ y)-K(2^j\circ x)|dx \\
			&=\sum_{j:2^j\rho(y)<1} \int_{\R^2}|K(x-2^{j-s}\circ y)-K(x)|dx \\
			&\lesssim\sum_{j:2^j\rho(y)<1} \left(2^{j-s}\rho(y)\right)^\varepsilon=2^{-\varepsilon s}\sum_{2^j<1/\rho(y)}\left(2^j\rho(y)\right)^{\varepsilon} \\
			&\lesssim 2^{-\varepsilon s},
		\end{split}
	\end{equation*}
	where the constant $\varepsilon>0$ above comes from Lemma \ref{prop:K(x) kernel preoperties}. For the second part, we have
	\begin{equation}
		\label{eq:(II) estimate}
		\begin{split}
			(II)&=\sum_{j:2^j\rho(y)\geq1}\int_{\rho(x)\geq 10\rho(y)}2^{3j}|K(2^j\circ x-2^{j-s}\circ y)-K(2^j\circ x)|dx \\
			&=\sum_{j:2^j\rho(y)\geq1}\int_{2^{-j}\rho(x)\geq 10\rho(y)}|K(x-2^{j-s}\circ y)-K(x)|dx \\
			&\lesssim 2^{-s}\sum_{j:2^j\rho(y)\geq1}(2^j\rho(y))^2\int_{\rho(x)\geq 10\cdot2^{j}\rho(y)} \sup_{\rho(x')\geq \frac9{10}\rho(x)} \rho(\nabla K(x'))dx,
		\end{split}
	\end{equation}
	To justify the last line of (\ref{eq:(II) estimate}), put $R_j=2^j\rho(y)\geq1$ and $h_j=2^{j-s}\circ y$. For $\rho(x)\geq10R_j$ and $0\leq\theta\leq1$, the triangle inequality (\ref{eq:triangle ineq for quasinorm}) gives
	\[
	 \rho(x-\theta h_j)\geq\rho(x)-\rho(h_j)\geq\frac9{10}\rho(x)\geq9R_j\geq9.
	\]
Since $\supp(d\mu)\subset\{\rho\leq2\}$ and $\psi$ is smooth, $\nabla K$ is bounded on $\{\rho\geq9\}$, and hence $|\nabla K|\lesssim\rho(\nabla K)$ there. Also, by (\ref{eq:rho(x) and |x|_infty}),
	\[
	 |h_j|\leq\sqrt{2}\max\{2^{-s}R_j,2^{-2s}R_j^2\}\leq\sqrt{2}\,2^{-s}R_j^2.
	\]
	This bound is valid even when $\rho(h_j)<1$. The mean value theorem now gives
	\[
	 |K(x-h_j)-K(x)|\lesssim2^{-s}R_j^2\sup_{\rho(x')\geq\frac9{10}\rho(x)}\rho(\nabla K(x')),
	\]
	which proves the last line of (\ref{eq:(II) estimate}).
	\par By Lemma \ref{prop:K(x) kernel preoperties} (\ref{item:kernel-decay}) and smoothness on $\{\rho\geq9\}$, the supremum above is bounded by a constant multiple of $(10/9)^M\rho(x)^{-M}$. Thus we obtain
	\begin{equation}
		\label{eq:(II) estimate 2}
		\begin{split}
			(II)&\lesssim 2^{-s}\sum_{j:2^j\rho(y)\geq1} (2^j\rho(y))^2\int_{\rho(x)\geq 10\cdot2^{j}\rho(y)}\frac1{\rho(x)^M}dx \\
			&\lesssim 2^{-s}\sum_{j:2^j\rho(y)\geq1}(2^j\rho(y))^2\left(10\cdot2^{j}\rho(y)\right)^{-(M-3)}\lesssim 2^{-s},
		\end{split}
	\end{equation}
	where we choose $M>5$ and in the second line of (\ref{eq:(II) estimate 2}), we use $|B(r)|=4r^3$ to obtain
	\begin{equation*}
		\int_{\rho(x)\geq R}\frac1{\rho(x)^M}dx=12\int_R^\infty r^{2-M}dr=\frac{12}{M-3}R^{3-M}, \qquad R>0.
	\end{equation*}
	These bounds are uniform in $y\in\R^2$ (the case $y=0$ is immediate), which proves the claim with $\alpha=\min\{\varepsilon,1\}>0$.

\subsection{Estimate for the good function}
\label{subsec:good-function}
\par In this subsection, we prove (\ref{eq:good}), the estimate for the good function. We handle the diagonal and off-diagonal parts separately.
\par The diagonal part is treated as in the classical case. We use the properties of $g_d$, which can be found in Parcet \cite{Parcet2009}:
\begin{equation}
	\label{eq:g_d properties}
	\Norm{g_d}_{L^1(\A)}\lesssim\Norm{f}_{L^1(\A)}, \quad \Norm{g_d}_{L^{\infty}(\A)}\lesssim \lambda.
\end{equation}
\begin{lemma}
	\label{lem:e_d estimate}
	There exists a projection $e_d\in\PA$ such that
	\begin{equation*}
		\|e_dT_j(g_d)e_d\|_\infty \leq  \lambda/4 \ \text{ for all } j , \quad \nu(1-e_d)\lesssim\lambda^{-1}\Norm{f}_{L^1(\A)}.
	\end{equation*}
\end{lemma}
\begin{proof}
	Let $e_d=1_{[0, \lambda/4]}(|S(g_d)|)$, where $S(g_d)=\big(\sum_{j\in\Z}|T_j(g_d)(x)|^2\big)^{1/2}$ is the square function defined in (\ref{eq:S(f) def}). Since $T_j(g_d)$ is self-adjoint, $e_dT_j(g_d)e_d\leq e_d|T_j(g_d)|e_d\leq e_dS(g_d)e_d$, and therefore
	\begin{equation*}
		\|e_dT_j(g_d)e_d\|_\infty\leq \|e_dS(g_d)e_d\|_\infty\leq \lambda/4.
	\end{equation*}
	Moreover, by Chebyshev's inequality and the $L^2$-boundedness of the square function in Section \ref{sec:l2-boundedness}, we have
	\begin{equation*}
		\nu(1-e_d)\leq\frac{16}{\lambda^2}\Norm{S(g_d)}_{L^2(\A)}^2\lesssim\frac1{\lambda^2}\Norm{g_d}_{L^2(\A)}^2\leq \frac1{\lambda^2}\Norm{g_d}_{L^1(\A)}\Norm{g_d}_{L^\infty(\A)}\lesssim\lambda^{-1}\Norm{f}_{L^1(\A)},
	\end{equation*}
	where in the last inequality we have used the estimates (\ref{eq:g_d properties}) for $g_d$.
\end{proof}
\par We next consider the off-diagonal part $g_\off$.
\begin{lemma}
	\label{lem:e_off estimate}
	There exists a projection $e_{\off}\in\PA$ such that
	\begin{equation*}
		\|e_\off \zeta T_j(g_\off)\zeta e_\off\|_\infty \leq  \lambda/4 \ (\forall j), \quad \nu(1-e_\off)\lesssim\lambda^{-1}\Norm{f}_{L^1(\A)}.
	\end{equation*}
	where the projection $\zeta$ is defined in Lemma \ref{lem:zeta projection}.
\end{lemma}
Lemma \ref{lem:e_d estimate} and Lemma \ref{lem:e_off estimate} imply (\ref{eq:good}), once we take $e_g=e_d\wedge e_\off\wedge\zeta$: the two compressed terms have norms at most $\lambda/4$, so their sum has norm at most $\lambda/2$. In the rest of this subsection, we prove Lemma \ref{lem:e_off estimate}.
\par We first introduce another representation of the off-diagonal part $g_\off$ and the associated estimates, which were proved in \cite{Parcet2009} (see also \cite{Lai2024}).
\begin{lemma}
	\label{lem:g_off properties}
	Recall that $df_k:=f_k-f_{k-1}$ is the martingale difference. Then $g_\off$ can be written as
	\begin{equation*}
		\begin{split}
			g_\off\phantom{:}&=\sum_{s\geq1}\sum_{k\in\Z}\big[p_kdf_{k+s}q_{k+s-1}+q_{k+s-1}df_{k+s}p_k\big] \\
			:&=\sum_{s\geq1}\sum_{k\in\Z}g_{k,s}:=\sum_{s\geq1}g_{(s)}.
		\end{split}
	\end{equation*}
	The martingale difference sequence of $g_{(s)}$ satisfies $d\left(g_{(s)}\right)_{k+s}=g_{k,s}$ and $\supp^*g_{k,s}\leq p_k\leq 1_\A-q_k$.
Here $\supp^*$ is the weak support, which is defined by $\supp^* a=1_\A-q^*$, with $q^*$ the greatest projection satisfying $q^*aq^*=0$.
Moreover, we have the estimates
	\begin{equation*}
		\sup_{s\geq1} \Norm{g_{(s)}}_{L^2(\A)}^2=\sup_{s\geq1}\sum_{k\in\Z}\Norm{g_{k,s}}_{L^2(\A)}^2\lesssim\lambda\Norm{f}_{L^1(\A)}.
	\end{equation*}
\end{lemma}
\par Fix $0<\varepsilon<1$ and $0<\gamma<1$.
\par Let $\phi(x)$ be a smooth bump function on $\R^2$ supported in $\{\rho(x)<1\}$ that satisfies $\phi(x)\equiv1$ for $\rho(x)<1/2$. Then let $\phi_{s}(x)=\phi(2^s\circ x)$.
For each fixed $s\geq1$, decompose the kernel $K_j$ into two parts
\begin{equation*}
	K_j(x)=K_j(x)\phi_{j-\varepsilon s}(x)+K_j(x)\left(1-\phi_{j-\varepsilon s}(x)\right).
\end{equation*}
We prove the following two lemmas,  which imply Lemma \ref{lem:e_off estimate} immediately.
\begin{lemma}
	\label{lem:K(phi) estimate}
	There exists a constant $\alpha_1>0$ such that
	\begin{equation*}
		\sum_{j\in\Z}\Norm{\zeta(K_j\phi_{j-\varepsilon s})*g_{(s)}\zeta}_{L^2(\A)}^2\lesssim 2^{-\alpha_1s}\lambda\Norm{f}_{L^1(\A)}.
	\end{equation*}
\end{lemma}
\begin{lemma}
	\label{lem:K(1-phi) estimate}
	There exists a constant $\alpha_2>0$ such that
	\begin{equation}
		\label{eq:K(1-phi) estimate}
		\Big\|{\sup_{j\in\Z}\left[K_j(1-\phi_{j-\varepsilon s})\right]*g_{(s)}}\Big\|_{L^2(\A)}^2\lesssim 2^{-\alpha_2s}\lambda\Norm{f}_{L^1(\A)}.
	\end{equation}
\end{lemma}

\medskip

\begin{proof}[Proof of Lemma \ref{lem:e_off estimate} assuming Lemma \ref{lem:K(phi) estimate} and \ref{lem:K(1-phi) estimate}]
	Fix $c_0=3/(4\pi^2)$. For each $s\geq1$ and $j\in\Z$, let
	\begin{equation*}
		e_{s,j}:=1_{[0,\frac{c_0\lambda}{s^2}]}\left(\norm{\zeta(K_j\phi_{j-\varepsilon s})*g_{(s)}\zeta}\right)
	\end{equation*}
	By Lemma \ref{lem:K(1-phi) estimate}, for each $s\geq1$ we can find a positive element $a_s\in\A$ such that
	\begin{equation}\label{e:30gp1}
		-a_s\leq \left[K_j(1-\phi_{j-\varepsilon s})\right]*g_{(s)} \leq a_s\ (\forall j), \quad \Norm{a_s}_{L^2(\A)}^2\lesssim 2^{-\alpha_2s}\lambda\Norm{f}_{L^1(\A)}.
	\end{equation}
	For each $s\geq1$, we set
	\begin{equation}\label{e:30gp2}
		e_s'=1_{[0,\frac{c_0\lambda}{s^2}]}(\zeta a_s\zeta) \quad \text{ and } \quad e_{\off}:=\bigwedge_{s\geq1}\bigwedge_{j\in\Z}e_{s,j}\wedge e_s'.
	\end{equation}
The projection $e_\off$ satisfies the conclusion of Lemma \ref{lem:e_off estimate}.
Indeed, by Chebyshev's inequality and Lemma \ref{lem:K(phi) estimate}, we have
	\begin{equation*}
		\begin{split}
			\nu(1-e_\off)&\leq \sum_{s\geq1}\bigg(\sum_{j\in\Z}\nu(1-e_{s,j})+\nu(1-e_s')\bigg) \\
			&\leq\sum_{s\geq1}\frac{s^4}{c_0^2\lambda^2}\bigg[\sum_{j\in\Z}\Norm{\zeta(K_j\phi_{j-\varepsilon s})*g_{(s)}\zeta}_{L^2(\A)}^2+\Norm{a_s}_{L^2(\A)}^2\bigg] \\
			&\lesssim \sum_{s\geq1}s^42^{-\alpha's}\lambda^{-1}\Norm{f}_{L^1(\A)} \lesssim\lambda^{-1}\Norm{f}_{L^1(\A)},
		\end{split}
	\end{equation*}
	where $\alpha'=\min\{\alpha_1,\alpha_2\}>0$. On the other hand,
	\begin{equation*}
		T_j(g_\off)=\sum_{s\geq1}(K_j\phi_{j-\varepsilon s})*g_{(s)}+\sum_{s\geq1}[K_j(1-\phi_{j-\varepsilon s})]*g_{(s)},
	\end{equation*}
using this identity, we get
	\begin{equation*}
		\begin{split}
			& \ \Norm{e_\off\zet T_j(g_\off)\zet e_\off}_\infty \\
			\leq & \sum_{s\geq1}\Big(\|e_\off\zeta (K_j\phi_{j-\varepsilon s})*g_{(s)}\zeta e_{\off}\|_\infty + \|e_\off\zeta [K_j(1-\phi_{j-\varepsilon s})]*g_{(s)}\zeta e_\off\|_\infty\Big) \\
			\leq & \sum_{s\geq1}\Big(\|e_{s,j}\zeta (K_j\phi_{j-\varepsilon s})*g_{(s)}\zeta e_{s,j}\|_\infty + \|e_{s}'\zeta [K_j(1-\phi_{j-\varepsilon s})]*g_{(s)}\zeta e_{s}'\|_\infty\Big) \\
			\leq & \sum_{s\geq1}\Big(\|e_{s,j}|\zeta (K_j\phi_{j-\varepsilon s})*g_{(s)}\zeta| e_{s,j}\|_\infty + \|e_{s}'\zeta a_s\zeta e_{s}'\|_\infty\Big) \\
			\leq & 2c_0\lambda\sum_{s\geq1}s^{-2}=\frac{c_0\pi^2}{3}\lambda=\frac{\lambda}{4},
		\end{split}
	\end{equation*}
where the fourth inequality follows from (\ref{e:30gp1}) and (\ref{e:30gp2}).
This finishes the proof of Lemma \ref{lem:e_off estimate}.
\end{proof}
\par It remains to prove Lemma  \ref{lem:K(phi) estimate} and Lemma \ref{lem:K(1-phi) estimate}.

\bigskip

\subsubsection{Proof of Lemma \ref{lem:K(1-phi) estimate}}
\label{subsubsec:kernel-away-estimate}
\par In this subsection we first prove Lemma \ref{lem:K(1-phi) estimate}.
	\par It is harmless to assume that $s$ is larger than a sufficiently large fixed constant, so that $2^{\varepsilon s-1}$ lies beyond the support of $\psi$ and the cutoff localizes the kernel to the region where the rapid decay estimate of Lemma \ref{prop:K(x) kernel preoperties} applies; the finitely many remaining values of $s$ can be treated separately and absorbed into the constant.
	In fact, we have
	\begin{equation*}
		\begin{split}
			|K_j(x)(1-\phi_{j-\varepsilon s}(x))| &= |\nu_j^s(x)(1-\phi_{j-\varepsilon s}(x))| \lesssim 2^{3j}\nu^s(2^j\circ x)\chi_{\{\rho(x)\geq 2^{-j+\varepsilon s-1}\}}(x) \\
			&\lesssim 2^{-j(M-3)}\rho(x)^{-M}\chi_{\{\rho(x)\geq 2^{-j+\varepsilon s-1}\}}(x),
		\end{split}
	\end{equation*}
	where $M>0$ can be chosen sufficiently large
and the second expression follows from the region $\phi(y)\equiv1$ when $\rho(y)\leq1/2$. Therefore, for all self-adjoint elements $h$ in $\A$, we have
	\begin{equation*}
		\begin{split}
			{[K_j(1-\phi_{j-\varepsilon s})]*h(x)} &\leq \big[|K_j(1-\phi_{j-\varepsilon s})|*|h|\big](x) \\
			&\lesssim 2^{-j(M-3)}\int_{\rho(x-y)\geq 2^{-j+\varepsilon s-1}}\frac1{\rho(x-y)^M}|h(y)|dy \\
			&\leq 2^{-j(M-3)}\sum_{k\geq \lfloor\varepsilon s\rfloor}{\int_{2^{-j+k-1}\leq\rho(x-y)<2^{-j+k}}\frac1{\rho(x-y)^M}|h(y)|dy} \\
	\end{split}
	\end{equation*}
which is further majorized by
\begin{equation*}
		\begin{split}&\quad2^{-j(M-3)}\sum_{k\geq \lfloor\varepsilon s\rfloor}{\int_{2^{-j+k-1}\leq\rho(x-y)<2^{-j+k}}\frac{2^M}{(2^{-j+k})^M}|h(y)|dy} \\
			&= 2^{-j(M-3)}\sum_{k\geq \lfloor\varepsilon s\rfloor}\frac{2^M}{(2^{-j+k})^M}{\int_{2^{-j+k-1}\leq\rho(x-y)<2^{-j+k}}|h(y)|dy} \\
			&\lesssim 2^{-j(M-3)}\sum_{k\geq \lfloor\varepsilon s\rfloor}\frac1{(2^{-j+k})^{M-3}}{\tilde{A}_{r(j,k)}(|h|)}(x),
		\end{split}
	\end{equation*}
	where $r(j,k)=2^{-j+k}$ and $\tilde{A}_r(h)$ is the non-isotropic averaging operator defined in (\ref{eq:average nonisotropic}).
	A similar calculation gives, for some fixed $C_M>0$,
	\begin{equation*}
		\begin{split}
			-C_M2^{-j(M-3)}\sum_{k\geq \lfloor\varepsilon s\rfloor}\frac1{(2^{-j+k})^{M-3}}{\tilde{A}_{r(j,k)}(|h|)}(x)
			&\leq -\big[|K_j(1-\phi_{j-\varepsilon s})|*|h|\big](x) \\
			&\leq [K_j(1-\phi_{j-\varepsilon s})]*h(x).
		\end{split}
	\end{equation*}
	Note that each function $g_{(s)}\in L^2(\A)$ is self-adjoint.
	By the strong type $(2,2)$ estimate in Theorem \ref{thm: H-L maximal nonisotropic}, there exists a positive element $a\in L^2(\A)$ such that
	\begin{equation}
		\label{eq:A_r(g_(s)) bounded}
		0\leq \tilde{A}_r(|g_{(s)}|) \leq a \ \ \text{ for } \forall r>0, \quad \Norm{a}_{L^2(\A)}\lesssim \big\||{g_{(s)}}|\big\|_{L^2(\A)} = \big\|{g_{(s)}}\big\|_{L^2(\A)}.
	\end{equation}
	Since $\sum_{k\geq\lfloor\varepsilon s\rfloor}2^{-k(M-3)}\lesssim_M2^{-\varepsilon s(M-3)}$ for $M>3$, let
	\begin{equation*}
		a'= C'_M2^{-\varepsilon s(M-3)}a,
	\end{equation*}
	with $C'_M$ sufficiently large. Then for all $j\in \Z$, we get
	\begin{equation*}
		\begin{split}
			{[K_j(1-\phi_{j-\varepsilon s})]*g_{(s)}} &\leq C_M2^{-j(M-3)}\sum_{k\geq \lfloor\varepsilon s\rfloor}\frac1{(2^{-j+k})^{M-3}}{\tilde{A}_{r(j,k)}(|g_{(s)}|)} \\
			&\leq C_M2^{-j(M-3)}\sum_{k\geq \lfloor\varepsilon s\rfloor}\frac1{(2^{-j+k})^{M-3}}a \leq a'.
		\end{split}
	\end{equation*}
	A similar argument gives
	\begin{equation*}
		-a'\leq -C_M2^{-j(M-3)}\sum_{k\geq \lfloor\varepsilon s\rfloor}\frac1{(2^{-j+k})^{M-3}}a \leq{[K_j(1-\phi_{j-\varepsilon s})]*g_{(s)}}.
	\end{equation*}
	Moreover, by (\ref{eq:A_r(g_(s)) bounded}) we have
	\begin{equation*}
		\Norm{a'}_{L^2(\A)}^2\lesssim 2^{-2\varepsilon s(M-3)}\Norm{g_{(s)}}_{L^2(\A)}^2\lesssim 2^{-\alpha_2 s}\lambda\Norm{f}_{L^1(\A)}
	\end{equation*}
	with $\alpha_2=2\varepsilon (M-3)>0$. This is exactly (\ref{eq:K(1-phi) estimate}).
$\hfill{} \Box$

\bigskip

\subsubsection{Proof of Lemma \ref{lem:K(phi) estimate}}
\label{subsubsec:kernel-near-estimate}
\par In this subsection, we prove Lemma \ref{lem:K(phi) estimate}.
	\par First we note that, by the support of $\zeta$ and $\phi_{j-\varepsilon s}$, we have
	\begin{equation}
		\label{eq:n>-epsilon s}
		\zeta(K_j\phi_{j-\varepsilon s})*g_{(s)}\zeta=\zeta(K_j\phi_{j-\varepsilon s})*\bigg(\sum_{n+s\geq s(1-\varepsilon)}g_{n+j,s}\bigg)\zeta,
	\end{equation}
	which is a consequence of the following result.
	\begin{claim}
		\label{claim:zeta (K phi)*g zeta=0}
		For $n<-\varepsilon s$, we have
		\begin{equation*}
			\zeta(K_j\phi_{j-\varepsilon s})*g_{n+j,s}\zeta=0.
		\end{equation*}
	\end{claim}
\begin{proof}
	Fix $s\geq1$ and $n<-\varepsilon s$. Lemma \ref{lem:g_off properties} says $\supp^*g_{n+j,s}\leq p_{n+j}\leq 1-q_{n+j}$, which implies
	\begin{equation*}
		q_{n+j}g_{n+j,s}q_{n+j}=0.
	\end{equation*}
	Since $n+j<j-\varepsilon s \leq j-\lfloor\varepsilon s\rfloor$ and the projections $\{q_k\}$ are decreasing, we have
	\begin{equation*}
		q_{j-\lfloor\varepsilon s\rfloor}g_{n+j,s}q_{j-\lfloor\varepsilon s\rfloor}=0.
	\end{equation*}
	Note that
	\begin{equation*}
		\begin{split}
			&\phantom{==}\zeta(x)(K_j\phi_{j-\varepsilon s})*g_{n+j,s}(x)\zeta(x) \\
			&=\int_{\R^2}K_j(x-y)\phi_{j-\varepsilon s}(x-y)\zeta(x)g_{n+j,s}(y)\zeta(x)dy \\
			&=\sum_{Q\in\Q_{j-\lfloor\varepsilon s\rfloor}} \int_Q K_j(x-y)\phi_{j-\varepsilon s}(x-y)\zeta(x)g_{n+j,s}(y)\zeta(x)dy.
		\end{split}
	\end{equation*}
	Fix an arbitrary $Q_0\in\Q_{j-\lfloor\varepsilon s\rfloor}$. We show that
	\begin{equation}
		\label{eq:int zeta g_(n+j,s) zeta=0}
		\int_{Q_0} K_j(x-y)\phi_{j-\varepsilon s}(x-y)\zeta(x)g_{n+j,s}(y)\zeta(x)dy=0,\quad \forall \ x\in\R^2.
	\end{equation}
	For $x\in 100 Q_0$ and $y\in Q_0$, by Lemma \ref{lem:zeta projection}, we have $\zeta(x)\leq\xi_{Q_0}=q_{j-\lfloor\varepsilon s\rfloor}(y)$. Therefore
	\begin{equation*}
		\zeta(x)g_{n+j,s}(y)\zeta(x)= \zeta(x)q_{j-\lfloor\varepsilon s\rfloor}(y)g_{n+j,s}(y)q_{j-\lfloor\varepsilon s\rfloor}(y)\zeta(x)=0
	\end{equation*}
    and (\ref{eq:int zeta g_(n+j,s) zeta=0}) holds for $x\in 100Q_0$.
	For $x\notin 100 Q_0$, the same cube separation estimate used above gives $\rho(x-y)\geq49\cdot2^{-j+\lfloor\varepsilon s\rfloor}>\frac{49}{2}\cdot2^{-j+\varepsilon s}>10\cdot2^{-j+\varepsilon s}$, and therefore $\phi_{j-\varepsilon s}(x-y)=0$. Thus (\ref{eq:int zeta g_(n+j,s) zeta=0}) holds also for $x\notin 100Q_0$. Therefore we have proved (\ref{eq:int zeta g_(n+j,s) zeta=0}), hence Claim \ref{claim:zeta (K phi)*g zeta=0}.
\end{proof}

	\medskip

	\par To prove Lemma \ref{lem:K(phi) estimate}, it suffices to prove that, whenever $n+s\geq s(1-\varepsilon)$, one has
	\begin{equation}
		\label{eq:sum_j g_{n+j,s} estimate}
		\sum_{j\in\Z}\Norm{(K_j\phi_{j-\varepsilon s})*g_{n+j,s}}_{L^2(\A)}^2\lesssim 2^{-\alpha_3(n+s)}\lambda\Norm{f}_{L^1(\A)}
	\end{equation}
	for some absolute constant $\alpha_3>0$. Indeed, by (\ref{eq:n>-epsilon s}) and (\ref{eq:sum_j g_{n+j,s} estimate}) we have
	\begin{equation*}
		\begin{split}
			(\text{LHS of Lemma \ref{lem:K(phi) estimate}})^{1/2}
			&\leq \Bigg[\sum_{j\in\Z}\bigg(\sum_{n+s\geq s(1-\varepsilon)}\Norm{(K_j\phi_{j-\varepsilon s})*g_{n+j,s}}_{L^2(\A)}\bigg)^2\Bigg]^{\frac12} \\
			&\leq \sum_{n+s\geq s(1-\varepsilon)}\Bigg[\sum_{j\in\Z}\Norm{(K_j\phi_{j-\varepsilon s})*g_{n+j,s}}_{L^2(\A)}^2\Bigg]^{\frac12} \\
			&\lesssim \sum_{n+s\geq s(1-\varepsilon)}\left(2^{-\alpha_3(n+s)}\lambda\Norm{f}_{L^1(\A)}\right)^{\frac12}.
		\end{split}
	\end{equation*}
	This implies Lemma \ref{lem:K(phi) estimate} with $\alpha_1=\alpha_3(1-\varepsilon)$.

	\medskip

	\par It remains to prove (\ref{eq:sum_j g_{n+j,s} estimate}). To this end, we introduce the non-isotropic Littlewood-Paley decomposition. Let $\psi$ be a $C^{\infty}$ function on $\R^2$ such that $\psi(\xi)=1$ for $\rho(\xi)\leq 1$, $\psi(\xi)=0$ for $\rho(\xi)\geq 2$ and $0\leq \psi(\xi)\leq1$ for all $\xi\in\R^2$. Define $\beta_k(\xi):=\psi(2^k\circ\xi)-\psi(2^{k+1}\circ\xi)$. Then $\beta_k$ is a positive $C^\infty$ function that is supported in $\{\xi\in\R^2:2^{-k-1}\leq\rho(\xi)\leq 2^{-k+1}\}$. It is easy to see that $\sum_{k\in\Z}\beta_k(\xi)=1$ for $\xi\neq0$.
	Note that the supports of $\beta_k$ have uniformly finite overlap. Using Plancherel's theorem twice, we get
	\begin{equation*}
		\begin{split}
			\text{LHS of (\ref{eq:sum_j g_{n+j,s} estimate})}
			\lesssim & \sum_{j\in\Z}\sum_{k:k+j>-\gamma(s+n)}\Norm{(K_j\phi_{j-\varepsilon s})*\check{\beta}_k*g_{n+j,s}}_{L^2(\A)}^2 \\
			&+ \sum_{j\in\Z}\sum_{k:k+j\leq -\gamma(s+n)}\Norm{(K_j\phi_{j-\varepsilon s})*\check{\beta}_k*g_{n+j,s}}_{L^2(\A)}^2 \\
			=&:(\text{I}_{n,s})+(\text{II}_{n,s}).
		\end{split}
	\end{equation*}
	We show that both (I$_{n,s}$) and (II$_{n,s}$) satisfy (\ref{eq:sum_j g_{n+j,s} estimate}).
	\medskip
	
\subsubsection{\bfseries Estimate for (I$_{n,s}$).}
\label{subsubsec:off-diagonal-I}
 To estimate the first part, corresponding to the case in which $k$ is relatively large and $\beta_k$ is supported on small frequencies, we need the following lemma.
	\begin{lemma}
		\label{lem:beta_k*g_(n,s)}
		For $k,n,s$ satisfying $k+n+s\geq0$,
		\begin{equation*}
			\Norm{\check{\beta_k}*g_{n,s}}_{L^2(\A)}\lesssim 2^{-n-s-k} \Norm{g_{n,s}}_{L^2(\A)}.
		\end{equation*}
	\end{lemma}
	Indeed, Lemma \ref{lem:beta_k*g_(n,s)} is trivial for $k+n+s<0$, since Young's inequality
	\[ \Norm{\check{\beta_k}*g_{n,s}}_{L^2(\A)}\leq \Norm{\check{\beta_0}}_{L^1(\R^2)}\Norm{g_{n,s}}_{L^2(\A)}\lesssim\Norm{g_{n,s}}_{L^2(\A)} \]
	holds for all $k,n,s$.
	In \cite[Section 5D]{Lai2024}, this lemma was proved for the isotropic Littlewood-Paley functions $\beta_k$ by using Schur's lemma and the cancellation property of $g_{n,s}$, whereas $\beta_k$ is non-isotropic in our setting.
	Although this non-isotropic version can be proved in a similar way, for completeness we give a proof of this lemma.
	\par We first state Schur's lemma.
	\begin{lemma}[Schur]
		\label{lem:Schur}
		Suppose $T$ is an operator on $L^2(\A)$ defined by
		\begin{equation*}
			T(f)(x)=\int_{\R^d}K(x,y)f(y)dy,
		\end{equation*}
		where the kernel $K(x,y)$ satisfies
		\begin{equation*}
			c_1:=\sup_{x\in\R^d}\int_{\R^d}|K(x,y)|dy<\infty \quad \text{ and } \quad c_2:=\sup_{y\in\R^d}\int_{\R^d}|K(x,y)|dx<\infty.
		\end{equation*}
		Then $T$ is bounded on $L^2(\A)$ with operator norm at most $\sqrt{c_1c_2}$.
	\end{lemma}
	A proof of Schur's lemma can be found, for example, in \cite{Parcet2009}. Next we give a proof of Lemma \ref{lem:beta_k*g_(n,s)}.
	\begin{proof}[Proof of Lemma \ref{lem:beta_k*g_(n,s)}]
		Recalling the definition of $g_{n,s}$, we have the cancellation property: for all $s\geq1$ and $Q\in\Q_{n+s-1}$, we have $\int_Qg_{n,s}(y)dy=0$. This implies
		\begin{equation*}
			\begin{split}
				\check{\beta_k}*g_{n,s}(x)&=\phantom{:}\int_{\R^2}\sum_{Q\in\Q_{n+s-1}}\left(\check{\beta_k}(x-y)-\check{\beta_k}(x-c_Q)\right)\chi_Q(y)g_{n,s}(y)dy \\
				&=:\int_{\R^2} L_k(x,y)g_{n,s}(y)dy,
			\end{split}
		\end{equation*}
		where the kernel is
		\begin{equation*}
			L_k(x,y):=\sum_{Q\in\Q_{n+s-1}}\left(\check{\beta_k}(x-y)-\check{\beta_k}(x-c_Q)\right)\chi_Q(y).
		\end{equation*}
		We apply Schur's lemma to estimate $\Norm{\check{\beta_k}*g_{n,s}}_{L^2}$.
In the following we write $\beta(\xi):=\beta_0(\xi)$. Recall that
		\begin{equation*}
			\beta_k(\xi)=\beta(2^k\circ\xi), \quad \check{\beta_k}(x)=2^{-3k}\check{\beta}(2^{-k}\circ x),
		\end{equation*}
		\begin{equation*}
			\nabla(\check{\beta_k})(x)=2^{-3k}\left(2^{-k}\partial_1\check{\beta}(2^{-k}\circ x),2^{-2k}\partial_2\check{\beta}(2^{-k}\circ x)\right).
		\end{equation*}
		Then by the chain rule, we have
		\begin{equation*}
			\begin{split}
				&\norm{\check{\beta_k}(x-y)-\check{\beta_k}(x-c_Q)} \\
				=&\norm{\int_0^1(c_Q-y)\cdot(\nabla\check{\beta_k})(x-ty-(1-t)c_Q)dt} \\
				\leq& 2^{-k}\int_0^1\norm{y_1-c_{Q,1}}2^{-3k}\norm{\partial_1\check{\beta}(2^{-k}\circ(x-ty-(1-t)c_Q))}dt \\
				&\qquad +2^{-2k}\int_0^1\norm{y_2-c_{Q,2}}2^{-3k}\norm{\partial_2\check{\beta}(2^{-k}\circ(x-ty-(1-t)c_Q))}dt \\
				\leq& 2^{-(n+s+k)}\int_0^12^{-3k}\norm{\partial_1\check{\beta}(2^{-k}\circ(x-ty-(1-t)c_Q))}dt \\
				&\qquad +2^{-2(n+s+k)}\int_0^12^{-3k}\norm{\partial_2\check{\beta}(2^{-k}\circ(x-ty-(1-t)c_Q))}dt.
			\end{split}
		\end{equation*}
		Therefore
		\begin{equation*}
			\begin{split}
				&\int_{\R^2}|L_k(x,y)|dx \\
				\leq& 2^{-(n+s+k)}\int_0^1 2^{-3k}\int_{\R^2}\norm{\partial_1\check{\beta}(2^{-k}\circ(x-ty-(1-t)\bar{y}^{n+s-1}))}dx\ dt \\
				&\qquad +2^{-2(n+s+k)}\int_0^1 2^{-3k}\int_{\R^2}\norm{\partial_2\check{\beta}(2^{-k}\circ(x-ty-(1-t)\bar{y}^{n+s-1}))}dx\ dt \\
				\leq& 2^{-(n+s+k)}\Norm{\partial_1\check{\beta}}_{L^1(\R^2)}+2^{-2(n+s+k)}\Norm{\partial_2\check{\beta}}_{L^1(\R^2)} \\
				\lesssim& 2^{-(n+s+k)}+2^{-2(n+s+k)} \lesssim 2^{-(n+s+k)},
			\end{split}
		\end{equation*}
		where the last step is due to the assumption $n+s+k\geq0$.
		The integral over $y$ can be controlled by
		\begin{equation*}
			\begin{split}
				&\int_{\R^2}|L_k(x,y)|dy \\
				\leq& 2^{-(n+s+k)} \sum_{Q\in\Q_{n+s-1}}\int_0^1\int_Q 2^{-3k} \norm{\partial_1\check{\beta}(2^{-k}\circ(x-ty-(1-t)c_Q))} dy\ dt \\
				&\qquad +2^{-2(n+s+k)} \sum_{Q\in\Q_{n+s-1}}\int_0^1\int_Q 2^{-3k} \norm{\partial_2\check{\beta}(2^{-k}\circ(x-ty-(1-t)c_Q))} dy\ dt.
			\end{split}
		\end{equation*}
		Since $\max\{|\partial_1\check{\beta}(y)|,|\partial_2\check{\beta}(y)|\}\leq |\nabla\check{\beta}(y)|$, and
		\begin{equation*}
			\begin{split}
				&\phantom{=}\sum_{Q\in\Q_{n+s-1}}\int_Q 2^{-3k}\norm{\nabla\check{\beta}(2^{-k}\circ(x-ty-(1-t)c_Q))}dy \\
				&=\sum_{Q\in \Q_{n+s+k-1}}\int_Q\norm{\nabla\check{\beta}(2^{-k}\circ x-ty-(1-t)c_Q)}dy   \lesssim 1.
			\end{split}
		\end{equation*}
		The preceding estimate is uniform in $x,t,k$.
		Indeed, put $m=n+s+k\geq0$, $a=2^{-k}\circ x$, and $F=|\nabla\check\beta|$.
		For $Q\in\Q_{m-1}$ and $y\in Q$, the coordinate bounds
		$|y_1-c_{Q,1}|\leq1$ and $|y_2-c_{Q,2}|\leq2$ show that
		$(1-t)(y-c_Q)\in D:=[-1,1]\times[-2,2]$ for every $t\in[0,1]$.
		Define $H(z):=\sup_{v\in D}F(z+v)$.
		The Schwartz decay of $\check\beta$ implies that $H\in L^1(\R^2)$.
		Since $a-ty-(1-t)c_Q=a-y+(1-t)(y-c_Q)$, we obtain
		\begin{equation*}
			\begin{split}
				\sum_{Q\in\Q_{m-1}}\int_Q F(a-ty-(1-t)c_Q)dy
				&\leq\sum_{Q\in\Q_{m-1}}\int_Q H(a-y)dy \\
				&=\Norm{H}_{L^1(\R^2)}\lesssim_\beta\Norm{F}_{L^1(\R^2)}.
			\end{split}
		\end{equation*}
		Here the last constant depends only on the fixed cutoff $\beta$, and is independent of $m, a, t$, where $m\geq0,a\in\R^2,0\leq t\leq1$.
		Consequently,
		\begin{equation*}
			\int_{\R^2}|L_k(x,y)|dy\lesssim 2^{-(n+s+k)}+2^{-2(n+s+k)}\lesssim 2^{-(n+s+k)}.
		\end{equation*}
		Finally, by Schur's lemma in Lemma \ref{lem:Schur}, we obtain Lemma \ref{lem:beta_k*g_(n,s)}.
	\end{proof}
	\par By Lemma \ref{lem:beta_k*g_(n,s)} and Lemma \ref{lem:g_off properties}, we have
	\begin{equation*}
		\begin{split}
			{\rm(I_{n,s})}
			&\leq \sum_{j\in\Z}\sum_{k:k+j>-\gamma(s+n)}\Norm{\check{\beta_k}*g_{n+j,s}}_{L^2(\A)}^2 \\
			&\lesssim \sum_{j\in\Z}\sum_{k:k+j>-\gamma(s+n)} 2^{-2(j+n+s+k)}\Norm{g_{n+j,s}}_{L^2(\A)}^2 \\
			&\lesssim 2^{-2(1-\gamma)(n+s)}\lambda\Norm{f}_{L^1(\A)}.
		\end{split}
	\end{equation*}
	This gives the desired result for (I$_{n,s}$).
	\medskip
	
\subsubsection{\bfseries Estimate of (II$_{n,s}$).}
\label{subsubsec:off-diagonal-II}
	\par Now consider (II$_{n,s}$), which is the case corresponding to $k$ near $-\infty$. We use Plancherel's identity to obtain
	\begin{equation*}
		\begin{split}
			(\text{II}_{n,s})=\sum_{j\in\Z}\sum_{k:k+j\leq -\gamma(s+n)} \int_{\R^2}\norm{\widehat{K}_j*\widehat{\phi_{j-\varepsilon s}}(\xi)}^2\norm{\beta_k(\xi)}^2\Norm{\widehat{g_{n+j,s}}(\xi)}_{L^2(\M)}^2 d\xi ,
		\end{split}
	\end{equation*}
	where
	\begin{equation*}
		\begin{split}
			\widehat{K}_j*\widehat{\phi_{j-\varepsilon s}}(\xi)&=\int_{\R^2}\widehat{K_j}(\xi-y)\widehat{\phi_{j-\varepsilon s}}(y)dy \\
			&=\int_{\R^2}\widehat{K}(2^{-j}\circ \xi-2^{-j}\circ y)\widehat{\phi}(2^{\varepsilon s-j}\circ y)2^{3(\varepsilon s-j)}dy \\
			&=\int_{\R^2}\widehat{K}(2^{-j}\circ \xi- y)\widehat{\phi}(2^{\varepsilon s}\circ y)2^{3\varepsilon s}dy.
		\end{split}
	\end{equation*}
	By the weighted Cauchy-Schwarz inequality,
	\begin{equation*}
		\begin{split}
			\norm{\widehat{K}_j*\widehat{\phi_{j-\varepsilon s}}(\xi)}^2
			&=\norm{\int_{\R^2}\widehat{K}(2^{-j}\circ \xi- y)\widehat{\phi}(2^{\varepsilon s}\circ y)2^{3\varepsilon s}dy}^2 \\
			&\leq \|\widehat{\phi}\|_{1} \int_{\R^2}\norm{\widehat{K}(2^{-j}\circ \xi- y)}^2|\widehat{\phi}(2^{\varepsilon s}\circ y)|2^{3\varepsilon s}dy.
		\end{split}
	\end{equation*}
	Therefore, (II$_{n,s}$) can be split into two parts:
	\begin{equation*}
		\begin{split}
			(\text{II}_{n,s})\leq & \sum_{j\in\Z}\sum_{k+j\leq -\gamma(s+n)} \\
			&\quad \int_{\R^2}\bigg(\int_{\rho(y)\leq \frac1{10} 2^{(s+n)\gamma}}\norm{\widehat{K}(2^{-j}\circ \xi- y)}^2|\widehat{\phi}(2^{\varepsilon s}\circ y)|2^{3\varepsilon s}dy\bigg) \\
			&\qquad{}\times\norm{\beta_k(\xi)}^2\Norm{\widehat{g_{n+j,s}}(\xi)}_{L^2(\M)}^2 d\xi  \\
			&+ \sum_{j\in\Z}\sum_{k+j\leq -\gamma(s+n)} \\
			&\quad \int_{\R^2}\bigg(\int_{\rho(y)> \frac1{10} 2^{(s+n)\gamma}}\norm{\widehat{K}(2^{-j}\circ \xi- y)}^2|\widehat{\phi}(2^{\varepsilon s}\circ y)|2^{3\varepsilon s}dy\bigg) \\
			&\qquad{}\times\norm{\beta_k(\xi)}^2\Norm{\widehat{g_{n+j,s}}(\xi)}_{L^2(\M)}^2 d\xi  \\
			=&: (\text{II}_{n,s,1}) + (\text{II}_{n,s,2}).
		\end{split}
	\end{equation*}
	On $\supp(\beta_k)$, we have $2^{-j-k-1}\leq\rho(2^{-j}\circ\xi)\leq2^{-j-k+1}$. If $k+j\leq-\gamma(s+n)$ and $\rho(y)\leq\frac1{10}2^{(s+n)\gamma}$, then $\rho(y)\leq\frac15\rho(2^{-j}\circ\xi)$. The triangle inequality therefore gives
	\[
	 \frac45\rho(2^{-j}\circ\xi)\leq\rho(2^{-j}\circ\xi-y)\leq\frac65\rho(2^{-j}\circ\xi).
	\]
	In particular, $\rho(2^{-j}\circ\xi-y)\geq\frac25\,2^{-j-k}$, so $\rho(2^{-j}\circ\xi-y)^{-2\delta}\leq(5/2)^{2\delta}(2^{-j-k})^{-2\delta}$. Therefore, by Lemma \ref{prop:K(x) kernel preoperties} and Plancherel's identity we have
	\begin{equation*}
		\begin{split}
			(\text{II}_{n,s,1})&\lesssim \sum_{j\in\Z}\sum_{k+j\leq -\gamma(s+n)} \\
			&\qquad \int_{\R^2}\bigg(\int_{\rho(y)\leq \frac1{10} 2^{(s+n)\gamma}}\rho(2^{-j}\circ\xi-y)^{-2\delta}|\widehat{\phi}(2^{\varepsilon s}\circ y)|2^{3\varepsilon s}dy\bigg) \\
			&\qquad\qquad{}\times\norm{\beta_k(\xi)}^2\Norm{\widehat{g_{n+j,s}}(\xi)}_{L^2(\M)}^2 d\xi \\
			&\lesssim \sum_{j\in\Z}\sum_{k:k+j\leq -\gamma(s+n)} (2^{-j-k})^{-2\delta}\Norm{g_{n+j,s}}_{L^2(\A)}^2 \\
			&\lesssim 2^{-2(s+n)\gamma\delta}\lambda\Norm{f}_{L^1(\A)},
		\end{split}
	\end{equation*}
	where in the last line we use Lemma \ref{lem:g_off properties}.
	Next consider the case $\rho(y)>\frac1{10} 2^{(s+n)\gamma}$. We use the estimates
	\begin{equation*}
		|\widehat{\phi}(2^{\varepsilon s}\circ y)|\lesssim\frac1{(2^{\varepsilon s}\rho(y))^N}, \quad \norm{\widehat{K}(2^{-j}\circ\xi-y)}\lesssim 1,
	\end{equation*}
	where $N$ can be chosen sufficiently large, to obtain
	\begin{equation*}
		\begin{split}
			(\text{II}_{n,s,2})&\lesssim \sum_{j\in\Z}\sum_{k+j\leq -\gamma(s+n)} \int_{\R^2}\int_{\rho(y)\geq\frac1{10} 2^{(s+n)\gamma}}\frac{2^{3s\varepsilon}}{(2^{\varepsilon s}\rho(y))^N}\norm{\beta_k(\xi)}^2dy\Norm{\widehat{g_{n+j,s}}(\xi)}_{L^2(\M)}^2d\xi \\
			&\lesssim 2^{-\varepsilon s(N-3)}\int_{\rho(y)\geq \frac1{10}2^{(s+n)\gamma}}\frac1{\rho(y)^N}dy \  \sum_{j\in\Z}\Norm{g_{n+j,s}}_{L^2(\A)}^2 \\
			&\lesssim 2^{-(s+n)\gamma(N-3)}\lambda\Norm{f}_{L^1(\A)}.
		\end{split}
	\end{equation*}
	Combining the estimates of ($\rm II_{n,s,1}$) and ($\rm II_{n,s,2}$), we obtain the desired estimate of (II$_{n,s}$).
Hence we obtain (\ref{eq:sum_j g_{n+j,s} estimate}), the proof of Lemma \ref{lem:K(phi) estimate} is complete.

\section{$L^p$-boundedness}
\label{sec:lp-boundedness}
\par
In this section, we prove Proposition \ref{prop:A_j^s Lp bound} and then use the complex interpolation to show Theorem \ref{thm:sup A_jf bound}. In the previous section, we proved Proposition \ref{prop:A_j^s weak L1 bound}, i.e., the operator $f\mapsto\{A_j^s(f)\}_{j\in\Z}$ is bounded from $L^1(\A)$ to $\Lambda^{1,\infty}(\A;\ell^{\infty})$, at least when $s$ is a negative real number.
Moreover, $f\mapsto\{A_j^s(f)\}_{j\in\Z}$ is trivially bounded from $L^{\infty}(\A)$ to $L^{\infty}(\A,\ell^\infty)$.
To deduce the $L^p(\A)\to L^p(\A;\ell^\infty)$ boundedness, we introduce a noncommutative version of the Marcinkiewicz interpolation theorem.
\begin{theorem}[{\cite[Theorem 4.2.2]{PisierXu2003}}]
	\label{thm:NC interpolation}
	Let $1\leq p_0<p_1\leq\infty$. Let $\{S_j\}_{j\in\Z}$ be a sequence of subadditive maps from $L^{p_0}_+(\A)+L^{p_1}_+(\A)$ into $L^0_+(\A)$.
	If
	\begin{equation*}
		\begin{split}
			\Norm{\{S_j(f)\}_j}_{\Lambda^{p_0,\infty}(\A;\ell^\infty)}&\leq C_0\Norm{f}_{L^{p_0}(\A)}, \quad \forall f\in L^{p_0}_+(\A), \\
		    \Norm{\{S_j(f)\}_j}_{L^{p_1}(\A;\ell^\infty)}&\leq C_1\Norm{f}_{L^{p_1}(\A)}, \quad \forall f\in L^{p_1}_+(\A),
		\end{split}
	\end{equation*}
	then for any $p_0<p<p_1$,
	\begin{equation*}
		\Norm{\{S_j(f)\}_j}_{L^{p}(\A;\ell^\infty)}\leq C_p\Norm{f}_{L^{p}(\A)}
	\end{equation*}
	holds for all $f\in L^{p}_+(\A)$. The constant $C_p$ is controlled by
	\begin{equation*}
		C_p\lesssim C_0^{1-\theta}C_1^\theta \left(\frac1{p_0}-\frac1p\right)^{-2}\left(\frac1p-\frac1{p_1}\right)^{-1},
	\end{equation*}
	where $\theta$ is determined by $1/p=(1-\theta)/p_0+\theta/p_1$.
\end{theorem}

We first prove Proposition \ref{prop:A_j^s Lp bound} in the case $s$ is real.
\begin{proof}[Proof of Proposition \ref{prop:A_j^s Lp bound} for real $s$]
	We first consider real $s<0$. Under this assumption, for each $j$, $A_j^s$ is a linear (hence subadditive) map and preserves positive elements (since the kernels $\nu_j^s$ are positive, see Lemma \ref{lem:Bessel kernel properties}).
	Since $\{A_j^s\}$ is bounded from $L^1(\A)$ to $\Lambda^{1,\infty}(\A;\ell^\infty(\Z))$, and also bounded from $L^\infty(\A)$ to $L^\infty(\A;\ell^\infty(\Z))$, by Theorem \ref{thm:NC interpolation}, for $1<p<\infty$, we have that $f\mapsto\{A_j^s(f)\}_{j\in\Z}$ is bounded from $L^p(\A)\to L^p(\A;\ell^{\infty})$.
\end{proof}

For complex $s$ with $\Re(s)<0$, the kernel $\nu_j^s$ may be complex-valued and $A_j^s$ need not preserve positivity. Thus Theorem \ref{thm:NC interpolation} does not apply directly.
To treat complex $s$ with negative real part, we need the following two lemmas.
\begin{lemma}
	\label{lem:s'+s''}
	For two complex numbers $s'$ and $s''$, let $s=s'+s''$. Then, initially for Schwartz functions,
	\begin{equation*}
		A_j^{s}(f)=A_j^{s'}(f)*(G_{s''})_j,
	\end{equation*}
	where $(G_{s''})_j(x)=2^{3j}G_{s''}(2^j\circ x)$.
\end{lemma}
The identity follows on the Fourier side from $F_s=F_{s'}F_{s''}$. In the applications below, $\Re(s'')<0$, so Lemma \ref{lem:Bessel kernel properties} gives $G_{s''}$ is an $L^1$ function, and convolution with $(G_{s''})_j$ is bounded on $L^p$.

\begin{lemma}
	\label{lem:sup_j (G_s)_j*h_j Lp-bounded}
	For $1<p<\infty$ and $s=\sigma+it\in\mathbb{C}$ with $\sigma<0$, there exists a constant $C(p,s)>0$ such that
	\begin{equation*}
		\Big\|\sup_{j\in\Z}\big((G_s)_j*h_j\big)\Big\|_{L^p(\A)}\leq C(p,s)\Big\|\sup_{j\in\Z} h_j\Big\|_{L^p(\A)}, \quad \forall \ \{h_j\}_{j\in\Z}\in L^p(\A;\ell^\infty(\Z)).
	\end{equation*}
	The constant $C(p,s)$ satisfies
	\begin{equation*}
		C(p,s) \lesssim_p \frac{\Gamma(-\sigma/2)}{|\Gamma(-s/2)|}.
	\end{equation*}
\end{lemma}
\begin{proof}
	It suffices to consider self-adjoint $h_j$, by decomposing a general sequence into its real and imaginary parts. By Lemma \ref{lem:Bessel kernel properties}, $G_s$ is radial, and $G_\sigma$ is positive, radially decreasing and has integral $1$. Write
	\[
	 G_s=G_s^{(1)}+iG_s^{(2)}, \qquad G_\sigma(x)=H_\sigma(|x|).
	\]
	It suffices to prove
	\begin{equation}
		\label{eq:G_s,j^1 maixmal Lp ineq}
		\Big\|\sup_j\big((G_s^{(1)})_j*h_j\big)\Big\|_{L^p(\A)}\leq C_1(p,s)\Big\|\sup_jh_j\Big\|_{L^p(\A)}
	\end{equation}
	and
	\begin{equation}
		\label{eq:G_s,j^2 maixmal Lp ineq}
		\Big\|\sup_j\big((G_s^{(2)})_j*h_j\big)\Big\|_{L^p(\A)}\leq C_2(p,s)\Big\|\sup_jh_j\Big\|_{L^p(\A)}.
	\end{equation}
	Choose $\tilde h\geq0$ in $L^p(\A)$ such that
	\begin{equation}
		\label{eq:h_j<tilde h}
		-\tilde h\leq h_j\leq\tilde h \quad\text{for all }j, \qquad
		\big\|{\tilde h}\big\|_{L^p(\A)}\leq2\Big\|\sup_jh_j\Big\|_{L^p(\A)}.
	\end{equation}
	Set $C_s=\Gamma(-\sigma/2)/|\Gamma(-s/2)|$. By (\ref{eq:G_s domination}), both $|G_s^{(1)}|$ and $|G_s^{(2)}|$ are bounded by $C_sG_\sigma$. Hence, for $\ell=1,2$,
	\[
	 -C_s(G_\sigma)_j*\tilde h\leq (G_s^{(\ell)})_j*h_j\leq C_s(G_\sigma)_j*\tilde h.
	\]
	To bound the positive convolutions, let
	\[
	 H_{\sigma,j}(r)=2^{3j}H_\sigma(2^{3j/2}r), \qquad
	 E_j(r)=\{y:2^{-j}y_1^2+2^jy_2^2\leq r^2\}.
	\]
	The ellipse $E_j(r)$ is contained in the rectangle
	$[-2^{j/2}r,2^{j/2}r]\times[-2^{-j/2}r,2^{-j/2}r]$, whose area is $4r^2$.
	By Theorem \ref{thm:strong maximal boundedness}, there is a positive $a\in L^p(\A)$ such that
	\begin{equation}
		\label{eq:int h_j <ar^2}
		\int_{E_j(r)}\tilde h(x-y)dy\leq a(x)r^2 \quad\text{for all }j,r, \qquad
		\Norm{a}_{L^p(\A)}\lesssim_p\big\|{\tilde h}\big\|_{L^p(\A)}.
	\end{equation}
	Since $H_{\sigma,j}$ is decreasing and tends to zero at infinity, its Lebesgue--Stieltjes measure $d(-H_{\sigma,j})$ is positive. Tonelli's theorem gives
	\[
	 \big[(G_\sigma)_j*\tilde h\big](x)
	 =\int_0^\infty\left(\int_{E_j(r)}\tilde h(x-y)dy\right)d(-H_{\sigma,j}(r))
	 \leq a(x)\int_0^\infty r^2d(-H_{\sigma,j}(r)).
	\]
	The mass normalization in Lemma \ref{lem:Bessel kernel properties} yields
	\[
	 \int_0^\infty r^2d(-H_{\sigma,j}(r))
	 =2\int_0^\infty rH_{\sigma,j}(r)dr
	 =\frac1\pi\Norm{G_\sigma}_{L^1(\R^2)}=\frac1\pi.
	\]
	Consequently,
	\begin{equation*}
		(G_s^{(1)})_j*h_j\leq\frac{C_s}\pi a,
	\end{equation*}
	and
	\begin{equation*}
		-\frac{C_s}\pi a\leq(G_s^{(1)})_j*h_j.
	\end{equation*}
	The same bounds hold for $(G_s^{(2)})_j*h_j$, with the same positive majorant $G_\sigma$. Combining these inequalities with (\ref{eq:h_j<tilde h}) and (\ref{eq:int h_j <ar^2}) proves (\ref{eq:G_s,j^1 maixmal Lp ineq}) and (\ref{eq:G_s,j^2 maixmal Lp ineq}), with
	\[
	 C_1(p,s)+C_2(p,s)\lesssim_p C_s=\frac{\Gamma(-\sigma/2)}{|\Gamma(-s/2)|}.
	\]
\end{proof}

Now we start to prove Proposition \ref{prop:A_j^s Lp bound} for general complex $s$.
\begin{proof}[Proof of Proposition \ref{prop:A_j^s Lp bound} for complex $s$]
\par For complex $s$ with $\Re(s)<0$, we prove that $f\mapsto\{A_j^s(f)\}_{j\in\Z}$ is bounded from $L^p(\A)\to L^p(\A;\ell^{\infty})$.
Suppose $s=\sigma+i\rho$, where $\sigma<0$ and $\rho\in\R$. We choose $s'\in\R$ that satisfies $\sigma<s'<0$ and let $s''=s-s'$. Then $\Re(s'')<0$, so Lemma \ref{lem:Bessel kernel properties} gives $G_{s''}\in L^1$ and $G_{s''}(-x)=G_{s''}(x)$, as used in the duality calculation below. Lemma \ref{lem:s'+s''} implies
\begin{equation*}
	A_j^s(f)=A_j^{s'}(f)*(G_{s''})_j, \quad \forall\ j\in\Z.
\end{equation*}
Let $p'$ be the conjugate exponent of $p$. Choose a   sequence $\{g_j\}_j$ in the unit ball of the space $L^{p'}(\A;\ell^1)$. By duality, we have
\begin{equation}
	\label{eq:sum_j phi(A_j^sf.g)}
	\begin{split}
	    \Big|\sum_j\nu(A_j^sf\cdot g_j)\Big|&=\Big|\sum_j\nu\big(A_j^{s'}(f)*(G_{s''})_j\cdot g_j\big)\Big| \\
		&=\Big|\sum_j\nu\big(A_j^{s'}(f)\cdot(G_{s''})_j*g_j\big)\Big| \\
		&\leq \big\|{\{A_j^{s'}(f)\}_j}\big\|_{L^p(\A;\ell^\infty)}\big\|\{(G_{s''})_j*g_j\}_j\big\|_{L^{p'}(\A;\ell^1)}.
	\end{split}
\end{equation}
Since we have proved Proposition \ref{prop:A_j^s Lp bound} for real $s$, we have
\begin{equation}
	\label{eq:A_j^s'(f) L^p(l^infty) norm}
	\big\|{\{A_j^{s'}(f)\}_j}\big\|_{L^p(\A;\ell^\infty)} \leq A(p,s') \Norm{f}_{L^p(\A)},
\end{equation}
where $A(p,s')>0$ is a constant depending only on $s'$ and $p$. Moreover, by duality and Lemma \ref{lem:sup_j (G_s)_j*h_j Lp-bounded} we have
\begin{equation}
	\label{eq:G_s''_j*g_j L^p'(l^1) norm}
	\begin{split}
		\big\|\{(G_{s''})_j*g_j\}_j\big\|_{L^{p'}(\A;\ell^1)}&=\sup_{\|\{h_j\}_j\|_{L^p(\A;\ell^\infty)}\leq1}\Big|\nu\Big(\sum_j\big((G_{s''})_j*g_j\big)\cdot h_j\Big)\Big| \\
		&=\sup_{\|\{h_j\}_j\|_{L^p(\A;\ell^\infty)}\leq1}\Big|\nu\Big(\sum_jg_j\cdot \big((G_{s''})_j*h_j\big)\Big)\Big| \\
		&\leq \big\|\{g_j\}_j\big\|_{L^{p'}(\A;\ell^1)} \\
        &\qquad{}\times\sup_{\|h\|_{L^p(\A;\ell^\infty)}\leq1}
        \big\|\sup_j (G_{s''})_j*h_j\big\|_{L^p(\A)} \\
		&\leq C(p,s''),
	\end{split}
\end{equation}
where the constant $C(p,s'')$ in (\ref{eq:G_s''_j*g_j L^p'(l^1) norm}) comes from Lemma \ref{lem:sup_j (G_s)_j*h_j Lp-bounded}. Combining (\ref{eq:sum_j phi(A_j^sf.g)}), (\ref{eq:A_j^s'(f) L^p(l^infty) norm}) and (\ref{eq:G_s''_j*g_j L^p'(l^1) norm}) gives
\begin{equation}
	\label{eq:sum_j phi(A_j^s(f)g) estimate}
	\bigg|\sum_j\nu(A_j^sf\cdot g_j)\bigg|\leq A(p,s')C(p,s'')\Norm{f}_{L^p(\A)}.
\end{equation}
Taking the supremum over $\{g_j\}$ in the unit ball of $L^{p'}(\A;\ell^1)$, we obtain Proposition \ref{prop:A_j^s Lp bound} for general complex $s$.
\end{proof}

\begin{remark}
	\label{rem:growth of C_p(s)}
	Let $C_p(s)$ be the best constant in the inequality
	\begin{equation*}
		\Big\|{\sup_jA_j^sf}\Big\|_{L^p(\A)}\lesssim C_p(s) \Norm{f}_{L^p(\A)}.
	\end{equation*}
	Suppose $s=\sigma+it$. For fixed $1<p<\infty$ and $\sigma<0$, we estimate the growth of $t\mapsto C_p(\sigma+it)$.
	For convenience, we may let $s'=\sigma/2$ and $s''=\sigma/2+it$.
	Note that $s'<0$ and $\Re(s'')<0$. By (\ref{eq:sum_j phi(A_j^s(f)g) estimate}) and $C_p(s)$ estimate in Lemma \ref{lem:sup_j (G_s)_j*h_j Lp-bounded}, we deduce that
	\begin{equation*}
		C_p(\sigma+it)\lesssim A(p,\sigma/2)C(p,\sigma/2+it)\lesssim_{p,\sigma} \frac{1}{\big|\Gamma(-\frac{\sigma}{4}-i\frac t2)\big|}.
	\end{equation*}
	Then we use the asymptotic formula for the gamma function
	\begin{equation*}
		\frac{1}{|\Gamma(\sigma+it)|}\sim \frac{1}{\sqrt{2\pi}}|t|^{\frac12-\sigma}e^{\frac{\pi}{2}|t|}, \quad \text{ as } |t|\to\infty.
	\end{equation*}
	We deduce that
	\begin{equation*}
		C_p(\sigma+it)\lesssim_{p,\sigma} |t|^{\frac12+\frac{\sigma}{4}}e^{\frac{\pi}{4}|t|} \lesssim_{\sigma} e^{|t|}, \quad \text{ as } |t|\to\infty.
	\end{equation*}
\end{remark}

\medskip

\par We now use the complex interpolation to prove Theorem \ref{thm:sup A_jf bound}.
\begin{proof}[Proof of Theorem \ref{thm:sup A_jf bound}]
	The case $p>2$ follows by interpolating Proposition \ref{prop:A_j^s L2 bound} with the trivial $L^{\infty}$ bound. It suffices to consider the case $1<p<2$.
	\par Let $1<p<2$. By duality and homogeneity, it suffices to prove
	\begin{equation*}
		\bigg|\sum_{j\in\Z}\nu\big(A_j(f)g_j\big)\bigg|\leq C
	\end{equation*}
	for positive $f$ and finite sequences $g=\{g_j\}_j$ with $\Norm{f}_{L^p(\A)}\leq1$ and $\Norm{g}_{L^{p'}(\A;\ell^1)}\leq1$, where $C$ depends only on $p$. We first assume that $f$ has finite spectrum and finite-trace support, and that the entries of $g$ lie in a common finite-trace corner of $\A$. Choose $p_1$, $s_0$, $s_1$, depending only on $p$, such that
	\begin{equation*}
		1<p_1<p, \quad 0<s_0<\frac12, \quad s_1<0,
	\end{equation*}
	and
	\begin{equation*}
		\begin{cases}
			0&=(1-t)s_0+ts_1, \\
			\frac1p&=(1-t)\frac12+t\cdot\frac1{p_1},
		\end{cases}
	\end{equation*}
	for some $0<t<1$. Define a complex function
	\begin{equation*}
		F(z)=f^{(\frac{1-z}2+\frac z{p_1})p}.
	\end{equation*}
	Then $F(t)=f$, and functional calculus gives
	\begin{equation*}
		\Norm{F(iy)}_{L^2(\A)}=\Norm{f}_{L^p(\A)}^{p/2}\leq1,
		\qquad
		\Norm{F(1+iy)}_{L^{p_1}(\A)}=\Norm{f}_{L^p(\A)}^{p/p_1}\leq1
	\end{equation*}
	for all $y\in\R$. Next, choose a function $h=\{h_j\}_j$,
	\begin{equation*}
		h: \{z\in\mathbb{C}:0\leq\Re(z)\leq1\}
		\to L^2(\A;\ell^1)+L^{p_1'}(\A;\ell^1),
	\end{equation*}
	such that $h$ is bounded and continuous on the closed strip and analytic in its interior, and satisfies
	\begin{equation*}
		h(t)=\{h_j(t)\}_j=\{g_j\}_j=g,
	\end{equation*}
	\begin{equation}
		\label{eq:||h(ia)|| bounded}
		\sup_{a\in\R}\Norm{h(ia)}_{L^2(\A;\ell^1)}\leq2\Norm{g}_{L^{p'}(\A;\ell^1)}\leq2,
	\end{equation}
	\begin{equation}
		\label{eq:||h(1+ia)|| bounded}
		\sup_{a\in\R}\Norm{h(1+ia)}_{L^{p_1'}(\A;\ell^1)}\leq2\Norm{g}_{L^{p'}(\A;\ell^1)}\leq2.
	\end{equation}
	The existence of $h$ with these bounds follows from the isometric complex interpolation identity
	\begin{equation*}
		L^{p'}(\A;\ell^1)=\big(L^2(\A;\ell^1),L^{p_1'}(\A;\ell^1)\big)_t,
	\end{equation*}
	since $1/p'=(1-t)/2+t/p_1'$; see Junge and Xu \cite[Proposition 2.5]{JungeXu2007}. Applying this identity in the common finite-trace corner and to the finitely many nonzero coordinates of $g$, we may also require that each $h_j(z)$ lies in that corner and that $h_j(z)=0$ whenever $g_j=0$.
	\par Since $f$ has finite spectrum and finite-trace support, $F$ is bounded and continuous on the closed strip as an $L^2(\A)$-valued function and analytic in its interior. Moreover, $p_1'>2$, so $L^{p_1'}$ embeds continuously into $L^2$ on the common finite-trace corner. Thus each $h_j$ is also bounded and continuous on the closed strip with values in $L^2(\A)$ and analytic in its interior. The Fourier multiplier defining $A_j^{(1-z)s_0+zs_1}$ is analytic in $z$ and uniformly bounded on the closed strip, by $s_0<1/2$ and the Fourier decay of $d\mu$. Hence these operators form an analytic family on $L^2(\A)$.
	For $\delta>0$, we define
	\begin{equation*}
		\Phi_\delta(z)=e^{\delta(z^2-t^2)}\sum_{j\in\Z}\nu\big(A_j^{(1-z)s_0+zs_1}(F(z))h_j(z)\big).
	\end{equation*}
	The sum is finite, so $\Phi_\delta$ is analytic in the open strip $0<\Re(z)<1$ and continuous on its closure. The sum before multiplication by $e^{\delta(z^2-t^2)}$ is bounded on the closed strip, by the $L^2$ bounds just established. Consequently, $\Phi_\delta(z)$ tends to zero uniformly for $0\leq\Re(z)\leq1$ as $|\Im(z)|\to\infty$. Note also that $\Phi_\delta(t)=\sum_{j\in\Z}\nu\left(A_j(f)g_j\right)$.
	\par We now estimate $\Phi_\delta$ on the two boundary lines. By duality, (\ref{eq:||h(ia)|| bounded}) and Remark \ref{rem:growth of C(s)}, we have
	\begin{equation*}
		\begin{split}
			|\Phi_\delta(iy)|
			&\leq e^{-\delta(y^2+t^2)}
			\big\|\{A_j^{s_0+i(s_1-s_0)y}(F(iy))\}_j\big\|_{L^2(\A;\ell^\infty)} \\
			&\qquad{}\times\Norm{h(iy)}_{L^2(\A;\ell^1)} \\
			&\lesssim_p e^{-\delta(y^2+t^2)}\big(1+|(s_1-s_0)y|\big).
		\end{split}
	\end{equation*}
	Similarly, by the $L^{p_1}(\A;\ell^\infty)$--$L^{p_1'}(\A;\ell^1)$ duality, (\ref{eq:||h(1+ia)|| bounded}) and Remark \ref{rem:growth of C_p(s)}, we obtain
	\begin{equation*}
		\begin{split}
			|\Phi_\delta(1+iy)|
			&\leq e^{\delta(1-t^2-y^2)}
			\big\|\{A_j^{s_1+i(s_1-s_0)y}(F(1+iy))\}_j\big\|_{L^{p_1}(\A;\ell^\infty)} \\
			&\qquad{}\times\Norm{h(1+iy)}_{L^{p_1'}(\A;\ell^1)} \\
			&\lesssim_p e^{\delta(1-t^2-y^2)+|(s_1-s_0)y|}.
		\end{split}
	\end{equation*}
	Fixing $\delta=1$, we therefore have
	\begin{equation*}
		\norm{\Phi_\delta(i\R)}\leq C, \quad \norm{\Phi_\delta(1+i\R)}\leq C,
	\end{equation*}
	where $C$ depends only on $p$, and is independent of $f$, $g$, the number of nonzero coordinates and the finite-trace corner. Applying the maximum modulus principle on expanding rectangles in the strip gives $|\Phi_\delta(t)|\leq C$. Therefore
	\begin{equation*}
		\bigg|{\sum_{j\in\Z}\nu\left(A_j(f)g_j\right)}\bigg|\leq C.
	\end{equation*}
	Finite sequences whose entries lie in a common finite-trace corner are dense in $L^{p'}(\A;\ell^1)$, by coordinate truncation and compression by finite-trace projections. Since the bound is uniform over these sequences, the map $g\mapsto\sum_j\nu(A_j(f)g_j)$ extends to a bounded functional on $L^{p'}(\A;\ell^1)$. By Proposition \ref{prop:duality between L^p l^1 and L^p l^infty}, we obtain
	 \begin{equation*}
	 	\Big\|{\sup_{j\in\Z}A_j(f)}\Big\|_{L^p(\A)}\leq C.
	 \end{equation*}
	By homogeneity and the density of positive elements with finite spectrum and finite-trace support in $L^p_+(\A)$, and then by the usual decomposition into real and imaginary parts, this estimate extends to general $f$ in $L^p(\A)$. Thus
	 \begin{equation*}
	 	\Big\|{\sup_{j\in\Z}A_j(f)}\Big\|_{L^p(\A)}\lesssim\Norm{f}_{L^p(\A)}
	 \end{equation*}
	 holds for general $f$.
\end{proof}

\section{Generalization to submanifolds}
\label{sec:submanifolds}
\par In this section, we generalize the results stated in Section \ref{sec:parabola-model} to submanifolds of finite type.
\subsection{The polynomial case}
\label{subsec:polynomial-case}
\par We first consider polynomial parametrizations, as described in the introduction:
\[ \gamma(t)=P(t)=(P_1(t),\dots,P_n(t)) \]
where the $P_j(t)$ are real-valued polynomials of $t\in\R^k$. We define $M^P_r(f)$ by
\[ M^P_r(f)(x)=\frac{1}{|B_k(r)|} \int_{|t|<r} f(x-P(t))dt. \]
In this case, we let $\A=L^{\infty}(\R^n)\ot\M$. In the following, we prove Theorem \ref{thm:polynomial case}.
Our goal is to show that, for each $1<p\leq\infty$, we have
\begin{equation}
	\label{eq:sup_r>0 M_r^P Lp bounded}
	\Big\|{\sup_{r>0} M_r^P(f)}\Big\|_{L^p(\A)}\lesssim_p \Norm{f}_{L^p(\A)}.
\end{equation}
\par Following \cite[Chapter XI]{Stein1993}, we separate the monomials appearing in $P_1,\dots,P_n$ by lifting the problem to a higher-dimensional space.
\par Let $d$ denote the maximum degree of $P_j(t)$, $j=1,\dots,n$. We consider the collection of all monomials $t^\alpha$ with $1\leq|\alpha|\leq d$, where $t^\alpha=t_1^{\alpha_1}\cdots t_k^{\alpha_k}$. The monomial $t^{\alpha}$ has degree $|\alpha|=\alpha_1+\cdots+\alpha_k$. Let $N$ denote the number of these monomials and work in the larger space $\R^N$, whose coordinates are labeled by the multi-indices $\alpha$ with $1\leq|\alpha|\leq d$, i.e. $\R^N=\{(x_\alpha)\}_{1\leq|\alpha|\leq d}$.
\par Consider the monomial map
\begin{equation*}
	\mathfrak{p} : \R^k\to\R^N, \quad \mfp(t)=(t^\alpha)_{1\leq|\alpha|\leq d}.
\end{equation*}
For a function $f$ on $\R^N$, define the averaging operator by
\begin{equation*}
	M^\mfp_r(f)(x)=\frac{1}{|B_k(r)|}{\int_{|t|<r} f(x-\mfp(t))dt}.
\end{equation*}
For each positive integer $k$, write
\begin{equation*}
	\A^k:=L^\infty(\R^k)\ot\M.
\end{equation*}
In the monomial case, we work with the von Neumann algebra $\A^N=L^\infty(\R^N)\ot\M$.

\begin{prop}
	\label{prop:sup_r M_r^mfp Lp bounded}
	For each $1<p\leq\infty$, we have
	\begin{equation}
		\label{eq:M_r^p Lp bounded}
		\Big\|{\sup_{r>0}M_r^\mfp(f)}\Big\|_{L^p(\A^N)}\lesssim_p \Norm{f}_{L^p(\A^N)}.
	\end{equation}
\end{prop}
\par To prove (\ref{eq:M_r^p Lp bounded}), as in the parabola case, choose a smooth nonnegative function $\eta$ on $\R^k$ with $\eta(t)=1$ for $|t|\leq1$ and $\eta(t)=0$ for $|t|\geq2$.
Define the averaging operator
\begin{equation}
	\label{eq:A_j^mfp def}
	A_j^\mfp(f)= 2^{kj}\int_{\R^k}f(x-\mfp(t))\eta(2^jt)dt.
\end{equation}
It suffices to prove that, for $f\in\A^N_{c,+}$,
\begin{equation}
	\label{eq:A_j^p Lp bounded}
	\Big\|{\sup_{j\in\Z}A_j^{\mfp}(f)}\Big\|_{L^p(\A^N)}\lesssim \Norm{f}_{L^p(\A^N)}.
\end{equation}
\par The proof of (\ref{eq:A_j^p Lp bounded}) follows the same scheme as in the parabola case. Indeed, we define the corresponding non-isotropic dilation on $\R^N$ by
\begin{equation}
	\label{eq:dilation in R^N}
	x\mapsto \delta\circ x=(\delta^{|\alpha|}x_\alpha)
\end{equation}
for $x=(x_\alpha)\in \R^N$.
With this notation, we have
\begin{equation*}
	\begin{split}
		A_j^\mfp f(x) =\int_{\R^k} f(x-2^{-j}\circ \mfp(t))\eta(t)dt =\int_{\R^k}f[(x_\alpha-2^{-j|\alpha|}t^{\alpha})_{\alpha}]\eta(t)dt.
	\end{split}
\end{equation*}
\par We also have the corresponding quasi-norm on $\R^N$
\begin{equation*}
	\rho(x)=\max_{1\leq|\alpha|\leq d}\{|x_\alpha|^{1/|\alpha|}\}
\end{equation*}
and the associated metric $\rho(x,y)=\rho(x-y)$, since
\begin{equation*}
	\rho(x+y)\leq \rho(x)+\rho(y),
\end{equation*}
by the subadditivity of $u\mapsto u^{1/|\alpha|}$ on $[0,\infty)$ for $|\alpha|\geq1$. Note that $\rho(\delta\circ x)=\delta\rho(x)$. The corresponding family of balls is
\begin{equation*}
	B^\rho(x,r):=\{y\in\R^N:\rho(x,y)<r\}, \quad B^\rho(r):=B^\rho(0,r).
\end{equation*}

\par Let $\tilde{A}_r(f)$ be the non-isotropic Hardy-Littlewood averaging operator
\begin{equation*}
	\tilde{A}_r(f)(x)=\frac1{|B^\rho(r)|}\int_{B^\rho(r)\subset\R^N}f(x-y)dy, \qquad r>0.
\end{equation*}
For any cube $I_1\times\cdots\times I_N$ in $\R^N$ centered at $0$, define
\begin{equation*}
	A_{I_1,\cdots,I_N}(f)(x)=\frac{1}{|I_1\times\cdots\times I_N|}\int_{I_1\times\cdots\times I_N}f(x-y)dy.
\end{equation*}
As in Theorem \ref{thm: H-L maximal nonisotropic} and Theorem \ref{thm:strong maximal boundedness}, we have
\begin{theorem}
	The non-isotropic maximal operator $f\mapsto\{\tilde{A}_r(f)\}_{r>0}$ is bounded from $L^p(\A^N)$ to $L^p(\A^N;\ell^\infty(\R_+))$ for $1<p<\infty$ and also $L^1(\A^N)$ to $\Lambda^{1,\infty}(\A^N;\ell^\infty(\R_+))$.
	The strong maximal operator $f\mapsto \{A_{I_1,\cdots,I_N}(f)\}_{I_1,\cdots,I_N}$ is bounded from $L^p(\A^N)$ to $L^p(\A^N;\ell^\infty(\R_+^N))$ for $1<p<\infty$.
\end{theorem}
The $L^p$ boundedness and weak type $(1,1)$ assertion for the non-isotropic family follow from \cite[Theorem 4.1]{HongLiaoWang2021}, since $(\R^N,\rho,dx)$ has Radon measure and $|B^\rho(x,r)|=2^N r^\Delta$, hence doubling constant $2^\Delta$, where $\Delta$ is the homogeneous dimension defined below.

\par As in Section \ref{sec:parabola-model}, we now take $d\mu$ to be the measure on $\R^N$ defined by
\begin{equation*}
	\int_{\R^N}f d\mu=\int_{\R^k}f(\mfp(t))\eta(t)dt
\end{equation*}
and similarly define $d\mu_j$ by
\begin{equation*}
	\int_{\R^N}fd\mu_j=\int_{\R^k}f(2^{-j}\circ\mfp(t))\eta(t)dt=\int_{\R^k}f(\mfp(2^{-j}t))\eta(t)dt.
\end{equation*}
Then we have
\begin{equation*}
	A_j^\mfp f=f*d\mu_j,\quad \widehat{d\mu_j}(\xi)=\widehat{d\mu}(2^{-j}\circ \xi).
\end{equation*}
For each complex $s$, we denote by $\nu^s$ the distribution on $\R^N$ with Fourier transform
\begin{equation*}
	\widehat{\nu^s}(\xi)=\widehat{d\mu}(\xi)(1+|\xi|^2)^{s/2}.
\end{equation*}
and then define
\begin{equation*}
	A_j^{\mfp,s}(f)=f*\nu^s_j, \quad (\widehat{\nu^s_j})(\xi)=\widehat{\nu^s}(2^{-j}\circ\xi).
\end{equation*}
Note that the metric $\rho$, the measure $d\mu_j$ and the distribution $\nu^s_j$ all depend on the monomial map $t\mapsto\mfp(t)$. However, for simplicity, we omit the subscript $\mfp$ when no confusion is possible.
We first show how to prove the $L^2$ boundedness of $\{A_j^{\mfp,s}\}$  for $\Re(s)<1/d$.

\medskip

\subsubsection{The $L^2$-boundedness of $\{A_j^{\mfp,s}\}$  for $\Re(s)<1/d$}
\label{subsubsec:polynomial-l2}
\par Fix a smooth nonnegative function $\psi\in C_0^{\infty}(\R^N)$ such that
\begin{equation*}
	\int_{\R^N}\psi(x)dx=\int_{\R^k}\eta(t)dt,
\end{equation*}
and set $\psi_j(x)=2^{\Delta j}\psi(2^j\circ x)$, where $\Delta$ denotes the ``homogeneous dimension" of $\R^N$, i.e.
\begin{equation*}
	\Delta=\sum_{1\leq |\alpha|\leq d} |\alpha|.
\end{equation*}
Define $B_j(f)=f*\psi_j$ and the square function
\begin{equation*}
	S(f)(x)=\bigg(\sum_{j\in\Z}|A_j^{\mfp,s}f(x)-B_j(f)(x)|^2\bigg)^{1/2}.
\end{equation*}
By an argument analogous to (\ref{eq:sup A_j<sup B_j +S(f)}) in Section \ref{sec:l2-boundedness}, we have
\begin{equation*}
	\Big\|{\sup_{j\in\Z} A_j^{\mfp,s}(f)}\Big\|_{L^2(\A^N)}\leq \Big\|{\sup_{j\in\Z} B_j(f)}\Big\|_{L^2(\A^N)}+\Norm{S(f)}_{L^2(\A^N)}.
\end{equation*}
Therefore, it suffices to estimate the $L^2$-norm of $S(f)$:
\begin{equation*}
	\Norm{S(f)}_{L^2(\A^N)} \lesssim \Norm{f}_{L^2(\A^N)}.
\end{equation*}
As in Section \ref{sec:l2-boundedness}, this follows from the kernel estimate:
for each real $s<1/d$, we have
\begin{equation*}
	\sum_{j\in\Z}\big|{\widehat{\nu^s}(2^{-j}\circ\xi)-\widehat{\psi}(2^{-j}\circ\xi)}\big|^2\leq \text{Const}, \quad \text{uniformly in } \xi.
\end{equation*}
This is a consequence of the kernel estimate
\begin{equation*}
	\begin{cases}
		|\widehat{\nu^s}(\xi)-\widehat{\psi}(\xi)|\leq A_s \rho(\xi), & \text{ if } \rho(\xi) \leq 1; \\
		|\widehat{\nu^s}(\xi)-\widehat{\psi}(\xi)|\leq A_s \rho(\xi)^{-1/d+s}, &\text{ if } \rho(\xi) \geq 1.
	\end{cases}
\end{equation*}
Indeed, as mentioned in \cite{Stein1993}, one can verify that the submanifold $t\mapsto\mfp(t)$ is of finite type at each point, and is of type at most $d$.
Thus by Lemma \ref{lem:decay of Fourier transform of measure}, we find that $|\widehat{d\mu}(\xi)|$ has decay $\rho(\xi)^{-1/d}$ for large $\rho(\xi)$. This implies the estimates above.
This completes the proof for $s\in\R$ and $s<1/d$.
\par For general complex $s$ with $\Re(s)<1/d$, it suffices to split the kernel into real and imaginary parts and then estimate the two parts separately.
As in Remark \ref{rem:growth of C(s)}, the constant $C(s)$ in
\begin{equation*}
	\Big\|{\sup_j A_j^{\mfp,s}(f)}\Big\|_{L^2(\A^N)} \leq C(s)\Norm{f}_{L^2(\A^N)},
\end{equation*}
satisfies the estimate
\begin{equation*}
	C(\sigma+it)\lesssim_{\sigma} 1+|t|, \quad \text{ as } |t|\to\infty.
\end{equation*}

\medskip

\subsubsection{The $L^p$ boundedness of $\{A_j^{\mfp,s}\}$ for $\Re(s)<0$ and $1<p<\infty$}
\label{subsubsec:polynomial-lp}
\par We first prove the weak type $(1,1)$ estimate for $\{A_j^{\mfp,s}\}$ for $s\in\R$ and $s<0$.
\par To this end, we introduce the basic notation for the non-isotropic Calder\'on-Zygmund decomposition, adapted to the submanifold $t\mapsto \mfp(t)$ given by monomials. The construction follows that in Section \ref{subsec:cz-decomposition}.
For $k\in\Z$, let $\Q_k$ be the set of non-isotropic dyadic cubes in $\R^N$ of the form
\begin{equation*}
	\Q_k:=\bigg\{\prod_{1\leq|\alpha|\leq d}[a_\alpha2^{-|\alpha|k},(a_\alpha+1)2^{-|\alpha|k}]:\text{ each }a_\alpha\in\Z\bigg\}, \quad k\in\Z.
\end{equation*}
Let $\Q=\bigcup_{k\in\Z}\Q_k$. For any rectangle $I=:\prod_{1\leq|\alpha|\leq d}[x_\alpha,y_\alpha]$ in $\R^N$, we call $I$ a non-isotropic cube with respect to the metric $\rho$ if $|y_\alpha-x_\alpha|^{1/|\alpha|}$ is the same for all $\alpha$ with $1\leq|\alpha|\leq d$.
In this case, the non-isotropic side length of $I$ is
\begin{equation*}
	\ell(I)=\rho((y_\alpha)_{\alpha}-(x_\alpha)_{\alpha})=|y_\alpha-x_\alpha|^{1/|\alpha|} \quad \text{ for any } 1\leq |\alpha| \leq d.
\end{equation*}
The objects introduced in Section \ref{subsec:cz-decomposition}---the $\sigma$-algebras $\sigma_k$, von Neumann algebras $\A_k$, conditional expectations $E_k$, martingale differences $df_k$, Cuculescu projections $q_k$, and the noncommutative Calder\'on-Zygmund decomposition---are defined in the same way.
For convenience, we retain the same notation. By Appendix \ref{app:Bessel kernels} in dimension $n=N$, Lemma \ref{prop:K(x) kernel preoperties} also holds for $K(x):=\nu^s(x)-\psi(x)$, the kernel of $A_0^{\mfp,s}-B_0$, with $\delta=1/d$ in (\ref{item:kernel-fourier-decay}). The subtraction of $\psi$ ensures $\widehat K(0)=0$.
Therefore, the estimates for the bad and good parts in Section \ref{sec:weak-11-boundedness} remain valid. Thus, we can prove the weak type $(1,1)$ estimate for $\{A_j^{\mfp,s}\}_j$ in a similar way.
\begin{prop}
	Suppose $s<0$. For any $\lambda>0$ and $f\in\A^N_{c,+}$ there exists a projection $e\in\A^N$ such that
	\begin{equation*}
		eA_j^{\mfp,s}(f)e\leq\lambda \text{ for all } j\in\Z, \quad \nu(1-e)\lesssim  \lambda^{-1}\Norm{f}_{L^1(\A^N)}.
	\end{equation*}
	In other words, the map $f\mapsto\{A_j^{\mfp,s}(f)\}_j$ is bounded from $L^1(\A^N)$ to $\Lambda^{1,\infty}(\A^N;\ell^\infty(\Z))$.
\end{prop}
\par The noncommutative Marcinkiewicz interpolation theorem shows that, for $1<p<\infty$, the estimate
\begin{equation}
	\label{eq:sup_j A_j^(p,s) Lp bounded}
	\Big\|{\sup_j A_j^{\mfp,s}f}\Big\|_{L^p(\A^N)}\leq C_p(s)\Norm{f}_{L^p(\A^N)}
\end{equation}
holds for real $s<0$. An argument similar to that in Section \ref{sec:lp-boundedness} then shows that (\ref{eq:sup_j A_j^(p,s) Lp bounded}) holds for general complex $s$ with $\Re(s)<0$.
(Recall that Lemma \ref{lem:s'+s''} still holds. The properties in Lemma \ref{lem:Bessel kernel properties} hold in dimension $n=N$, and the same argument using the strong maximal function over rectangles in $\R^N$ extends Lemma \ref{lem:sup_j (G_s)_j*h_j Lp-bounded} to this setting.)
As in Remark \ref{rem:growth of C_p(s)}, the constant satisfies
\begin{equation*}
	C_p(\sigma+it)\lesssim_{\sigma} e^{|t|}, \quad  \text{ as } |t|\to\infty.
\end{equation*}

\medskip

\par Finally, as in Section \ref{sec:lp-boundedness}, the complex interpolation yields the following result.
\begin{prop}
	\label{prop:sup A_j^mfp Lp bounded}
	For each $1<p\leq\infty$, we have
	\begin{equation*}
		\Big\|{\sup_{j\in\Z}A_j^{\mfp}(f)}\Big\|_{L^p(\A^N)}\lesssim_p \Norm{f}_{L^p(\A^N)}.
	\end{equation*}
\end{prop}
This proves Proposition \ref{prop:sup_r M_r^mfp Lp bounded}, i.e. the monomial case.

\medskip

\par To prove Theorem \ref{thm:polynomial case}, we use descent to transfer these bounds for operators on $\R^N$ to the corresponding operators on $\R^n$.
\par We begin by considering a fixed linear mapping $L:\R^N\to\R^n$. For this fixed $L$, suppose $T$ is a translation-invariant operator acting on functions on $\R^N$. This naturally induces an operator $T^L$, acting on functions on $\R^n$.
Although the following lemma holds for all translation-invariant operators, for simplicity, we limit ourselves to the situation
\begin{equation}
	\label{eq: T(f) def}
	T(f)(y)=f*d\mu(y)=\int_{\R^N}f(y-z)d\mu(z), \quad y\in\R^N,
\end{equation}
where $d\mu$ is a finite measure on $\R^N$ with compact support, and $f:\R^N\to\M$ is an $\M$-valued function. The operator $T^L$ is then defined by
\begin{equation*}
	T^L(f)(x)=\int_{\R^N} f(x-L(z))d\mu(z), \quad x\in\R^n.
\end{equation*}
Note that the operator $T^L$ is also a convolution operator with a measure, i.e. $T^L(f)=f*d\mu^L$, where $d\mu^L$ is the measure on $\R^n$ defined by
\begin{equation*}
	\int_{\R^n} \vphi(x) d\mu^L(x):=\int_{\R^N} \vphi(L(z)) d\mu(z).
\end{equation*}
\par Now let $\{T_j\}_{j\in\Z}$ be a family of operators acting on functions on $\R^N$, given by
\begin{equation}
	\label{eq:T_j(f) def}
	T_j(f)=f*d\mu_j,
\end{equation}
where each measure $d\mu_j$ is a finite measure on $\R^N$ with compact support. The corresponding induced operator is $T_j^L(f)=f*d\mu_j^L$, where each $d\mu_j^L$ is a measure on $\R^n$ defined as above.
\par Recall that we use the notation
\begin{equation*}
	\A^n=L^\infty(\R^n)\ot\M \quad \text{ and } \quad \A^N=L^\infty(\R^N)\ot\M.
\end{equation*}
\begin{lemma}
	\label{lem:T bound implies T^L bound}
	Let $1\leq p\leq\infty$. Suppose $L:\R^N\to\R^n$ is a fixed linear mapping as above, and let $T$ and $T_j$ be as in (\ref{eq: T(f) def}) and (\ref{eq:T_j(f) def}). Then
	\begin{enumerate}
		\item[(a)] If $T$ is bounded on $L^p(\A^N)$, then $T^L$ is bounded on $L^p(\A^n)$.
		\item[(b)] Assume additionally that the measures $d\mu_j$ are nonnegative. If the operator $f\mapsto\{T_j(f)\}_j$ is bounded from $L^p(\A^N)\to L^p(\A^N;\ell^\infty(\Z))$, then the operator $f\mapsto\{T_j^L(f)\}_j$ is bounded from $L^p(\A^n)\to L^p(\A^n;\ell^\infty(\Z))$.
	\end{enumerate}
\end{lemma}
\begin{proof}[Proof of Lemma \ref{lem:T bound implies T^L bound}]
	Part (a) follows as in the commutative setting (see \cite[Chapter XI, Section 2.4]{Stein1993}), so we prove only part (b). The case $p=\infty$ is immediate; we therefore consider $p<\infty$. Without loss of generality, we also assume that $f$ takes values in $\mathcal{S}_\M^+$.
	\par Fix a finite subset $J\subset\Z$. By assumption, each $d\mu_j$ is a positive finite measure with compact support. We assume the measures $\{d\mu_j\}_{j\in J}$ are supported in a common compact set $\{z:|z|\leq M\}$.
	For $u\in\R^N$, let $\tau_u$ denote the translation operator acting on functions on $\R^n$ by
	\begin{equation*}
		(\tau_uf)(x)=f(x+L(u)).
	\end{equation*}
	Then for large $R>0$,
	\begin{equation}
		\label{eq:T_j^L average on |u|<R}
		\Big\|{\sup_{j\in J} T_j^Lf}\Big\|_{L^p(\A^n)}^p=\frac{1}{v_NR^N}\int_{|u|\leq R}\Big\|{\sup_{j\in J}\tau_u T_j^Lf}\Big\|_{L^p(\A^n)}^p du,
	\end{equation}
	where $v_N$ is the volume of the unit ball in $\R^N$. Note that, for $|u|\leq R$,
	\begin{equation*}
		\begin{split}
			\tau_uT_j^Lf(x)&=\int_{\R^N} f(x+L(u-z))d\mu_j(z) \\
&=\int_{\R^N} f(x+L(u-z))\chi_{B_N(R+M)}(z-u)d\mu_j(z) \\
			&=\int_{\R^N} F_x(u-z) d\mu_j(z) 
			:=T_j(F_x)(u),
		\end{split}
	\end{equation*}
	where $B_N(R+M)$ is the ball of radius $R+M$ in $\R^N$, and
	\begin{equation*}
		F_x(z)=f(x+L(z))\chi_{B_N(R+M)}(z)
	\end{equation*}
	is an $\mathcal{S}_\M^+$-valued function on $\R^N$, for each $x\in\R^n$. Next we claim that
	\begin{equation}
		\label{eq:interange integral for supremum func}
		\int_{|u|\leq R}\Big\|{\sup_j (T_jF_{(\cdot)})(u)}\Big\|_{L^p(\A^n)}^p du = \int_{\R^n} \Big\|{\sup_j T_jF_x(\cdot)}\Big\|_{L^p(L^\infty(B^N(R))\ot\M)}^p dx.
	\end{equation}
	Here the notation $T_jF_{(\cdot)}(u)$ is understood as an $\mathcal{S}_\M^+$-valued function on $\R^n$, $x\mapsto T_jF_x(u)$, for each fixed $u$. The notation $T_jF_x(\cdot)$ is understood similarly.
	We prove this identity later. Assuming it for the moment and combining it with (\ref{eq:T_j^L average on |u|<R}), we have
	\begin{equation*}
		\begin{split}
			\Big\|{\sup_{j\in J} T_j^Lf}\Big\|_{L^p(\A^n)}^p
			&=\frac{1}{v_NR^N} \int_{|u|\leq R}\Big\|{\sup_{j\in J} (T_jF_{(\cdot)})(u)}\Big\|_{L^p(\A^n)}^p du \\
			&=\frac{1}{v_NR^N} \int_{\R^n} \Big\|{\sup_{j\in J} T_jF_x(\cdot)}\Big\|_{L^p(L^\infty(B^N(R))\ot\M)}^p dx \\
			&\leq \frac{1}{v_NR^N} \int_{\R^n} \Big\|{\sup_{j\in J} T_jF_x}\Big\|_{L^p(\A^N)}^p dx  \lesssim_p \frac{1}{v_NR^N} \int_{\R^n} \Norm{F_x}_{L^p(\A^N)}^p dx \\
			&=\frac{1}{v_NR^N} \int_{\R^n}\int_{|u|\leq R+M}\Norm{f(x+L(u))}_{L^p(\M)}^p dudx \\
			&=\Big(\frac{R+M}{R}\Big)^N \Norm{f}_{L^p(\A^n)}^p,
		\end{split}
	\end{equation*}
	where the fourth line is due to the assumption that the maximal operator $f\mapsto\{T_jf\}$ is $L^p(\A^N)\to L^p(\A^N;\ell^\infty)$ bounded, and in the last line we use Fubini's theorem.
	Finally letting $R\to\infty$, we obtain
	\begin{equation*}
		\Big\|{\sup_{j\in J}T_j^Lf}\Big\|_{L^p(\A^n)} \lesssim \Norm{f}_{L^p(\A^n)}.
	\end{equation*}
	Since the finite set $J\subset\Z$ is arbitrary, we deduce the desired result.
\end{proof}
Now we turn to the proof of (\ref{eq:interange integral for supremum func}). For clarity, we consider a more general setting.
Let $X$ and $Y$ be two $\sigma$-finite measure spaces. For any countable index set $J$, let
\begin{equation*}
	\vphi_j: X\times Y\to \mathcal{S}_\M^+, \quad j\in J
\end{equation*}
be a family of self-adjoint integrable operator-valued functions. The following lemma holds.
\begin{lemma}
	For $X$, $Y$, $\{\varphi_j\}_{j\in J}$ as above, we have
	\begin{equation}
		\label{eq:interange integral for supremum func2}
		\int_X \Big\|{\sup_j \vphi_j(x,\cdot)}\Big\|_{L^p(L^\infty(Y)\ot\M)}^p dx =
		\int_Y \Big\|{\sup_j \vphi_j(\cdot,y)}\Big\|_{L^p(L^\infty(X)\ot\M)}^p dy.
	\end{equation}
\end{lemma}
\begin{proof}
	Let $1\leq p<\infty$. If some $\varphi_j$ does not belong to $L^p(X\times Y;L^p(\M))$, both sides of (\ref{eq:interange integral for supremum func2}) are infinite. We may therefore assume that every $\varphi_j$ belongs to this space.
	\par For a $\sigma$-finite measure space $Z$, put $\mathcal B_Z=L^\infty(Z)\ot\M$. We first show that, for every finite index set $F$ and positive functions $u_j\in L^p(Z;L^p(\M))$, $j\in F$,
	\begin{equation*}
		\Norm{\{u_j\}_{j\in F}}_{L^p(\mathcal B_Z;\ell^\infty(F))}^p
		=\int_Z\Norm{\{u_j(z)\}_{j\in F}}_{L^p(\M;\ell^\infty(F))}^p dz\,.
	\end{equation*}
	For common simple functions $u_j=\sum_{k=1}^m\mathbf1_{E_k}a_{j,k}$, where the sets $E_k$ are disjoint and have finite measure, the positive-majorant characterization gives both sides as
	\begin{equation*}
		\sum_{k=1}^m|E_k|\,\Norm{\{a_{j,k}\}_{j\in F}}_{L^p(\M;\ell^\infty(F))}^p\,.
	\end{equation*}
	Indeed, every positive majorant bounds the family $\{a_{j,k}\}_{j\in F}$ almost everywhere on $E_k$, while almost optimal constant majorants may be chosen independently on these finitely many sets. For general $u_j$, use positive simple approximations on common finite partitions. The estimate
	\begin{equation*}
		\Norm{\{a_j-b_j\}_{j\in F}}_{L^p(\M;\ell^\infty(F))}
		\leq\sum_{j\in F}\Norm{a_j-b_j}_{L^p(\M)}
	\end{equation*}
	and its version with $\M$ replaced by $\mathcal B_Z$ allow passage to the limit on both sides. The same estimate shows that the finite-family norm is continuous on $L^p(\M)^F$, so the integrand is measurable.
	\par Applying this identity in $Y$ and in $X$, and then using Fubini's theorem, shows that both sides of (\ref{eq:interange integral for supremum func2}), with $j$ restricted to a finite set $F\subset J$, equal
	\begin{equation*}
		\int_{X\times Y}\Norm{\{\varphi_j(x,y)\}_{j\in F}}_{L^p(\M;\ell^\infty(F))}^p dxdy\,.
	\end{equation*}
	Finally, choose finite sets $F_n\uparrow J$. The finite-subfamily characterization of the maximal norm (\cite[Proposition 2.1(ii)]{JungeXu2007}) and the monotone convergence theorem give (\ref{eq:interange integral for supremum func2}) for countable $J$.
\end{proof}
To obtain (\ref{eq:interange integral for supremum func}), it suffices to apply this lemma to $X=\R^n$ and $Y=B^N(R)$. Thus we finish the proof of (\ref{eq:interange integral for supremum func}).

\medskip

\par Lemma \ref{lem:T bound implies T^L bound} allows us to obtain the polynomial case from the (universal) monomial case.
Indeed, we let $T_j=A_j^\mfp$, defined in (\ref{eq:A_j^mfp def}). Then
\begin{equation*}
	T_j(f)=A_j^\mfp(f)=f*d\mu_j.
\end{equation*}
We define the linear operator $L$ as follows. Assume that the polynomial $P_j$ vanishes at the origin (this causes no loss of generality, since the constant terms only produce a translation). We write
\begin{equation*}
	P_j(t)=\sum_{1\leq|\alpha|\leq d}a_{j\alpha}t^\alpha, \quad j=1,\dots,n.
\end{equation*}
Use the coefficients $a_{j\alpha}$ to define $L:\R^N\to\R^n$,
\begin{equation*}
	L(x)_j=\sum_{1\leq|\alpha|\leq d} a_{j\alpha}x_\alpha, \quad j=1,\dots,n.
\end{equation*}
where $x=(x_\alpha)\in\R^N$ and $L(x)=(L(x)_j)\in\R^n$. Then $P(t)=L(\mfp(t))$. By definition, the measure $d\mu_j^L$ is
\begin{equation*}
	\int_{\R^n}f(x) d\mu_j^L(x)=\int_{\R^N}f\big(L(z)\big)d\mu_j(z)=\int_{\R^k}f(L(\mfp(2^{-j}t)))\eta(t)dt,
\end{equation*}
and the operator $T^L_j$ is exactly
\begin{equation*}
	T_j^Lf(x)=f*d\mu_j^L(x)=\int_{\R^k}f(x-L(\mfp(2^{-j}t)))\eta(t)dt=2^{kj}\int_{\R^k}f(x-P(t))\eta(2^jt)dt.
\end{equation*}
By Proposition \ref{prop:sup A_j^mfp Lp bounded} and Lemma \ref{lem:T bound implies T^L bound}, for $1<p\leq\infty$ we have
\begin{equation*}
	\Big\|{\sup_{j\in\Z}T_j^L(f)}\Big\|_{L^p(\A^n)} \lesssim_p \Norm{f}_{L^p(\A^n)}.
\end{equation*}
Finally, since the two maximal norms $\big\|{\sup_j T_j^L(f)}\big\|_p$ and $\big\|{\sup_r M_r^P(f)}\big\|_p$ are comparable when $f$ is positive, we obtain (\ref{eq:sup_r>0 M_r^P Lp bounded}). This completes the proof of Theorem \ref{thm:polynomial case}.

\subsection{A generalized result}
\label{subsec:generalized-averages}
\par We have already proved the boundedness of maximal averages over submanifolds given by polynomials.
To pass from polynomial maps to general submanifolds, we extend the results of the previous subsection to the following setting.
\par Let $\{dm^j\}_j$ be a sequence of positive measures on $\R^N$ satisfying
\begin{enumerate}
	\item\label{item:measure-common-mass} $\int_{\R^N}dm^j=c$ for all $j$;
	\item\label{item:measure-common-support} $\{dm^j\}$ has a common compact support;
	\item\label{item:measure-fourier-decay} $\exists$ $\delta>0$ and $A>0$, such that $|\widehat{dm^j}(\xi)|\leq A|\xi|^{-\delta}$ uniformly in $j$.
\end{enumerate}
\par Using the non-isotropic dilations $x\mapsto\delta\circ x$ on $\R^N$ given in (\ref{eq:dilation in R^N}), define the dilated measure $d\mu_j$ by
\begin{equation*}
	\int_{\R^N}f(x)d\mu_j(x)=\int_{\R^N}f(2^{-j}\circ x)dm^j(x),
\end{equation*}
which means that $\widehat{d\mu_j}(\xi)=\widehat{dm^j}(2^{-j}\circ \xi)$. Now let
\begin{equation*}
	A_j(f)=f*d\mu_j
\end{equation*}
Then we have the following $L^p$-boundedness:
\begin{prop}
	\label{prop:generalised sup_j A_jf Lp estimate}
	Let $1<p\leq\infty$. Under the above assumptions (\ref{item:measure-common-mass})-(\ref{item:measure-fourier-decay}) on the measures $\{dm^j\}_j$, we have
	\begin{equation*}
		\Big\|{\sup_{j\in\Z}A_jf}\Big\|_{L^p(\A^N)}\lesssim_p \Norm{f}_{L^p(\A^N)}.
	\end{equation*}
\end{prop}
The idea of the proof is essentially the same as that in Sections \ref{sec:parabola-model}--\ref{sec:lp-boundedness}. Let $M_j^s$ be the distribution defined by
\begin{equation*}
	M_j^s=dm^j*G_s,
\end{equation*}
where $\widehat{G_s}(\xi)=(1+|\xi|^2)^{s/2}$ on $\R^N$, as in Appendix \ref{app:Bessel kernels}. Also, define its dilation $\nu_j^s$ by
\begin{equation*}
	\widehat{\nu^s_j}(\xi)=\widehat{M_j^s}(2^{-j}\circ\xi)
\end{equation*}
and, using the same notation as above, write $A_j^s(f)=f*\nu_j^s$.
\par For $\Re(s)<\delta$, we have the $L^2$-boundedness
\begin{equation*}
	\Big\|{\sup_{j\in\Z} A_j^s(f)}\Big\|_{L^2(\A^N)}\lesssim_s \Norm{f}_{L^2(\A^N)}.
\end{equation*}
The proof follows from a square function estimate. In fact, fix a smooth function $\psi\geq0$ with compact support in $\R^N$ such that $\int \psi(x)dx=c$, so that $\psi$ has the same normalization as each $dm^j$.
Let $\psi_j(x)=2^{j\Delta}\psi(2^j\circ x)$ and set $B_j(f)=f*\psi_j$. As we have seen, the crucial point is to show the $L^2(\A^N)$-boundedness of the square function
\begin{equation*}
	S_s(f)(x)=\bigg(\sum_{j\in\Z}|A_j^sf(x)-B_jf(x)|^2\bigg)^{1/2}, \quad s\in\R \text{ and } s<\delta.
\end{equation*}
Indeed, by Plancherel's theorem this is equivalent to
\begin{equation*}
	\sum_{j\in\Z}|\widehat{M_j^s}(2^{-j}\circ\xi)-\widehat{\psi}(2^{-j}\circ\xi)|^2\leq A_s
\end{equation*}
uniformly in $\xi$, for each $s<\delta$. As usual, this is guaranteed by the moderate decay of $|\widehat{M_j^s}(\xi)-\widehat{\psi}(\xi)|$ as both $\rho(\xi)\to0$ and $\rho(\xi)\to\infty$.
The proof is essentially the same as in the parabola case. See \cite{Stein1993} for the details.
\par For $\Re(s)<0$ we have the $L^p$-boundedness
\begin{equation*}
	\Big\|{\sup_{j\in\Z}A_j^s(f)}\Big\|_{L^p(\A^N)}\lesssim_{p,s} \Norm{f}_{L^p(\A^N)}.
\end{equation*}
This is deduced from the weak type $(1,1)$ estimate: Suppose $s<0$.
For every positive $f\in L^1(\A^N)$ and every $\lambda>0$, there exists a projection $e\in\A^N$ such that
\begin{equation*}
	\Norm{eA_j^s(f)e}_\infty \leq \lambda \text{ for all } j\in\Z, \quad \nu(1-e)\lesssim \lambda^{-1}\Norm{f}_{L^1(\A^N)}.
\end{equation*}
The proof of the weak type $(1,1)$ estimate follows Section \ref{sec:weak-11-boundedness} with a few modifications.
Indeed, as in Section \ref{sec:weak-11-boundedness}, we consider the operator and its kernel
\begin{equation*}
	T_j=A_j^s-B_j, \quad \text{ with kernel } K_j(x)=\nu_j^s(x)-\psi_j(x).
\end{equation*}
If we let $\K_{j}(x):=M_j^s(x)-\psi(x)$, then $K_j(x)=2^{j\Delta}\K_{j}(2^j\circ x)$.
\par It will suffice to prove the weak type $(1,1)$ estimate for the operator $f\mapsto\{T_j(f)\}_j$.
We need the following kernel properties:
\begin{lemma}
	\label{prop:M_j^s kernel properties}
	Fix $s<0$. The family of kernels $\{\K_j\}_j$ satisfies:
	\begin{enumerate}
		\item\label{item:general-kernel-local-integrability} Each $\K_j$ is a locally integrable function on $\R^N$ that is smooth for $|x|$ large.
		\item\label{item:general-kernel-translation} There exists $\varepsilon>0$ such that
		      \begin{equation*}
				\int_{\R^N}|\K_j(x-y)-\K_j(x)|dx \lesssim \rho(y)^\varepsilon, \quad \text{ uniformly in } j\in\Z.
			  \end{equation*}
		\item\label{item:general-kernel-decay} For arbitrary $M>0$ we have
		      \begin{equation*}
				\max\{|\K_j(x)|,\rho(\nabla \K_j(x))\} \lesssim \rho(x)^{-M}, \text{ for large } x.
			  \end{equation*}
			  In particular, the estimate above holds uniformly in $j\in\Z$.
		\item\label{item:general-kernel-fourier-decay} The Fourier transforms $\widehat{\K_j}$ are $L^\infty$ functions satisfying
		      \begin{equation*}
			    |\widehat{\K_j}(\xi)|\lesssim
			    \begin{cases}
				    \rho(\xi), &\text{ as } \xi \text{ near } 0, \\
				    \rho(\xi)^{-\delta}, &\text{ as } \xi \text{ near } \infty, \text{ for some } \delta>0.
			    \end{cases}
		      \end{equation*}
			  Note the estimate above holds uniformly in $j\in\Z$.
	\end{enumerate}
\end{lemma}
Comparing this lemma with Lemma \ref{prop:K(x) kernel preoperties}, we find that the family of kernels $\{\K_j\}_j$ plays the role of the single kernel $K(x)=K_0(x)$ there.
Indeed, the model case in Sections \ref{sec:l2-boundedness} and \ref{sec:weak-11-boundedness} corresponds here to the situation where there is only one measure, i.e., the measure $dm^j$ is independent of $j$.
For (\ref{item:general-kernel-local-integrability})--(\ref{item:general-kernel-decay}), apply Appendix \ref{app:Bessel kernels} to $\K_j=G_s*dm^j-\psi$: the common mass and compact support give uniform integrability, smoothness away from the common support, and the bounds (\ref{eq:Bessel translation estimate}) and (\ref{eq:Bessel kernel decay}) uniformly in $j$. The Euclidean decay of arbitrary order implies (\ref{item:general-kernel-decay}) for the quasi-norm $\rho$ by taking a sufficiently large decay exponent. For (\ref{item:general-kernel-fourier-decay}), use $\widehat{\K_j}(0)=0$, the common compact support and the assumed uniform Fourier decay of $dm^j$.

\par We indicate the main changes needed to adapt the weak type $(1,1)$ argument in Section \ref{sec:weak-11-boundedness} to this family of kernels.
\par First, in place of Claim \ref{claim:sum int |K_j(x-y)-K_j(x)| estimate}, we need to prove
\begin{equation*}
	\sup_{y\in\R^N} \sum_{j\in\Z} \int_{\rho(x)\geq C\rho(y)}|K_j(x-2^{-t}\circ y)-K_j(x)|dx\lesssim 2^{-\alpha t}, \quad \text{ for some } \alpha>0.
\end{equation*}
(To avoid a conflict of notation, we use $t$ instead of $s$ to denote the nonnegative integer here.)
Here $C>1$ is fixed sufficiently large in terms of the common compact support, and the fixed dilation in Lemma \ref{lem:zeta projection} is enlarged accordingly if necessary. On the integration region, the metric triangle inequality gives $\rho(x-\theta(2^{-t}\circ y))\geq(1-C^{-1})\rho(x)$ for $0\leq\theta\leq1$.
For this we need only to use (\ref{item:general-kernel-translation}) and (\ref{item:general-kernel-decay}) (for the $\rho(\nabla \K_j)$ part) in Lemma \ref{prop:M_j^s kernel properties} instead of the corresponding properties in Lemma \ref{prop:K(x) kernel preoperties}. The large-radius argument uses $|2^{-t}\circ z|\leq\sqrt{N}\,2^{-t}\rho(z)^d$ for $\rho(z)\geq1$, with $M>\Delta+d$ in the summation.
\par Second, as in the beginning of the proof of Lemma \ref{lem:K(1-phi) estimate}, we need the estimate, for sufficiently large $t$,
\begin{equation*}
	\nu_j^s(x)(1-\phi_{j-\varepsilon t}(x))\lesssim 2^{-j(M-\Delta)}\rho(x)^{-M}\chi_{\{\rho(x)\geq 2^{-j+\varepsilon t-1}\}}(x).
\end{equation*}
To deal with this, we use the bound for $|\K_j|$ in (\ref{item:general-kernel-decay}) in Lemma \ref{prop:M_j^s kernel properties}; for sufficiently large $t$, the truncation region is disjoint from the support of $\psi_j$, so $K_j=\nu_j^s$ there.
\par Finally, in the proof of Lemma \ref{lem:K(phi) estimate}, the part ``Estimate of (II)", we need the Fourier transform estimate (\ref{item:general-kernel-fourier-decay}) in Lemma \ref{prop:M_j^s kernel properties}, which is uniform in $j$, instead of the corresponding estimate in Lemma \ref{prop:K(x) kernel preoperties}.
\par The weak type $(1,1)$ estimate yields $L^p$-boundedness for $s<0$, and this extends to all $s\in\mathbb{C}$ with $\Re(s)<0$ similarly. Using the interpolation method of Section \ref{sec:lp-boundedness}, we have proved Proposition \ref{prop:generalised sup_j A_jf Lp estimate}.

\subsection{Averages over submanifolds of finite type}
\label{subsec:finite-type-submanifolds}
\par We now complete the proof of Theorem \ref{thm:k-submanifold case} for general submanifolds of finite type.
\par Suppose that $S$ is a $k$-dimensional submanifold in $\R^n$ that is of type $d$ at $t=0$, given by $t\mapsto\gamma(t)$. We may also assume that $\gamma(0)=0$ (after a translation). By Taylor expansion,
\begin{equation*}
	\gamma(t)=P(t)+R(t),
\end{equation*}
where $P$ is a polynomial (without constant term) of degree $d$, and $R$ vanishes to order $d+1$ at the origin.
The fact that $\gamma$ is of type $d$ at the origin implies that $P$ is also of type $d$, because this property is determined by the derivatives of order at most $d$ at the origin.
\par As in the previous subsection, we write $P=(P_1,\dots,P_n)$, where
\begin{equation*}
	P_j(t)=\sum_{1\leq|\alpha|\leq d} a_{j\alpha}t^\alpha,
\end{equation*}
and define the linear mapping $L:\R^N\to\R^n$ by
\begin{equation*}
	L(x)_j=\sum_{\alpha} a_{j\alpha}x_\alpha, \quad j=1,\dots,n.
\end{equation*}
Now the assumption that $P$ is of type $d$ implies that the polynomials
\begin{equation*}
	P_1(t),\dots, P_n(t)
\end{equation*}
are linearly independent (otherwise $S$ would lie in a hyperplane, which contradicts the finite type assumption).
As a result, the range of $L$ is the whole space $\R^n$, and there is a linear mapping $Q:\R^n\to\R^N$ so that $L \circ Q=\rm\Id_{\R^n}$.
\par Next, recall the notation $\mfp$, the free polynomial of degree $d$ mapping $\R^k$ to $\R^N$: $\mfp(t)=(t^\alpha)_{1\leq |\alpha|\leq d}$.
Then we have $P(t)=L(\mfp(t))$. Define
\begin{equation*}
	\mathfrak{r}(t):=Q(R(t)) \quad \text{ and } \quad \Gamma(t)=\mfp(t)+\mathfrak{r}(t).
\end{equation*}
Then we have
\begin{equation*}
	L(\Gamma)=L(\mfp)+L(\mathfrak{r})=P+R=\gamma.
\end{equation*}
In other words, $\Gamma(t)$ is a lift of $\gamma(t)$ to $\R^N$ through the linear map $L$.
Therefore, by the descent procedure described in the previous subsection, it will suffice to control the corresponding operator (acting on functions on $\R^N$) given by
\begin{equation*}
	f \mapsto \{M_r^\Gamma(f)\}_{0<r<1},
\end{equation*}
where
\begin{equation*}
	M_r^\Gamma(f)(x)=\frac1{|B_k(r)|}\int_{|t|<r}f(x-\Gamma(t))dt.
\end{equation*}
We need to prove the following lemma:
\begin{lemma}
	\label{lem:sup_r M_r^Gamma Lp boundedness}
	For $1<p\leq\infty$ we have
	\begin{equation*}
		\Big\|{\sup_{0<r<1} M_r^\Gamma(f)}\Big\|_{L^p(\A^N)} \lesssim_p \Norm{f}_{L^p(\A^N)}.
	\end{equation*}
\end{lemma}
\begin{proof}
To prove Lemma \ref{lem:sup_r M_r^Gamma Lp boundedness}, we first define the measure $d\mu_j$ on $\R^N$ by
\begin{equation*}
	\int_{\R^N}f(x)d\mu_j(x):=2^{jk}\int_{\R^k}f(\Gamma(t))\eta(2^jt)dt=\int_{\R^k}f(\Gamma(2^{-j}t))\eta(t)dt, \quad  \text{ for } j\geq j_0.
\end{equation*}
Here $j_0$ is a fixed positive integer which will be determined later.
Write $A_j(f)=f*d\mu_j$ for $j\geq j_0$. As usual, we may assume $f\in\A_{c,+}$. Then
\begin{equation*}
	\begin{split}
		\Big\|{\sup_{0<r<1} M_r^\Gamma(f)}\Big\|_{L^p(\A^N)}
		&\leq \Big\|{\sup_{0<r<2^{-j_0}} M_r^\Gamma(f)}\Big\|_{L^p(\A^N)}+\Big\|{\sup_{2^{-j_0}<r<1} M_r^\Gamma(f)}\Big\|_{L^p(\A^N)} \\
		&\lesssim \Big\|{\sup_{j\geq j_0} A_j(f)}\Big\|_{L^p(\A^N)}+\bigg\|{\int_{|t|\leq1}f(x-\Gamma(t))dt}\bigg\|_{L^p(\A^N)}.
	\end{split}
\end{equation*}
Note that, by Minkowski's inequality, the second term in the second line is bounded by a constant multiple of the $L^p$ norm of $f$.
To treat the first term $\Norm{\sup_{j\geq j_0} A_jf}_p$, we define the measures $dm^j$ by
\begin{equation*}
	\int_{\R^N}f(x)dm^j(x)=\int_{\R^k}f(2^j\circ\Gamma(2^{-j}t))\eta(t)dt,\qquad j\geq j_0.
\end{equation*}
We verify the uniform finite-type condition for these measures. Put $\Gamma_j(t)=2^j\circ\Gamma(2^{-j}t)$ and $K=\supp\eta$. Since $\mathfrak r$ vanishes to order $d+1$ at the origin, Taylor's theorem gives
\begin{equation*}
	\Norm{\Gamma_j-\mathfrak p}_{C^{d+1}(K_1)}\lesssim 2^{-j}
\end{equation*}
on a fixed compact neighborhood $K_1$ of $K$. For every $t\in K$ and unit vector $\zeta\in\R^N$, at least one derivative $\partial_t^\alpha(\mathfrak p(t)\cdot\zeta)$ with $1\leq|\alpha|\leq d$ is nonzero: otherwise the polynomial $\mathfrak p(\cdot)\cdot\zeta$ would be constant, forcing $\zeta=0$. By compactness and the displayed convergence, there are $c>0$ and $j_0$ such that
\begin{equation*}
	\max_{1\leq|\alpha|\leq d}
	\big|\partial_t^\alpha(\Gamma_j(t)\cdot\zeta)\big|\geq c,
	\qquad t\in K,\quad |\zeta|=1,\quad j\geq j_0.
\end{equation*}
This lower bound and the uniform $C^{d+1}$ bounds give the Fourier decay of $dm^j$ uniformly in $j\geq j_0$; see \cite[Chapter XI, Section 2.6]{Stein1993}. Moreover, these measures have common mass $\int\eta(t)dt$ and common compact support, since $\Gamma_j\to\mathfrak p$ uniformly on $K$. Thus they satisfy (\ref{item:measure-common-mass})--(\ref{item:measure-fourier-decay}) on $j\geq j_0$.
To apply Proposition \ref{prop:generalised sup_j A_jf Lp estimate} as stated, extend the family by setting $dm^j=dm^{j_0}$ for $j<j_0$. This extension preserves all three assumptions. Restricting its maximal bound to $j\geq j_0$ controls the first term above; together with the bound for the second term, this proves the lemma.
\end{proof}

\begin{proof}[Proof of Theorem \ref{thm:k-submanifold case}]
	By the descent method, we obtain Theorem \ref{thm:k-submanifold case} as a direct consequence of Lemma \ref{lem:sup_r M_r^Gamma Lp boundedness}.
\end{proof}

\section{Averages on variable hypersurfaces}
\label{sec:variable-hypersurfaces}
\par In this section, we study operator-valued averages over variable hypersurfaces. More precisely, for each $x\in\R^n$, we consider a family of hypersurfaces $S_{x,t}$, with parameter $t$ in the range $0<t\leq1$, such that, as $t\to0$, the hypersurfaces $S_{x,t}$ shrink to $x$.
We then define the averages for an $\M$-valued function $f$:
\begin{equation*}
	(A_tf)(x)=\int_{S_{x,t}}f(y)d\sigma_{x,t}(y),
\end{equation*}
where $d\sigma$ is a suitably normalized measure on $S_{x,t}$.
\par We first introduce some notation. For convenience, we temporarily fix the parameter $t$, and consider a mapping $x\mapsto S_x$ that assigns to each $x$ (in some region of $\R^n$) a hypersurface $S_x$.
This mapping will be determined by a smooth real-valued defining function $\Phi(x,y)$ given in some region of $\R^n\times\R^n$ by the condition
\begin{equation}
	\label{eq:S_x def}
	S_x=\{y:\Phi(x,y)=0\}.
\end{equation}
Define the \emph{rotational curvature} of $\Phi$, denoted by $\operatorname{rotcurv}(\Phi)$, by
\begin{equation*}
	\operatorname{rotcurv}(\Phi)=\det
\begin{pmatrix}
\Phi & \Phi_{x_1} & \cdots & \Phi_{x_n} \\
\Phi_{y_1} & \Phi_{x_1y_1} & \cdots & \Phi_{x_ny_1} \\
\vdots & \vdots & \ddots & \vdots \\
\Phi_{y_n} & \Phi_{x_1y_n} & \cdots & \Phi_{x_ny_n}
\end{pmatrix}.
\end{equation*}
For further discussion of rotational curvature, see \cite[Chapter XI, Section 3]{Stein1993}.
Our basic assumption is
\begin{equation}
	\label{eq:rotcurv neq 0}
	\operatorname{rotcurv}(\Phi)(x,y) \neq 0, \quad \text{ where } \Phi(x,y)=0.
\end{equation}
Condition (\ref{eq:rotcurv neq 0}) implies $\nabla_y\Phi\neq0$ where $\Phi=0$. Thus the hypersurfaces $S_x$, implicitly defined by (\ref{eq:S_x def}), are smooth submanifolds that vary smoothly with $x$.
\par We now assign a measure to a surface $S\subset\R^n$. A single hypersurface $S$ is implicitly defined by
\begin{equation*}
	S=\{x:\Phi(x)=0\},
\end{equation*}
provided $\nabla\Phi(x)\neq0$ whenever $x\in S$. $\Phi$ is called the \emph{defining function} for $S$.
We define the Dirac measure $\delta(\Phi)dy$ associated with $\Phi$ by
\begin{equation*}
	\int_{\R^n} f(y)\delta(\Phi)dy=\lim_{\varepsilon\to0}\frac{1}{2\varepsilon}\int_{-\varepsilon<\Phi(x)<\varepsilon}f(x)dx.
\end{equation*}
\par To study the variable hypersurfaces, we consider a parametrized family $\Phi^t$, depending smoothly on $t$ in the closed interval $0\leq t\leq1$.
The corresponding \emph{scaled family}, denoted by $\Phi_t$, is defined by
\begin{equation*}
	\Phi_t(x,y)=\Phi^t\Big(x,x+\frac{y-x}{t}\Big), \quad \text{ for } t>0.
\end{equation*}
(Here the notation $\Phi^t$ denotes the parametrized family with parameter $t>0$, rather than the $t$-power of $\Phi$.)
\par Now we can define the averaging operator.
For an $\M$-valued function $f$, define the averages
\begin{equation}
	\label{eq:A_t(f) on hypersurface}
	A_t(f)(x)=\int_{\R^n}f(y)\psi_t(x,y)\delta(\Phi_t)dy,
\end{equation}
where $\psi_t(x,y)=t^{-n}\psi(t,x,(x-y)/t)$, and $\psi$ is a fixed smooth function with compact support in all variables.
Note that the measure $\delta(\Phi_t)dy$ is on $\R^n$ (supported on the $n-1$-dimensional surface $S_{x,t}$). The factor $t^{-n}$ is introduced to normalize the total mass of the integral in (\ref{eq:A_t(f) on hypersurface}).
\par We can now state the main theorem of this section.
\begin{theorem}
	\label{thm:A_t bounded for p>n/(n-1)}
	Suppose that $\Phi_t$ is a scaled family satisfying the condition
	\begin{equation}
		\label{eq:t^2n rotcurv>c>0}
		t^{2n}\operatorname{rotcurv}(\Phi_t)(x,y)\geq c>0 \quad \text{ where }\quad \Phi_t(x,y)=0,
	\end{equation}
	for $0<t\leq 1$. Then the map $f\mapsto \{A_t(f)\}_{0<t\leq1}$ extends to a bounded operator from $L^p(\A^n) \to L^p(\A^n;\ell^\infty)$
	provided that $n\geq3$ and $p>n/(n-1)$. That is,
	\begin{equation*}
		\Big\|\sup_{0<t\leq 1}A_t(f)\Big\|_{L^p(\A^n)} \lesssim_p \Norm{f}_{L^p(\A^n)}.
	\end{equation*}
\end{theorem}

\subsection{Auxiliary lemmas}
\label{subsec:hypersurface-auxiliary-lemmas}
\par We introduce the auxiliary lemmas that will be used in the proof.
\par Suppose $\B$ is a Banach $\M$-bimodule, that is, $\B\subset L^0(\M,\tau)$ is a subspace equipped with a norm $\Norm{\cdot}_\B$ such that
\begin{enumerate}
	\item $(\B,\Norm{\cdot}_\B)$ is a Banach space;
	\item $\B$ is an $\M$-bimodule, i.e. $a,b\in\M$ and $x\in\B$ imply $axb\in\B$, and moreover, we have $\Norm{axb}_{\B}\leq \Norm{a}_\infty\Norm{b}_\infty\Norm{x}_\B$.
\end{enumerate}
For example, the noncommutative $L^p$-space $L^p(\M)$, the Orlicz space $L^\Phi(\M)$, the spaces $L^1+L^\infty(\M)$ and $L^1\cap L^\infty(\M)$ are all Banach $\M$-bimodules.
It is easy to see that $x\in\B$ implies $|x|\in\B$ and $x^*\in \B$.
\begin{lemma}
	\label{lem:Sobolev}
	Suppose $I\subset\R$ is a nondegenerate interval, $t_0\in I$, $0<l\leq|I|$. Then for every function $F\in C^1(I;\B)$ with $F^*=F$, we have
	\begin{equation*}
		F(t_0) \leq l^{-1/2}\bigg(\int_I|F(t)|^2dt\bigg)^{1/2}+l^{1/2}\bigg(\int_I|F'(t)|^2dt\bigg)^{1/2}
	\end{equation*}
	with respect to the partial order in $L^0(\M)_+$.
\end{lemma}
\begin{proof}
	Choose $I_0$ such that $t_0\in I_0\subset I$ and $|I_0|=l$. For $t_0,s\in I_0$, we have
	\begin{equation*}
		F(t_0)-F(s)=\int_s^{t_0}F'(u)du.
	\end{equation*}
	Since $F$ and $F'$ are self-adjoint, we have
	\begin{equation*}
		F(t_0)=F(s)+\int_s^{t_0}F'(u)du \leq |F(s)|+\int_{\min\{s,t_0\}}^{\max\{s,t_0\}}|F'(u)|du.
	\end{equation*}
	Taking averages over $s\in I_0$ gives
	\begin{equation*}
		F(t_0)\leq \frac1{l}\int_{I_0}|F(s)|ds+\int_{I_0}|F'(u)|du.
	\end{equation*}
	By the operator-valued H\"older's inequality (see Lemma \ref{lem:Holder for operator valued})
	\begin{equation*}
		F(t_0)\leq l^{-1/2}\left(\int_{I_0}|F(s)|^2ds\right)^{1/2}+l^{1/2}\left(\int_{I_0}|F'(u)|^2du\right)^{1/2}.
	\end{equation*}
	Finally, note that $\int_{I_0}|F(s)|^2ds\leq \int_{I}|F(s)|^2ds$ implies
	\begin{equation*}
		\left(\int_{I_0}|F(s)|^2ds\right)^{1/2}\leq \left(\int_{I}|F(s)|^2ds\right)^{1/2}.
	\end{equation*}
	This proves the lemma.
\end{proof}

\par We next consider an operator-valued Fourier integral operator.
For $f:\R^n\to\mathcal{S}_{\M}^+$ with compact support, define
\begin{equation*}
	T_\lambda(f)(\xi)=\int_{\R^n}e^{i\lambda\Phi(x,\xi)}\psi(x,\xi)f(x)dx,
\end{equation*}
where $\psi$ is a fixed smooth function with compact support in $x$ and $\xi$, and $\Phi$ is a real-valued smooth function on $\R^{2n}$ satisfying
\begin{equation}
	\label{eq:mix Hessian neq 0}
	\det\left(\left(\frac{\partial^2\Phi(x,\xi)}{\partial x_i \partial \xi_j}\right)_{1\leq i,j\leq n}\right)\neq 0 \quad \text{ on support of } \psi.
\end{equation}
\begin{lemma}
	\label{lem:oscillatory integral operator estimate}
	Let $T_\lambda$ be the operator defined above, and suppose that the function $\Phi$ satisfies the condition (\ref{eq:mix Hessian neq 0}). Then $T_\lambda$ is $L^2$ bounded with norm $O(\lambda^{-n/2})$. That is
	\begin{equation*}
		\Norm{T_\lambda(f)}_{L^2(\A^n)}\lesssim \lambda^{-n/2}\Norm{f}_{L^2(\A^n)}.
	\end{equation*}
\end{lemma}
\begin{proof}
	We use the $TT^*$ method. Write
	\begin{equation*}
		(T_\lambda T_\lambda^* f)(\xi)=\int_{\R^n}K_{\lambda}(\xi,\eta)f(\eta)d\eta,
	\end{equation*}
	where the kernel is
	\begin{equation*}
		K_{\lambda}(\xi,\eta)=\int_{\R^n}e^{i\lambda(\Phi(x,\xi)-\Phi(x,\eta))}\psi(x,\xi)\overline{\psi}(x,\eta)dx.
	\end{equation*}
	According to \cite[Page 379]{Stein1993}, for any $N>0$, the kernel satisfies
	\begin{equation*}
		|K_\lambda(\xi,\eta)|\lesssim_N \frac{1}{(1+\lambda|\xi-\eta|)^N}.
	\end{equation*}
	By the Banach-valued version of Schur's lemma in Lemma \ref{lem:Schur}, the operator $T_\lambda T_\lambda^*$ has norm bounded by
	\begin{equation*} \int_{\R^n}\frac{1}{(1+\lambda|\xi|)^N}d\xi\lesssim\lambda^{-n},
	\end{equation*}
if we choose $N>n$. This proves the lemma.
\end{proof}
As a corollary, we have the following lemma.
\begin{lemma}
	\label{lem:T_lambda estimate}
	Let
	\begin{equation*}
		T_\lambda(f)=\int_{\R^{n+1}}e^{i\lambda y_0\Phi(x,y)}f(y)\psi(x,y,y_0)dydy_0.
	\end{equation*}
	Here $\psi$ is a fixed function on $\R^n\times\R^n\times\R^1$ that has compact support, whose projection on $\R^1$ does not contain the point $y_0=0$. The function $\Phi(x,y)$ satisfies
	\begin{equation}
		\label{eq:rotcurv(Phi) neq 0}
		\operatorname{rot curv}(\Phi)\neq 0 \quad \text{ for all } (x,y)\in \supp(\psi).
	\end{equation}
	Then
	\begin{equation}
		\label{eq:T_lambda(f) estimate}
		\Norm{T_\lambda(f)}_{L^2(\A^n)} \lesssim \lambda^{-(n+1)/2}\Norm{f}_{L^2(\A^n)}.
	\end{equation}
\end{lemma}
\begin{proof}
	We define a new operator $\tilde{T}_\lambda$, mapping operator-valued functions on $\R^n$ to functions on $\R^{n+1}$, by
	\begin{equation}
		\label{eq:tilde T_lambda f}
		(\tilde{T}_{\lambda}f)(x,x_0)=\int_{\R^{n+1}}e^{i\lambda x_0y_0\Phi(x,y)}f(y)\tilde{\psi}(x,y,x_0,y_0)dydy_0.
	\end{equation}
	Here $\tilde{\psi}(x,y,x_0,y_0)=\eta(x_0)\psi(x,y,y_0)$, with $\eta$ a smooth function of $x_0$, supported in $\{1/4<x_0<2\}$, and so that $\eta(x_0)=1$ when $1/2\leq x_0\leq 3/2$.
	Now let $\tilde{\Phi}(x,y,x_0,y_0)=x_0y_0\Phi(x,y)$. A direct calculation gives
	\begin{equation*}
		\det\left(\frac{\partial^2\tilde{\Phi}}{\partial x_i\partial y_j}\right)_{0\leq i,j\leq n}=(x_0y_0)^n\operatorname{rotcurv}(\Phi).
	\end{equation*}
	By our assumption (\ref{eq:rotcurv(Phi) neq 0}), the phase function $\tilde{\Phi}$ has a nonvanishing Hessian in the support of $\tilde{\psi}$.
	Then it follows from Lemma \ref{lem:oscillatory integral operator estimate} in the $n+1$ dimensional case  that
	\begin{equation*}
		\|{\tilde{T}_\lambda(f)}\|_{L^2(\A^{n+1})}\lesssim\lambda^{-(n+1)/2}\Norm{f}_{L^2(\A^n)}.
	\end{equation*}
	Next apply the same reasoning to $\frac{\partial}{\partial x_0}\tilde{T}_\lambda f(x,x_0)$. In fact, differentiation with respect to $x_0$ in (\ref{eq:tilde T_lambda f}) merely brings down a factor of $i\lambda y_0\Phi(x,y)$ or replaces the function $\eta(x_0)$ by $\eta'(x_0)$, and note that
	\begin{equation*}
		i\lambda y_0\Phi(x,y)e^{i\lambda x_0y_0\Phi(x,y)}=\frac{y_0}{x_0}\frac{\partial}{\partial y_0}\Big(e^{i\lambda x_0y_0\Phi(x,y)}\Big).
	\end{equation*}
	Substituting this into \eqref{eq:tilde T_lambda f} and integrating by parts, we obtain an integral of the same kind as $\tilde{T}_\lambda$, and again by Lemma \ref{lem:oscillatory integral operator estimate},
	\begin{equation*}
		\Big\|\frac{\partial}{\partial x_0}\tilde{T}_\lambda(f)\Big\|_{L^2(\A^{n+1})}\lesssim\lambda^{-(n+1)/2}\Norm{f}_{L^2(\A^n)}.
	\end{equation*}
	Now we have $\tilde{T}_\lambda(f)(x,1)=T_\lambda(f)(x)$. Splitting $\tilde{T}_\lambda(f)$ and $T_\lambda(f)$ into real and imaginary parts gives
	\begin{equation*}
		\tilde{T}_\lambda(f)=\tilde{T}^{(1)}_\lambda(f)+i\tilde{T}^{(2)}_\lambda(f), \qquad T_\lambda(f)=T_\lambda^{(1)}(f)+iT_\lambda^{(2)}(f).
	\end{equation*}
	It suffices to show that (\ref{eq:T_lambda(f) estimate}) holds with $T_\lambda$ replaced by each $T_\lambda^{(i)}$ ($i=1,2$). Without loss of generality, we may assume $f$ is self-adjoint, hence $T_\lambda^{(i)}(f)$ and $\tilde{T}_\lambda^{(i)}(f)$ are self-adjoint for $i=1,2$. Then by Lemma \ref{lem:Sobolev}, for $i=1,2$, we have
	\begin{equation*}
\begin{split}
		T_\lambda^{(i)}&(f)(x)=\tilde{T}_\lambda^{(i)}(f)(x,1) \\
		&\leq \left(\int_{1/2}^{3/2}|\tilde{T}_\lambda^{(i)}(f)(x,x_0)|^2dx_0\right)^{1/2}+\left(\int_{1/2}^{3/2}\Big|\frac{\partial}{\partial x_0}\tilde{T}_\lambda^{(i)}(f)(x,x_0)\Big|^2dx_0\right)^{1/2} \\
		&:=(I_\lambda)+(II_\lambda),
\end{split}
	\end{equation*}
	and the same inequality also holds with $T_\lambda^{(i)}(f)(x)$ replaced by $-T_\lambda^{(i)}(f)(x)$.
	Since both $T_\lambda^{(i)}(f)(x)$ and $-T_\lambda^{(i)}(f)(x)$ are majorized by positive operators, we have
	\begin{equation*}
		\begin{split}
			\|T_\lambda^{(i)}(f)\|_{L^2(\A^n)}
			&\leq \Norm{(I_\lambda)}_{L^2(\A^n)}+\Norm{(II_\lambda)}_{L^2(\A^n)} \\
			&=\left[\tau\int_{\R^n}\int_{1/2}^{3/2} |\tilde{T}_\lambda^{(i)}(f)(x,x_0)|^2dx_0dx\right]^{1/2} \\
			&\qquad\qquad +\left[\tau\int_{\R^n}\int_{1/2}^{3/2} \Big|\frac{\partial}{\partial x_0}\tilde{T}_\lambda^{(i)}(f)(x,x_0)\Big|^2dx_0dx\right]^{1/2} \\
			&\leq \|\tilde{T}_\lambda^{(i)}(f)\|_{L^2(\A^{n+1})}+\Norm{\frac{\partial}{\partial x_0}\tilde{T}_\lambda^{(i)}(f)}_{L^2(\A^{n+1})} \\
			&\lesssim \lambda^{-\frac{n+1}{2}}\Norm{f}_{L^2(\A^n)}.
		\end{split}
	\end{equation*}
	This completes the proof.
\end{proof}

\par The constant in (\ref{eq:T_lambda(f) estimate}) is uniform for phases and amplitudes with fixed support size and uniformly bounded derivatives up to a fixed finite order. We also require uniform positive lower bounds for $|\operatorname{rotcurv}(\Phi)|$ and the distance of the $y_0$-support from zero. This follows from the proof above; see also \cite[Chapter XI, Section 3.2.1]{Stein1993}.

\par We also need the following local-to-global estimate.
Choose $0\neq\eta\in C_c^\infty(\R^n)$ with $\eta\geq0$ and $\supp\eta\subset B(0,r)$. Here $B(0,r)$ is the Euclidean ball of radius $r$. For $x_0\in\R^n$, define
\begin{equation*}
\eta_{x_0}(x)=\eta(x-x_0),\qquad T_{x_0}(f)=\eta_{x_0}T(f).
\end{equation*}
\begin{lemma}
	\label{lem:local-global}
	Let $T$ be an integral operator on $\R^n$ with kernel $K(x,y)$. Suppose that, for a fixed $\rho>0$,
	\begin{equation}
		\label{eq:K finite propagation}
		K(x,y)=0 \quad \text{ whenever }\quad |x-y|>\rho,
	\end{equation}
Assume also that
	\begin{equation*}
		\|T_{x_0}(f)\|_{L^2(\A^n)}\leq C\Norm{f}_{L^2(\A^n)}, \quad \text{ uniformly in } x_0\in\R^n.
	\end{equation*}
	Then
	\begin{equation*}
		\|T(f)\|_{L^2(\A^n)} \leq
		C\frac{|B(0,r+\rho)|^{1/2}}{\Norm{\eta}_{L^2(\R^n)}}\Norm{f}_{L^2(\A^n)}.
	\end{equation*}
\end{lemma}
\begin{proof}
	If $\eta_{x_0}(x)K(x,y)\neq0$, then
	\begin{equation*}
	|y-x_0|\leq|y-x|+|x-x_0|\leq\rho+r.
	\end{equation*}
	Thus (\ref{eq:K finite propagation}) gives
	\begin{equation*}
	T_{x_0}(f)=T_{x_0}\big(f1_{B(x_0,r+\rho)}\big).
	\end{equation*}
	Apply the assumed bound for $T_{x_0}$ to $f1_{B(x_0,r+\rho)}$. We obtain
	\begin{equation*}
		\begin{split}
		\int_{\R^n}\tau(|\eta(x-x_0)|^2|Tf(x)|^2)dx
		&=\|T_{x_0}(f)\|_{L^2(\A^n)}^2=\|T_{x_0}(f1_{B(x_0,r+\rho)})\|_{L^2(\A^n)}^2 \\
		&\leq C^2\int_{|y-x_0|\leq \rho+r}\tau(|f(y)|^2)dy.
		\end{split}
	\end{equation*}
	Integrate in $x_0\in\R^n$. Fubini's theorem gives
	\begin{equation*}
		\Norm{\eta}_{L^2(\R^n)}^2\Norm{Tf}_{L^2(\A^n)}^2
		\leq C^2|B(0,r+\rho)|\Norm{f}_{L^2(\A^n)}^2.
	\end{equation*}
	Divide by $\Norm{\eta}_{L^2(\R^n)}^2$ and take square roots. This proves the lemma.
\end{proof}
\medskip

In the next two subsections, we complete the proof of Theorem \ref{thm:A_t bounded for p>n/(n-1)}.

\subsection{Proof when $t$ is near 1}
\label{subsec:hypersurface-large-t}
Fix $0<t_0<1$. We first control the averages $A_t$ for $t_0\leq t\leq1$. It suffices to work with smooth compactly supported functions with coefficients in $\mathcal{S}_\M$. By splitting the amplitude into real and imaginary parts, we may assume it is real. All constants below may depend on the fixed phase and amplitude, and in this subsection also on $t_0$.

\par Since only the amplitude on $\Phi_t=0$ contributes to the surface integral, we may multiply it by a smooth cutoff supported in a neighborhood where the rotational curvature is bounded away from zero. This leaves $A_t$ unchanged.
Choose $\alpha$ to be a real even smooth function on $\R^1$ so that $\alpha(\mu)=1$ for $|\mu|\leq1$ and $\alpha(\mu)=0$ for $|\mu|\geq 2$. Then
\begin{equation*}
	\begin{split}
		\delta(u)=\int_{-\infty}^{\infty}e^{2\pi iu\mu}d\mu=\int_{-\infty}^{\infty}\alpha(\mu)e^{2\pi iu\mu}d\mu+\sum_{j=1}^\infty \int_{-\infty}^{\infty}\beta(2^{-j}\mu)e^{2\pi iu\mu}d\mu,
	\end{split}
\end{equation*}
where $\beta(\mu)=\alpha(\mu)-\alpha(2\mu)$ is supported in $1/2\leq|\mu|\leq2$. Now write
\begin{equation*}
	(A_t^jf)(x)=\int_\R\int_{\R^n}\beta(2^{-j}\mu)e^{2\pi i\mu\Phi_t(x,y)}f(y)\psi_t(x,y)dyd\mu, \quad \text{ for } j\geq1;
\end{equation*}
and for $j=0$ define
\begin{equation*}
	(A_t^0f)(x)=\int_\R\int_{\R^n}\alpha(\mu)e^{2\pi i\mu\Phi_t(x,y)}f(y)\psi_t(x,y)dyd\mu.
\end{equation*}
Then we have
\begin{equation*}
	(A_tf)(x)=\sum_{j=0}^\infty (A_t^jf)(x).
\end{equation*}
The change of variables $y_0=2^{-j}\mu$ gives, for $j\geq1$,
\begin{equation*}
	(A_t^jf)(x)=2^j\int_{\R^{n+1}}e^{i\lambda y_0\Phi_t(x,y)}f(y)\psi^*_t(x,y,y_0)dydy_0, \quad \lambda=2\pi\cdot 2^j,
\end{equation*}
with $\psi^*_t(x,y,y_0)=t^{-n}\psi(t,x,(x-y)/t)\beta(y_0)$. We first treat the high frequency term $j>0$. The low-frequency term $j=0$ will be estimated separately later.
Notice that Lemma \ref{lem:T_lambda estimate} guarantees that
\begin{equation}
	\label{eq:A_t^j L2 bound}
	\|A_t^j(f)\|_{L^2(\A^n)}\lesssim 2^j2^{-\frac{n+1}{2}j}\Norm{f}_{L^2(\A^n)}, \quad t_0\leq t\leq 1.
\end{equation}
The same argument applied to $\frac{\partial}{\partial t}A_t^j(f)$ gives
\begin{equation}
	\label{eq:A_t^j derivative L2 bound}
	\Big\| \frac{\partial}{\partial t}A_t^j(f) \Big\|_{L^2(\A^n)} \lesssim 2^{2j}\cdot2^{-\frac{n+1}{2}j}\Norm{f}_{L^2(\A^n)}, \quad t_0\leq t\leq 1,
\end{equation}
where an additional factor $2^j$ is obtained by the differentiation applied to $e^{i\lambda y_0\Phi_t(x,y)}$.
For self-adjoint $f$, the even real cutoff $\beta$ and the real amplitude imply that $A_t^j(f)$ is self-adjoint. Now by Lemma \ref{lem:Sobolev}, for all $t$ with $t_0\leq t\leq1$, we have
\begin{equation}
	\label{eq:A_t^j(f)<...}
	A_t^j(f) \leq l^{-\frac12}\left(\int_{t_0}^1|A_t^jf|^2dt\right)^{\frac12}+l^{\frac12}\left(\int_{t_0}^1\Big|\frac{\partial}{\partial t}(A_t^jf)\Big|^2dt\right)^{\frac12},
\end{equation}
and
\begin{equation}
	\label{eq:-A_t^j(f)<...}
	-A_t^j(f) \leq l^{-\frac12}\left(\int_{t_0}^1|A_t^jf|^2dt\right)^{\frac12}+l^{\frac12}\left(\int_{t_0}^1\Big|\frac{\partial}{\partial t}(A_t^jf)\Big|^2dt\right)^{\frac12}.
\end{equation}
The $L^2(\A^n)$-norm of the right-hand side of the two inequalities (\ref{eq:A_t^j(f)<...}) and (\ref{eq:-A_t^j(f)<...}) is controlled by
\begin{equation*}
	\begin{split}
		&\quad l^{-\frac12}\bigg\|\left(\int_{t_0}^1|A_t^jf|^2dt\right)^{\frac12}\bigg\|_{L^2(\A^n)}+l^{\frac12}\bigg\|\left(\int_{t_0}^1\Big|\frac{\partial}{\partial t}(A_t^jf)\Big|^2dt\right)^{\frac12}\bigg\|_{L^2(\A^n)} \\
	    & = l^{-\frac12}\left(\tau\int_{\R^n}\int_{t_0}^1|A_t^jf(x)|^2dtdx\right)^{1/2}+l^{1/2}\left(\tau\int_{\R^n}\int_{t_0}^1\Big|\frac{\partial}{\partial t}(A_t^jf)(x)\Big|^2dtdx\right)^{\frac12} \\
		& \lesssim l^{-1/2}(1-t_0)^{1/2} 2^j2^{-\frac{n+1}{2}j}\Norm{f}_{L^2(\A^n)} + l^{1/2}(1-t_0)^{1/2} 2^{2j}2^{-\frac{n+1}{2}j}\Norm{f}_{L^2(\A^n)},
	\end{split}
\end{equation*}
where the last inequality follows from (\ref{eq:A_t^j L2 bound}) and (\ref{eq:A_t^j derivative L2 bound}). By taking $l=(1-t_0)2^{-j}$, we obtain
that
\begin{equation*}
	\Norm{\text{RHS of}\quad (\ref{eq:A_t^j(f)<...}), (\ref{eq:-A_t^j(f)<...}) }_{L^2(\A^n)} \leq A\cdot2^{-\frac{n-2}{2}j}\Norm{f}_{L^2(\A^n)}.
\end{equation*}
By (\ref{eq:A_t^j(f)<...}) and (\ref{eq:-A_t^j(f)<...}) and the definition of the noncommutative maximal norm, we immediately get
\begin{equation}\label{e:30A^j_t}
	\Big\|\sup_{t_0\leq t\leq1}A_t^j(f)\Big\|_{L^2(\A^n)} \lesssim2^{-\frac{n-2}{2}j}\Norm{f}_{L^2(\A^n)}.
\end{equation}
This gives the desired decay $L^2$-estimate for a self-adjoint $f$, and hence for the general $f$ by decomposing it into real and imaginary parts.
\par Since $t$ is bounded away from $0$, we also have a crude $L^1$-estimate.
Indeed, integrating first in $\mu$, we can write
\begin{equation*}
	A_t^j(f)(x)=2^j\int_{\R^n}\check{\beta}(2^j\Phi_t(x,y))\psi_t(x,y)f(y)dy, \quad j\geq1.
\end{equation*}
Since $t\geq t_0$, the kernels are bounded by $A2^j\chi_K(x)\chi_K(y)$ for a fixed compact set $K\subset\R^n$, independent of $t$ and $j$. Thus, for self-adjoint $f$,
\begin{equation*}
	-A2^j\chi_K(x)\int_K|f(y)|dy
	\leq A_t^j(f)(x)\leq A2^j\chi_K(x)\int_K|f(y)|dy.
\end{equation*}
The common positive majorant on the right has $L^1(\A^n)$-norm at most $A'2^j\Norm{f}_{L^1(\A^n)}$. Decomposing a general $f$ into its real and imaginary parts gives
\begin{equation}
	\label{eq:A_t^j large t L1 maximal}
	\Big\|\sup_{t_0\leq t\leq1}A_t^j(f)\Big\|_{L^1(\A^n)}
	\lesssim 2^j\Norm{f}_{L^1(\A^n)}.
\end{equation}
The same kernel argument, with $\check{\alpha}$ for $j=0$, gives a bound independent of $j$ on every $L^p(\A^n)$, $1\leq p\leq\infty$.

\par We now interpolate (\ref{eq:A_t^j large t L1 maximal}) with the $L^2$ maximal estimate above. For this we regard $f\mapsto\{A_t^j(f)\}_{t\in J}$ as a linear operator with values in $L^p(\A^n;\ell^\infty(J))$, where $J\subset[t_0,1]$ is finite.
By the complex interpolation property recalled in Appendix \ref{subsec:nc-vector-valued-spaces}, for $1<p<2$ we have
\begin{equation*}
	\Big\|\sup_{t_0\leq t\leq1}A_t^j(f)\Big\|_{L^p(\A^n)}
	\lesssim_p 2^{j\delta}\Norm{f}_{L^p(\A^n)}, \qquad
	\delta=-(1-\theta)\frac{n-2}{2}+\theta,
\end{equation*}
where $1/p=(1-\theta)/2+\theta$. The bounds are uniform in $J$, so they pass to the whole parameter interval. Here we apply the vector-valued interpolation replacing the measurable-choice argument in the classical proof.
Since $\delta=1-n+n/p<0$ when $p>n/(n-1)$, the triangle inequality and summation in $j$ yield
\begin{equation}
	\label{eq:A_t large t Lp maximal}
	\Big\|\sup_{t_0\leq t\leq1}A_t(f)\Big\|_{L^p(\A^n)}
	\lesssim_p\Norm{f}_{L^p(\A^n)}, \quad \frac{n}{n-1}<p\leq2.
\end{equation}
Here the case $p=2$ follows directly from the $L^2$ estimates in \eqref{e:30A^j_t} because we only consider $n\geq3$.

\subsection{Proof when $t$ is near zero}
\label{subsec:hypersurface-small-t}
\par We use the geometric reductions in \cite[Chapter XI, Section 3.4]{Stein1993}. We first put the surfaces in graph form and prove uniform $L^2$ estimates for the dyadic pieces. We then construct a common positive majorant on $0<t\leq t_0$. This gives the noncommutative $L^2$ maximal bound. The $L^p$ estimate follows by interpolation with a bound obtained from ball averages.
As in Section \ref{subsec:hypersurface-large-t}, we first work with smooth compactly supported functions with coefficients in $\mathcal S_\M$. By splitting the amplitude into real and imaginary parts, we may assume it is real. Choose a sufficiently large integer $m_0$ and set $t_0=2^{-m_0}$.

\par We first choose a graph representation. Take a finite partition of unity near the zero set in the compact amplitude support, in the variables $(t,x,(y-x)/t)$. On each part, rotate the coordinates so that $|\partial_{y_n}\Phi^t|$ has a uniform positive lower bound. The implicit function theorem gives a function $h_t(x,z')$, $z'\in\R^{n-1}$, smooth up to $t=0$, such that
\begin{equation*}
\Phi^t\big(x,x+(z',h_t(x,z'))\big)=0.
\end{equation*}
Notice that the scaled surface is given by $S_{x,t}=\{y:\Phi_t(x,y)=0\}$ where $\Phi_t(x,y)=\Phi^t\big(x,x+\frac{y-x}{t}\big)$. Hence we have
$$\Phi^t\big(x,x+z\big)=0 \iff z_n=h_t(x,z')\iff \frac{y_n-x_n}{t}=h_t\Big(x,\frac{y'-x'}{t}\Big).$$
Thus the scaled surface has the defining function
\begin{equation}
\label{eq:Phi_t graph form}
\tilde\Phi_t(x,y)=h_t\Big(x,\frac{y'-x'}t\Big)-\frac{y_n-x_n}t,
\qquad y=(y',y_n).
\end{equation}
Fix one coordinate patch and absorb the corresponding smooth cutoff into $\psi$. We write $A_t$ for the average localized to this patch. We adjust the amplitude so that replacing $\Phi_t$ by $\tilde\Phi_t$ leaves $A_t$ unchanged.
Notice that 
$$\delta\big(\Phi_t\big)
=
\frac{\delta\big(z_n-h_t(x,z')\big)}
{\left|\partial_{y_n}\Phi^t
\big(x,x+(z',h_t(x,z'))\big)\right|}$$ 
Then we set
\begin{equation*}
\begin{split}
\tilde\psi(t,x,-z')
&=\frac{\psi\big(t,x,(-z',-h_t(x,z'))\big)}
{\big|\partial_{y_n}\Phi^t\big(x,x+(z',h_t(x,z'))\big)\big|},\\
\tilde\psi_t(x,y)&=t^{-n}\tilde\psi\Big(t,x,\frac{x'-y'}t\Big).
\end{split}
\end{equation*}
The partition-of-unity cutoff absorbed into \(\psi\) is compactly supported in the interior of the chosen coordinate patch.
Hence $\tilde\psi$ extends by zero to a smooth function with compact support in $(t,x,z')$.
On the graph $\tilde\Phi_t=0$, we have $t\tilde\psi_t=\psi_t/|\partial_{y_n}\Phi_t|$. Since $\partial_{y_n}\tilde\Phi_t=-t^{-1}$, we then have
\begin{equation*}
A_t(f)(x)=\int_{\R^n}f(y)\tilde\psi_t(x,y)\delta(\tilde\Phi_t(x,y))dy.
\end{equation*}

\par We retain the real even cutoffs $\alpha$ and $\beta$ from Section \ref{subsec:hypersurface-large-t}. For $j\geq1$, define the dyadic pieces in the Fourier variable $\mu$ by
\begin{equation}
\label{eq:tilde A_t^j graph def}
\tilde A_t^j(f)(x)=\int_\R\int_{\R^n}
\beta(2^{-j}\mu)e^{2\pi i\mu\tilde\Phi_t(x,y)}\tilde\psi_t(x,y)f(y)dyd\mu.
\end{equation}
Define $\tilde A_t^0$ by replacing $\beta(2^{-j}\mu)$ with $\alpha(\mu)$. Then $A_t=\sum_{j\geq0}\tilde A_t^j$ on smooth functions. The amplitude $\tilde\psi_t$ is independent of $y_n$. For the oscillatory integral estimate, we also introduce operators with kernels compactly supported in $(x,y)$. Choose $C$ so that $|h_t(x,z')|<C/2-1$ whenever $\tilde\psi(t,x,-z')\neq0$. Choose a real $\chi\in C_c^\infty(\R)$ with $0\leq\chi\leq1$, $\chi=1$ on $[-C,C]$ and $\supp\chi\subset[-2C,2C]$. Put
\begin{equation*}
\psi_t(x,y)=\tilde\psi_t(x,y)\chi\Big(\frac{y_n-x_n}t\Big).
\end{equation*}
From now on, $\psi_t$ denotes this compactly supported graph amplitude. We define $A_t^j$ by (\ref{eq:tilde A_t^j graph def}) with $\psi_t$ in place of $\tilde\psi_t$. The cutoff equals $1$ on the graph, so it leaves the localized surface average $A_t$ unchanged. The choice of $C$ also gives $|\tilde\Phi_t|\geq1$ on the support of $\tilde\psi_t-\psi_t$. We use this separation to estimate the cutoff error below.

\par The graph phase satisfies the required curvature condition. Near the graph, write $\Phi_t=q_t\tilde\Phi_t$, where $q_t$ is smooth and nonzero. Differentiating in $y_n$ and using $\partial_{y_n}\tilde\Phi_t=-t^{-1}$ gives
\begin{equation*}
\partial_{y_n}\Phi_t=(\partial_{y_n}q_t)\tilde\Phi_t-\frac{q_t}{t}.
\end{equation*}
On the graph, $\tilde\Phi_t=0$, so $q_t=-t\partial_{y_n}\Phi_t$. Subtracting multiples of the first row and column in the curvature determinant gives
\begin{equation*}
\operatorname{rotcurv}(\Phi_t)=q_t^{n+1}\operatorname{rotcurv}(\tilde\Phi_t)
\qquad\text{on }\tilde\Phi_t=0.
\end{equation*}
The function $|q_t|$ is bounded above and away from zero on the coordinate patch. By (\ref{eq:t^2n rotcurv>c>0}), we have
\begin{equation*}
t^{2n}|\operatorname{rotcurv}(\tilde\Phi_t)|
=\frac{t^{2n}|\operatorname{rotcurv}(\Phi_t)|}{|q_t|^{n+1}}\gtrsim1
\qquad\text{on }\tilde\Phi_t=0.
\end{equation*}
Also, $\partial_{y_n}\tilde\Phi_t=-t^{-1}$ and $\partial_{x_i}\partial_{y_n}\tilde\Phi_t=0$. Expanding the determinant along the $y_n$ row shows that $\operatorname{rotcurv}(\tilde\Phi_t)$ is independent of $y_n$. The lower bound therefore holds on the whole support of $\psi_t$.

\par We now prove, uniformly for $0<t\leq t_0$ and $j\geq1$,
\begin{equation}
\label{eq:A_t^j small t L2 bounds}
\begin{split}
\Norm{A_t^j(f)}_{L^2(\A^n)}&\lesssim2^{-j(n-1)/2}\Norm{f}_{L^2(\A^n)},\\
\Norm{\partial_t A_t^j(f)}_{L^2(\A^n)}&\lesssim t^{-1}2^j2^{-j(n-1)/2}\Norm{f}_{L^2(\A^n)}.
\end{split}
\end{equation}
To prove the first estimate in \eqref{eq:A_t^j small t L2 bounds}, let $D_tf(x)=t^{n/2}f(tx)$ and $B_t^j=D_tA_t^jD_t^{-1}$.
Since the map $D_t$ is unitary on $L^2(\A^n)$, $A_t^j$ and $B_t^j$ have the same $B(L^2(\A^n))$-norm. With $\lambda=2\pi\cdot2^j$, the kernel of $B_t^j$ is
\begin{equation*}
2^j\int_\R e^{i\lambda y_0\Psi_t(x,y)}
\tilde\psi(t,tx,x'-y')\chi(y_n-x_n)\beta(y_0)dy_0,
\end{equation*}
where
\begin{equation*}
\Psi_t(x,y)=\tilde\Phi_t(tx,ty)=h_t(tx,y'-x')-(y_n-x_n).
\end{equation*}
The phase and amplitude have uniformly bounded spatial derivatives on the amplitude support, and
\begin{equation*}
|\operatorname{rotcurv}(\Psi_t)(x,y)|
=t^{2n}|\operatorname{rotcurv}(\tilde\Phi_t)(tx,ty)|\gtrsim1.
\end{equation*}
The kernel vanishes for $|x-y|>\rho$, with $\rho$ independent of $t$. The $x$-support of the dilated amplitude may grow as $t\to0$. To keep the support size fixed in Lemma \ref{lem:T_lambda estimate}, we localize in $x$ by the functions $\eta_{x_0}$ in Lemma \ref{lem:local-global}. After translation by $x_0$, each localized kernel has support in fixed balls. All derivative and curvature bounds remain uniform in $x_0$ and $t$. Lemma \ref{lem:T_lambda estimate} gives
\begin{equation*}
\Norm{\eta_{x_0}B_t^j(f)}_{L^2(\A^n)}
\lesssim2^j\lambda^{-(n+1)/2}\Norm{f}_{L^2(\A^n)}.
\end{equation*}
Lemma \ref{lem:local-global} and the unitarity of $D_t$ prove the first bound in (\ref{eq:A_t^j small t L2 bounds}).

\par We next prove the second estimate in (\ref{eq:A_t^j small t L2 bounds}). Keeping $x,y$ fixed, we differentiate the kernel of $A_t^j$ before conjugating by $D_t$. The product rule gives
\begin{equation*}
\partial_t\big(e^{i\lambda y_0\tilde\Phi_t}\psi_t\big)=e^{i\lambda y_0\tilde\Phi_t}\left[\frac{\lambda}{t}iy_0(t\partial_t\tilde\Phi_t)\psi_t+\frac1t(t\partial_t\psi_t)\right].
\end{equation*}
After removing the factors $\lambda/t$ and $1/t$, we must check that the resulting amplitudes satisfy the uniform bounds required by Lemma \ref{lem:T_lambda estimate}. With $z'=(y'-x')/t$, we have
\begin{equation*}
t\partial_t\tilde\Phi_t(x,y)
=t\partial_t h_t(x,z')-z'\cdot\nabla_{z'}h_t(x,z')+\frac{y_n-x_n}t.
\end{equation*}
We now conjugate $\partial_tA_t^j$ by $D_t$. The function $(t\partial_t\tilde\Phi_t)(tx,ty)$ has uniformly bounded spatial derivatives on the support of the dilated amplitude. The normalized amplitudes
\begin{equation*}
t^n\psi_t(tx,ty),\qquad t^n[t\partial_t\psi_t](tx,ty)
\end{equation*}
also have uniformly bounded spatial derivatives. Both terms retain the phase $\Psi_t$ and the support properties used above. Since $y_0$ lies in the fixed support of $\beta$, the factor $iy_0$ also has uniform bounds. Applying Lemmas \ref{lem:T_lambda estimate} and \ref{lem:local-global} to the two terms, and using the unitarity of $D_t$, gives
\begin{equation*}
\begin{split}
\Norm{\partial_tA_t^j(f)}_{L^2(\A^n)}
&\lesssim\frac{2^j}{t}(\lambda+1)\lambda^{-(n+1)/2}\Norm{f}_{L^2(\A^n)}\\
&\lesssim t^{-1}2^j2^{-j(n-1)/2}\Norm{f}_{L^2(\A^n)},
\end{split}
\end{equation*}
where $\lambda=2\pi\cdot2^j$. This proves the second estimate in (\ref{eq:A_t^j small t L2 bounds}).

\medskip

\par We next estimate the error $\tilde A_t^j-A_t^j$ and its $t$-derivative to obtain both bounds in (\ref{eq:A_t^j small t L2 bounds}) for $\tilde A_t^j$. We then compare $\tilde A_t^j-A_t^j$ with ball averages to prove an $L^p$ maximal estimate for the error. Write $\tilde K_t^j$ and $K_t^j$ for the kernels of $\tilde A_t^j$ and $A_t^j$. Their difference is
\begin{equation*}
\tilde K_t^j-K_t^j
=2^j\check\beta(2^j\tilde\Phi_t)\tilde\psi_t
\left(1-\chi\Big(\frac{y_n-x_n}t\Big)\right).
\end{equation*}
On the support of the difference, $|x'-y'|\lesssim t$ and $|y_n-x_n|/t\geq C$. Hence
\begin{equation*}
|\tilde\Phi_t(x,y)|\geq\frac{|y_n-x_n|}{2t}
\gtrsim1+\frac{|x-y|}t.
\end{equation*}
The same lower bound holds on the support of the differentiated cutoff. The rapid decay of $\check\beta$ and $\check\beta'$ gives, for every $M>n$,
\begin{equation}
\label{eq:tilde A error kernel}
\begin{split}
&|\tilde K_t^j-K_t^j|(x,y)+t|\partial_t(\tilde K_t^j-K_t^j)|(x,y)\\
&\qquad\lesssim_M2^{-j(M-1)}t^{-n}\Big(1+\frac{|x-y|}t\Big)^{-M}.
\end{split}
\end{equation}
For the phase derivative, use one additional order of decay to absorb the factor $2^j(1+|x-y|/t)$. Integrating the kernel bounds in either $x$ or $y$ and applying Lemma \ref{lem:Schur} gives
\begin{equation*}
\begin{split}
\Norm{(\tilde A_t^j-A_t^j)(f)}_{L^2(\A^n)}
&\lesssim_M2^{-j(M-1)}\Norm{f}_{L^2(\A^n)},\\
\Norm{\partial_t(\tilde A_t^j-A_t^j)(f)}_{L^2(\A^n)}
&\lesssim_Mt^{-1}2^{-j(M-1)}\Norm{f}_{L^2(\A^n)}.
\end{split}
\end{equation*}
By the triangle inequality, both bounds in (\ref{eq:A_t^j small t L2 bounds}) also hold for $\tilde A_t^j$.

\par For positive $f$, apply (\ref{eq:tilde A error kernel}) with $M+1$ in place of $M$ and split the $y$-integral into Euclidean annuli. We obtain
\begin{equation*}
\pm(\tilde A_t^j-A_t^j)(f)(x)
\leq A_M2^{-jM}\sum_{\ell\geq0}
\frac{2^{-\ell(M+1-n)}}{|B(0,2^\ell t)|}
\int_{|x-y|\leq2^\ell t}f(y)dy.
\end{equation*}
For $1<p<\infty$, the noncommutative maximal inequality in \cite[Theorem 4.1]{HongLiaoWang2021} gives a positive $b_f\in L^p(\A^n)$ such that
\begin{equation*}
\begin{split}
\frac{1}{|B(0,r)|}\int_{|x-y|\leq r}f(y)dy&\leq b_f(x),\qquad r>0,\\
\Norm{b_f}_{L^p(\A^n)}&\lesssim_p\Norm{f}_{L^p(\A^n)}.
\end{split}
\end{equation*}
Since $2^{-\ell(M+1-n)}$ are summable in $\ell$, for $0<t\leq t_0$ we obtain
\begin{equation*}
-C_M2^{-jM}b_f\leq(\tilde A_t^j-A_t^j)(f)\leq C_M2^{-jM}b_f.
\end{equation*}
For general $f$, apply the positive estimate to the positive and negative parts of $\Re f$ and $\Im f$. For $p=\infty$, integrate the kernel bound directly. We conclude that
\begin{equation}
\label{eq:tilde A error maximal}
\Big\|\sup_{0<t\leq t_0}(\tilde A_t^j-A_t^j)(f)\Big\|_{L^p(\A^n)}
\lesssim_{p,M}2^{-jM}\Norm{f}_{L^p(\A^n)},\quad 1<p\leq\infty.
\end{equation}
We use the following kernel bounds and the noncommutative ball maximal inequality to obtain maximal estimates for $\tilde A_t^0$ and $A_t^j$ :
\begin{equation*}
|\tilde K_t^0(x,y)|\lesssim_Mt^{-n}\Big(1+\frac{|x-y|}t\Big)^{-M},
\qquad
|K_t^j(x,y)|\lesssim2^jt^{-n}\chi_{\{|x-y|\leq\rho t\}}.
\end{equation*}
The first bound follows from the decay of $\check\alpha$ and the explicit expression for $\tilde\Phi_t$ in \eqref{eq:Phi_t graph form}. Thus
\begin{equation}
\label{eq:tilde A zero maximal}
\Big\|\sup_{0<t\leq t_0}\tilde A_t^0(f)\Big\|_{L^p(\A^n)}
\lesssim_p\Norm{f}_{L^p(\A^n)},\quad 1<p\leq\infty.
\end{equation}
The bound for $K_t^j$, together with (\ref{eq:tilde A error maximal}), gives
\begin{equation}
\label{eq:tilde A small t crude maximal}
\Big\|\sup_{0<t\leq t_0}\tilde A_t^j(f)\Big\|_{L^{p_1}(\A^n)}
\lesssim_{p_1}2^j\Norm{f}_{L^{p_1}(\A^n)},\quad 1<p_1<\infty.
\end{equation}

\medskip

\par We now construct a common positive majorant for $\tilde A_t^j(f)$. The factor $t^{-1}$ in the derivative bound (\ref{eq:A_t^j small t L2 bounds}) is not square integrable near zero. We therefore apply Lemma \ref{lem:Sobolev} on dyadic intervals. For $t\sim2^{-m}$, the relation $\xi_n=\mu/t$ gives $|\xi_n|\sim2^{m+j}$ on the support of $\beta(2^{-j}\mu)$. We localize $f$ to these frequencies so that the interval bounds can be summed by Plancherel's identity. Choose a real even $\tilde\beta\in C_c^\infty(\R)$ equal to $1$ on $1/4\leq|u|\leq2$ and supported in $1/8\leq|u|\leq4$. Define
\begin{equation*}
\widehat{\Delta_kf}(\xi)=\tilde\beta(2^{-k}\xi_n)\hat f(\xi),\qquad k\in\Z.
\end{equation*}
Recall that $t_0=2^{-m_0}$ and set $I_m=[2^{-m},2^{-m+1}]$ for $m\geq m_0+1$. Then $(0,t_0]=\bigcup_{m\geq m_0+1}I_m$. For $t\in I_m$, the $y_n$-integral in (\ref{eq:tilde A_t^j graph def}) uses the partial Fourier transform at $\xi_n=\mu/t$. On the support of $\beta(2^{-j}\mu)$, we have
\begin{equation*}
\frac14\leq2^{-(m+j)}\frac{|\mu|}t\leq2.
\end{equation*}
Since $\tilde\psi_t$ is independent of $y_n$, it follows that
\begin{equation}
\label{eq:frequency-identity}
\tilde A_t^j\Delta_{m+j}=\tilde A_t^j,\qquad t\in I_m.
\end{equation}
Plancherel's identity gives
\begin{equation}
\label{eq:Delta square L2 bound}
\sum_{k\in\Z}\Norm{\Delta_kf}_{L^2(\A^n)}^2\lesssim\Norm{f}_{L^2(\A^n)}^2.
\end{equation}

\par First suppose $f=f^*$. The real even cutoffs and the real amplitude imply that $\Delta_k$ and $\tilde A_t^j$ preserve self-adjointness. For $m\geq m_0+1$, put $l=2^{-m-j}$ and define
\begin{equation*}
\begin{split}
a_{j,m}=l^{-1/2}\left(\int_{I_m}|\tilde A_t^j\Delta_{m+j}f|^2dt\right)^{1/2}
+l^{1/2}\left(\int_{I_m}|\partial_t\tilde A_t^j\Delta_{m+j}f|^2dt\right)^{1/2}.
\end{split}
\end{equation*}
For smooth compactly supported $f$ with $f=f^*$, the map $t\mapsto\tilde A_t^j\Delta_{m+j}f$ is a self-adjoint $C^1$ function on $I_m$. Apply Lemma \ref{lem:Sobolev} to this function and its negative. The frequency identity (\ref{eq:frequency-identity}) gives
\begin{equation*}
-a_{j,m}\leq\tilde A_t^j(f)\leq a_{j,m},\qquad t\in I_m.
\end{equation*}
Using (\ref{eq:A_t^j small t L2 bounds}), $|I_m|=2^{-m}$ and $\int_{I_m}t^{-2}dt=2^{m-1}$, we obtain
\begin{equation*}
\begin{split}
\Norm{a_{j,m}}_{L^2(\A^n)}
&\lesssim\left(l^{-1/2}2^{-m/2}+l^{1/2}2^{m/2+j}\right)
2^{-j(n-1)/2}\Norm{\Delta_{m+j}f}_{L^2(\A^n)}\\
&\lesssim2^{-j(n-2)/2}\Norm{\Delta_{m+j}f}_{L^2(\A^n)}.
\end{split}
\end{equation*}

\par The preceding estimate and (\ref{eq:Delta square L2 bound}) imply that $\sum_{m\geq m_0+1}a_{j,m}^2$ converges in $L^1(\A^n)$. Set
\begin{equation*}
a_j=\Bigg(\sum_{m\geq m_0+1}a_{j,m}^2\Bigg)^{1/2}.
\end{equation*}
For every $m$, we have $a_{j,m}^2\leq a_j^2$ and therefore $a_{j,m}\leq a_j$. The interval bounds therefore give
\begin{equation*}
-a_j\leq\tilde A_t^j(f)\leq a_j,\qquad 0<t\leq t_0.
\end{equation*}
Moreover,
\begin{equation*}
\begin{split}
\Norm{a_j}_{L^2(\A^n)}^2
&=\sum_{m\geq m_0+1}\Norm{a_{j,m}}_{L^2(\A^n)}^2\\
&\lesssim2^{-j(n-2)}\sum_{m\geq m_0+1}\Norm{\Delta_{m+j}f}_{L^2(\A^n)}^2
\lesssim2^{-j(n-2)}\Norm{f}_{L^2(\A^n)}^2.
\end{split}
\end{equation*}
Thus $a_j$ is a desired common positive majorant in $L^2(\A^n)$. The characterization in Appendix \ref{subsec:nc-vector-valued-spaces} proves the maximal estimate for self-adjoint $f$. Splitting a general $f$ into real and imaginary parts gives
\begin{equation}
\label{eq:tilde A small t L2 maximal}
\Big\|\sup_{0<t\leq t_0}\tilde A_t^j(f)\Big\|_{L^2(\A^n)}
\lesssim2^{-j(n-2)/2}\Norm{f}_{L^2(\A^n)}.
\end{equation}

\medskip

\par We finish by interpolation. Fix $n/(n-1)<p<2$ and choose $1<p_1<p$. Let
\begin{equation*}
\frac1p=\frac{1-\theta}{2}+\frac\theta{p_1},\qquad
\delta'=-(1-\theta)\frac{n-2}{2}+\theta.
\end{equation*}
As $p_1\downarrow1$, we have $\delta'\to1-n+n/p<0$. Hence $p_1$ can be chosen so that $\delta'<0$. Fix a finite set $J\subset(0,t_0]$. Apply the vector-valued interpolation argument from Section \ref{subsec:hypersurface-large-t} to the linear map $f\mapsto\{\tilde A_t^j(f)\}_{t\in J}$, using (\ref{eq:tilde A small t L2 maximal}) and (\ref{eq:tilde A small t crude maximal}). The constants are independent of $J$. Passing to the full indices gives
\begin{equation*}
\Big\|\sup_{0<t\leq t_0}\tilde A_t^j(f)\Big\|_{L^p(\A^n)}
\lesssim_p2^{j\delta'}\Norm{f}_{L^p(\A^n)}.
\end{equation*}
Since $\delta'<0$, the series $\sum_{j\geq1}\{\tilde A_t^j(f)\}_{0<t\leq t_0}$ converges in $L^p(\A^n;\ell^\infty((0,t_0]))$. Adding the low-frequency estimate (\ref{eq:tilde A zero maximal}) and using $A_t=\sum_{j\geq0}\tilde A_t^j$, we obtain
\begin{equation}\label{e:30A_tsmall}
\Big\|\sup_{0<t\leq t_0}A_t(f)\Big\|_{L^p(\A^n)}
\lesssim_p\Norm{f}_{L^p(\A^n)},\quad \frac{n}{n-1}<p\leq2.
\end{equation}
For $p=2$, sum (\ref{eq:tilde A small t L2 maximal}) directly, since $n\geq3$. Summing over the finite partition gives (\ref{e:30A_tsmall}) for the original average. Combining (\ref{e:30A_tsmall}) with (\ref{eq:A_t large t Lp maximal}) proves the maximal bound on $0<t\leq1$ for $n/(n-1)<p\leq2$.

\par To cover $p>2$, we use the total variation bound for the surface measures. On each graph patch,
\begin{equation*}
\int_{\R^n}|\tilde\psi_t(x,y)|\delta(\tilde\Phi_t(x,y))dy
=\int_{\R^{n-1}}|\tilde\psi(t,x,-z')|dz'\lesssim1.
\end{equation*}
For $t_0\leq t\leq1$, compactness and $\nabla_y\Phi_t\neq0$ on the zero set give the same bound. Thus
\begin{equation*}
\Big\|\sup_{0<t\leq1}A_t(f)\Big\|_{L^\infty(\A^n)}\lesssim\Norm{f}_{L^\infty(\A^n)}.
\end{equation*}
Interpolation with the $L^2$ maximal bound proves the result for $2<p<\infty$. For every finite $p$ in the stated range, the bounded map extends from smooth functions to $L^p(\A^n)$ by density. This completes the proof of Theorem \ref{thm:A_t bounded for p>n/(n-1)}.

\section{Applications to ergodic theory}
\label{sec:ergodic-theory}
\par We now apply the maximal inequalities to continuous-parameter noncommutative ergodic theory. We prove Theorems \ref{thm:sup N_r(a) Lp bounded} and \ref{thm:sup N_r^P(a) Lp bounded} as applications of the $L^p$ bounds for operator-valued maximal averages in Theorems \ref{thm:k-submanifold case} and \ref{thm:polynomial case}.
We then establish the corresponding individual ergodic theorems in the noncommutative setting.
\subsection{Maximal ergodic theorem}
\label{subsec:maximal-ergodic-theorem}
\par Suppose we are given a $w^*$-continuous group action of $\R^n$ on $\M$:
\begin{equation*}
	\begin{split}
		\R^n \times \M \quad &\longrightarrow \quad \M \\
		(s,a) \quad &\longmapsto  \quad T_s(a)
	\end{split}
\end{equation*}
where the operator group $\{T_s:\M\to\M\}_{s\in\R^n}$ satisfies
\begin{enumerate}
	\item[(i)] Each $T_s$ is a (positive) trace-preserving automorphism on $\M$.
	\item[(ii)] For each $x\in\M$, the map $s\mapsto T_s(x)$ is $w^*$-continuous.
\end{enumerate}
Each $T_s$ acts isometrically on $L^p(\M)$.
Let $\gamma:\R^k\to\R^n$ be a $k$-dimensional submanifold in $\R^n$ with $\gamma(0)=0$.
For each $r>0$ define the ergodic mean $N_r:\M\to\M$ to be the average of $T_s$ along $\gamma$, i.e.
\begin{equation*}
	N_r(a):=\frac1{|B_k(r)|}\int_{|t|<r}T_{\gamma(t)}(a)dt.
\end{equation*}
Also recall that $M_r$, acting on $\M$-valued functions, is defined by
\begin{equation*}
	M_r(\vphi)(x)=\frac1{|B_k(r)|}\int_{|t|<r}\vphi(x-\gamma(t))dt.
\end{equation*}
A transference argument (as in \cite{Bourgain1988a}, \cite{Bourgain1988b}) yields the following proposition.
\begin{prop}
	\label{prop:transference to ergodic}
	Let $1\leq p\leq\infty$ and let $\mathcal R\subset(0,\infty)$ be a set of scales. Assume that there exists $A>0$ such that
	\begin{equation}
		\label{eq:sup_r M_r(phi) Lp bound}
		\Big\|{\sup_{r\in\mathcal R}M_r(\vphi)}\Big\|_{L^p(\A)} \leq A\Norm{\vphi}_{L^p(\A)}
	\end{equation}
	for all $\vphi\in L^p(\A)$. Then
	\begin{equation}
		\label{eq:sup_r N_r(a) L^p bound}
	\Big\|{\sup_{r\in\mathcal R}N_r(a)}\Big\|_{L^p(\M)} \leq A\Norm{a}_{L^p(\M)}
	\end{equation}
	holds for all $a\in L^p(\M)$.
\end{prop}
\begin{proof}
	Assume the inequality (\ref{eq:sup_r M_r(phi) Lp bound}) holds. Fix $R>0$ and put $\mathcal R_R=\mathcal R\cap(0,R)$. We first consider $1\leq p<\infty$. To prove (\ref{eq:sup_r N_r(a) L^p bound}) on $\mathcal R_R$, we let $\vphi(x)=T_{-x}(a)$, $x\in\R^n$. Then
	\begin{equation*}
		\begin{split}
			M_r(\vphi)(x)&=\frac1{|B_k(r)|}\int_{|t|<r}\vphi(x-\gamma(t))dt \\
			&=\frac1{|B_k(r)|}\int_{|t|<r}T_{\gamma(t)}(T_{-x}a)dt
=N_r(T_{-x}a).
		\end{split}
	\end{equation*}
Now choose $J>R_\gamma$, where
\begin{equation*}
	R_\gamma:=\sup_{|t|<R}|\gamma(t)|.
\end{equation*}
Let $\psi(x)=\vphi(x)\chi_{\{|x|\leq J\}}(x)$. Then (\ref{eq:sup_r M_r(phi) Lp bound}) gives
\begin{equation}
	\label{eq:sup_r<R M_r(psi) Lp}
	\Big\|{\sup_{r\in\mathcal R_R}M_r(\psi)}\Big\|_{L^p(\A)}^p\leq \Big\|{\sup_{r\in\mathcal R}M_r(\psi)}\Big\|_{L^p(\A)}^p \leq A^p\Norm{\psi}_{L^p(\A)}^p.
\end{equation}
Note that, for $x\in \R^n$ satisfying $|x|<J-R_\gamma$ and $r\in\mathcal R_R$, we have
\begin{equation*}
	M_r(\vphi)(x)=M_r(\psi)(x).
\end{equation*}
Consequently,
\begin{equation}
	\label{eq:sup_r<R M_r(phi)(x) L^p}
	\begin{split}
		\int_{|x|<J-R_\gamma}\Big\|{\sup_{r\in\mathcal R_R}M_r(\vphi)(x)}\Big\|_{L^p(\M)}^p dx
		&=\int_{|x|<J-R_\gamma}\Big\|{\sup_{r\in\mathcal R_R}M_r(\psi)(x)}\Big\|_{L^p(\M)}^p dx \\
		&\leq \int_{\R^n}\Big\|{\sup_{r\in\mathcal R_R}M_r(\psi)(x)}\Big\|_{L^p(\M)}^p dx \\
		&\leq \Big\|{\sup_{r\in\mathcal R_R} M_r(\psi)}\Big\|_{L^p(\A)}^p
	\end{split}
\end{equation}
and
\begin{equation}
	\label{eq:phi(x) L^p}
	\Norm{\psi}_{L^p(\A)}^p=\int_{|x|<J}\Norm{\vphi(x)}_{L^p(\M)}^p dx.
\end{equation}

Since each $T_{-x}$ is trace-preserving and positive, for each $x$ we have
\begin{equation}
	\label{eq:trace invarient}
	\Norm{\vphi(x)}_{L^p(\M)}=\Norm{T_{-x}(a)}_{L^p(\M)}=\Norm{a}_{L^p(\M)},
\end{equation}
and moreover, it follows directly from the definition of the $L^p(\M;\ell^\infty)$-norm that
\begin{equation}
	\label{eq:trac invarient supremum}
	\Big\|{\sup_{r\in\mathcal R_R}M_r(\vphi)(x)}\Big\|_{L^p(\M)}=\Big\|{\sup_{r\in\mathcal R_R}N_r(T_{-x}(a))}\Big\|_{L^p(\M)}=\Big\|{\sup_{r\in\mathcal R_R}N_r(a)}\Big\|_{L^p(\M)}.
\end{equation}
More precisely, for any decomposition $N_r(T_{-x}(a))=uy_rv$, with $u,v\in L^{2p}(\M)$ and $y_r\in L^\infty(\M)$, we have $N_r(a)=T_x(uy_rv)=T_x(u)T_x(y_r)T_x(v)$ with $\Norm{T_x(u)}_{2p}=\Norm{u}_{2p}$, $\Norm{T_x(v)}_{2p}=\Norm{v}_{2p}$ and $\Norm{T_x(y_r)}_\infty=\Norm{y_r}_\infty$.
This gives
\begin{equation*}
	\Big\|{\sup_{r\in\mathcal R_R}N_r(a)}\Big\|_{L^p(\M)} \leq \Big\|{\sup_{r\in\mathcal R_R}N_r(T_{-x}(a))}\Big\|_{L^p(\M)}.
\end{equation*}
The reverse inequality can be proved similarly.

From (\ref{eq:trace invarient}) and (\ref{eq:trac invarient supremum}) we find that $\Norm{\vphi(x)}_{L^p(\M)}$ and $\Norm{\sup_{r\in\mathcal R_R}M_r(\vphi)(x)}_{L^p(\M)}$ are independent of $x$.
Combining this with (\ref{eq:sup_r<R M_r(psi) Lp}), (\ref{eq:sup_r<R M_r(phi)(x) L^p}) and (\ref{eq:phi(x) L^p}) we obtain
\begin{equation*}
	\Big\|{\sup_{r\in\mathcal R_R}N_r(a)}\Big\|_{L^p(\M)}^p \leq A^p\frac{J^n}{(J-R_\gamma)^n}\Norm{a}_{L^p(\M)}^p.
\end{equation*}
Letting $J\to\infty$ proves the required estimate on $\mathcal R_R$ for $p<\infty$. For $p=\infty$, the same cutoff identity on $|x|<J-R_\gamma$, together with the isometry of $T_{-x}$, gives directly
\begin{equation*}
	\sup_{r\in\mathcal R_R}\Norm{N_r(a)}_\infty
	\leq\Big\|{\sup_{r\in\mathcal R}M_r(\psi)}\Big\|_{L^\infty(\A)}
	\leq A\Norm{a}_\infty.
\end{equation*}
Every finite subset of $\mathcal R$ is contained in some $\mathcal R_R$. The bounds above are independent of $R$, so the finite-subfamily characterization of the maximal norm yields (\ref{eq:sup_r N_r(a) L^p bound}).
\end{proof}
\begin{proof}[Proof of Theorem \ref{thm:sup N_r(a) Lp bounded} and Theorem \ref{thm:sup N_r^P(a) Lp bounded}]
	\par By Proposition \ref{prop:transference to ergodic}, with $\mathcal R=(0,1)$ and $\mathcal R=(0,\infty)$ respectively, we obtain Theorem \ref{thm:sup N_r(a) Lp bounded} and Theorem \ref{thm:sup N_r^P(a) Lp bounded} from Theorem \ref{thm:k-submanifold case} and Theorem \ref{thm:polynomial case}.
\end{proof}

\par We next consider the variable hypersurface averages of Section \ref{sec:variable-hypersurfaces}. In this case the averaging measure depends on the base point, and we retain the spatial variable in the ergodic formulation. Recall that $\A=L^\infty(\R^n)\ot\M$, with trace $\nu=\int_{\R^n}\otimes\tau$. Let $\{\mu_{t,x}\}_{0<t<1,\,x\in\R^n}$ be a measurable family of finite positive Borel measures on $\R^n$, and define
\begin{equation}
	\label{eq:variable A_t measure form}
	A_t^\mu(f)(x)=\int_{\R^n}f(x+z)d\mu_{t,x}(z).
\end{equation}
For the averages in (\ref{eq:A_t(f) on hypersurface}) with nonnegative amplitude, choose $\mu_{t,x}$ as the image of the surface measure with its amplitude under $y\mapsto y-x$. Then $A_t^\mu=A_t$.
The action $\{T_s\}_{s\in\R^n}$ induces the trace-preserving diagonal action on $\A$ given by
\begin{equation*}
	(\beta_zf)(x)=T_z(f(x+z)), \quad z\in\R^n.
\end{equation*}
We define the associated variable ergodic averages by
\begin{equation*}
	\mathcal{N}_t^\mu(f)(x)=\int_{\R^n}T_z(f(x+z))d\mu_{t,x}(z), \quad 0<t<1.
\end{equation*}
The integral formulas are initially understood on the dense classes specified in the proof below; for general $L^p$ elements we use their bounded extensions.

\begin{prop}
	\label{prop:variable ergodic maximal}
	Let $0<p<\infty$. Suppose that the averages in (\ref{eq:variable A_t measure form}) satisfy
	\begin{equation}
		\label{eq:variable A_t maximal hypothesis}
		\Big\|\sup_{0<t<1}A_t^\mu(f)\Big\|_{L^p(\A)}\leq C_p\Norm{f}_{L^p(\A)}.
	\end{equation}
	Then, for every $f\in L^p(\A)$,
	\begin{equation*}
		\Big\|\sup_{0<t<1}\mathcal{N}_t^\mu(f)\Big\|_{L^p(\A)}\leq C_p\Norm{f}_{L^p(\A)}.
	\end{equation*}
	The constant is the same as in (\ref{eq:variable A_t maximal hypothesis}) and is independent of the action $\{T_s\}$. Both maximal norms are taken over the entire index set $(0,1)$.
\end{prop}
\begin{proof}
	Define the trace-preserving $*$-automorphism $U$ of $\A$ by
	\begin{equation*}
		(Uf)(x)=T_x(f(x)), \qquad (U^{-1}f)(x)=T_{-x}(f(x)).
	\end{equation*}
	The continuity of the action ensures that these maps are well defined on measurable operator-valued functions. Since each $T_x$ preserves $\tau$, $U$ preserves $\nu$ and induces an isometry on $L^p(\A)$. Applying $U$ and $U^{-1}$ to the factors in Definition \ref{def:NC maximal} also gives an isometry on $L^p(\A;\ell^\infty((0,1)))$.
	\par We have the conjugacy relation
	\begin{equation}
		\label{eq:variable ergodic conjugacy}
		\mathcal{N}_t^\mu=U^{-1}A_t^\mu U, \quad 0<t<1.
	\end{equation}
	To justify the formula, put
	\begin{equation*}
		\mathscr{E}=\operatorname{span}\{\varphi\otimes a:\varphi\in C_c^\infty(\R^n),\ a\in\M\cap L^p(\M)\}.
	\end{equation*}
	This is a dense subspace of $L^p(\A)$ on which the formula for $A_t^\mu$ is initially defined. For $f\in \mathscr{E}$, we have
	\begin{equation*}
		\begin{split}
			(U^{-1}A_t^\mu U f)(x)
			&=T_{-x}\left(\int_{\R^n}T_{x+z}(f(x+z))d\mu_{t,x}(z)\right) \\
			&=\int_{\R^n}T_z(f(x+z))d\mu_{t,x}(z)
			=\mathcal{N}_t^\mu(f)(x).
		\end{split}
	\end{equation*}
	Thus $\mathcal{N}_t^\mu$ has a bounded extension to $L^p(\A)$ satisfying (\ref{eq:variable ergodic conjugacy}).
	Using the isometric properties of $U$ and (\ref{eq:variable A_t maximal hypothesis}), we obtain
	\begin{equation*}
		\Big\|\sup_{0<t<1}\mathcal{N}_t^\mu(f)\Big\|_{L^p(\A)}
		=\Big\|\sup_{0<t<1}A_t^\mu(Uf)\Big\|_{L^p(\A)}
		\leq C_p\Norm{Uf}_{L^p(\A)}=C_p\Norm{f}_{L^p(\A)}.
	\end{equation*}
	The proposition is proved.
\end{proof}
\begin{proof}[Proof of Theorem \ref{thm:variable hypersurface ergodic maximal}]
With $d\sigma_{x,t}(y)=\psi_t(x,y)\delta(\Phi_t(x,y))dy$ as in (\ref{eq:A_t(f) on hypersurface}), define
\begin{equation*}
	\mathcal{N}_t(f)(x):=\int_{S_{x,t}}T_{y-x}(f(y))d\sigma_{x,t}(y), \quad 0<t<1.
\end{equation*}
For the corresponding pushforward measures, $A_t^\mu=A_t$ and $\mathcal{N}_t^\mu=\mathcal{N}_t$. Theorem \ref{thm:A_t bounded for p>n/(n-1)} supplies (\ref{eq:variable A_t maximal hypothesis}) for these averages when $n\geq3$ and $n/(n-1)<p<\infty$.
The conclusion now follows from Proposition \ref{prop:variable ergodic maximal}.
\end{proof}

\subsection{Individual ergodic theorems}
\label{subsec:individual-ergodic-theorems}
\par We refer to Appendix \ref{subsec:bau convergence} for the definition and basic facts concerning bilateral almost uniform convergence (denoted by b.a.u. for short) and the spaces $L^p(\M;c_0)$.
Our proof follows the approach of Junge and Xu; see \cite{JungeXu2007}.

\subsubsection{Small-scale ergodic theorems}
\label{subsubsec:small-scale-ergodic}
\par Recall that the group of operators $(T_s)_{s\in\R^n}$, the submanifold $\gamma$, and the ergodic mean $(N_r)_{r>0}$ were defined in the preceding subsection.
Their means are normalized by $|B_k(r)|$ throughout.

\begin{corollary}
	\label{cor:N_r bau as r to 0}
	Let $1<p<\infty$. Then for all $x\in L^p(\M)$, $N_r(x)$ converges to $x$ bilaterally almost uniformly as $r\to0$.
\end{corollary}
\begin{proof}
	First note that the $w^*$-continuity of $\{T_s\}_{s\in\R^n}$ on $\M$ implies that $\{T_s\}$ is strongly continuous on $L^p(\M)$ for $1\leq p<\infty$ (i.e., for any $x\in L^p$, $s\mapsto T_s(x)$ is continuous from $\R^n$ to $L^p(\M)$).
	Define
	\begin{equation*}
		H_r(x):=\frac1{|B_n(r)|}\int_{B_n(r)}T_s(x)ds,
	\end{equation*}
	where $B_n(r)$ is the $n$-dimensional ball centered at $0$ with radius $r$. Then it is easy to see that $H_r(x)$ converges to $x$ in $L^p(\M)$-norm.
	\par By Proposition \ref{prop:continuous-c0}, it suffices to prove $\{N_r(x)-x\}_{0<r\leq1}\in L^p(\M;c_0((0,1]))$. Theorem \ref{thm:sup N_r(a) Lp bounded} and the contractivity of $N_1$ give
	\begin{equation}
		\label{eq:N_r difference maximal bound}
		\Big\|\sup_{0<r\leq1}(N_r(x)-x)\Big\|_p\lesssim_p\Norm{x}_p.
	\end{equation}
	We first take $x=H_{r_0}(y)$ for some $y\in L^1(\M)\cap\M$ and $0<r_0<1$. Since $H_{r_0}$ is contractive on $L^q(\M)$ for $1\leq q\leq\infty$, such an $x$ belongs to $L^q(\M)\cap\M$ for every $1\leq q<\infty$.
	Fix such an $x$. We show that
	\begin{equation}
		\label{eq:N_r(x) tend to x infty norm}
		\Norm{N_r(x)-x}_{\infty}\to 0 \quad \text{ as } r\to0.
	\end{equation}
	Indeed, for $0<|t|<r<r_0$, we have
	\begin{equation*}
		N_r(x)-x=\frac1{|B_k(r)|}\int_{B_k(r)}(T_{\gamma(t)}(x)-x)dt,
	\end{equation*}
	Writing $B=B_n(r_0)$, cancellation on the intersection of $B$ and $\gamma(t)+B$ gives
	\begin{equation*}
		\begin{split}
			T_{\gamma(t)}(x)-x
			=\frac1{|B|}\int_{(\gamma(t)+B)\setminus B}T_s(y)ds 
			-\frac1{|B|}\int_{B\setminus(\gamma(t)+B)}T_s(y)ds.
		\end{split}
	\end{equation*}
	Since each $T_s$ is isometric on $\M$, it follows that
	\begin{equation*}
		\begin{split}
			\Norm{T_{\gamma(t)}(x)-x}_\infty
			&\leq\frac{|B\triangle(\gamma(t)+B)|}{|B|}\Norm{y}_\infty \\
			&\lesssim\frac{2|B_{n-1}(r_0)||\gamma(t)|}{|B_n(r_0)|}\Norm{y}_\infty,
			\quad \text{ as } |t|\to0.
		\end{split}
	\end{equation*}
	The notation $A\triangle B:=(A\setminus B)\cup(B\setminus A)$ is the symmetric difference of $A$ and $B$. In the last step we use the asymptotic formula for the measure of the symmetric difference
	\begin{equation*}
		|B_n(r)\triangle (v+B_n(r))| \sim 2|B_{n-1}(r)||v|, \quad \text{ as } v\to 0.
	\end{equation*}
	Therefore, by Minkowski's inequality, we have
	\begin{equation*}
		\Norm{N_r(x)-x}_\infty \lesssim \frac{2|B_{n-1}(r_0)|}{|B_n(r_0)|}\Norm{T_s}\Norm{y}_\infty \frac1{|B_k(r)|}\int_{B_k(r)}|\gamma(t)|dt, \quad \text{ as } r\to 0.
	\end{equation*}
	Note that ${|B_{n-1}(r_0)|}/{|B_n(r_0)|}$ is a constant, and the factor $\Norm{T_s}$ is bounded since each $T_s$ is automorphic and hence isometric.
	Since $\gamma$ is continuous and $\gamma(0)=0$, the average of $|\gamma(t)|$ over $|t|<r$ tends to $0$. This proves (\ref{eq:N_r(x) tend to x infty norm}).
	Next, choose $q\in(1,p)$. By (\ref{eq:J subset I maximal norm}) and (\ref{eq:N_r difference maximal bound}) with $p$ replaced by $q$, we have
	\begin{equation*}
		\Big\|\sup_{0<r<\varepsilon}(N_r(x)-x)\Big\|_q
		\leq\Big\|\sup_{0<r\leq1}(N_r(x)-x)\Big\|_q\lesssim_q\Norm{x}_q.
	\end{equation*}
	Applying complex interpolation to finite subfamilies and using the finite-subfamily characterization of the maximal norm, we obtain, for $0<\varepsilon<1$,
	\begin{equation*}
		\begin{split}
			\Big\|\sup_{0<r<\varepsilon}(N_r(x)-x)\Big\|_p
			&\leq\left(\sup_{0<r<\varepsilon}\Norm{N_r(x)-x}_\infty\right)^{1-\frac qp}
			\Big\|\sup_{0<r<\varepsilon}(N_r(x)-x)\Big\|_q^{\frac qp} \\
			&\lesssim_q\left(\sup_{0<r<\varepsilon}\Norm{N_r(x)-x}_\infty\right)^{1-\frac qp}\Norm{x}_q^{\frac qp}\to0
		\end{split}
	\end{equation*}
	as $\varepsilon\downarrow0$, by (\ref{eq:N_r(x) tend to x infty norm}). Definition \ref{def:continuous-c0} now gives $\{N_r(x)-x\}_{0<r\leq1}\in L^p(\M;c_0((0,1]))$ for $x=H_{r_0}(y)$.
	\par Finally, the elements $H_{r_0}(y)$ with $y\in L^1(\M)\cap\M$ and $0<r_0<1$ are dense in $L^p(\M)$, by the strong convergence of $H_{r_0}$ and the density of $L^1(\M)\cap\M$. For general $x\in L^p(\M)$, linearity and (\ref{eq:N_r difference maximal bound}) give
	\begin{equation*}
		\begin{split}
			\Big\|\sup_{0<r\leq1}\big((N_r(x)-x)-(N_r(H_{r_0}(y))-H_{r_0}(y))\big)\Big\|_p 
			\lesssim_p\Norm{x-H_{r_0}(y)}_p.
		\end{split}
	\end{equation*}
	By the closedness in Proposition \ref{prop:continuous-c0}, we conclude that
	\begin{equation*}
		\{N_r(x)-x\}_{0<r\leq1}\in L^p(\M;c_0((0,1])), \qquad \forall\ x\in L^p(\M).
	\end{equation*}
	Proposition \ref{prop:continuous-c0} then gives $N_r(x)\to x$ b.a.u. as $r\to0$, as required.
\end{proof}

\par We can also consider the case when $\gamma(t)=P(t)$ is given by polynomials with $P(0)=0$, using the normalized means $N_r^P$ defined in Section \ref{sec:introduction}.
We first prove b.a.u. convergence as $r\to0$. Large-scale convergence along $r=2^m$, under the additional finite-type assumption, is proved in Section \ref{subsubsec:large-scale-ergodic}.
\begin{corollary}
	Let $1<p<\infty$. Then for all $x\in L^p(\M)$, $N_r^P(x)$ converges to $x$ bilaterally almost uniformly as $r\to0$.
\end{corollary}
\begin{proof}
	The proof is the same as that of Corollary \ref{cor:N_r bau as r to 0}.
\end{proof}

\medskip

\par For the general measure families, $A_t^\mu$ and $\mathcal{N}_t^\mu$ also converge at small scales, provided that the measures have mass one and their supports shrink uniformly to the origin. The mass-one condition is an additional assumption; the cutoff measures in (\ref{eq:A_t(f) on hypersurface}) need not satisfy it.
\begin{corollary}
	\label{cor:variable ergodic bau at zero}
	Under the assumptions of Proposition \ref{prop:variable ergodic maximal}, suppose in addition that there exists $R>0$ such that, for every $0<t<1$ and almost every $x\in\R^n$,
	\begin{equation}
		\label{eq:variable measures normalization support}
		\mu_{t,x}(\R^n)=1, \qquad \supp\mu_{t,x}\subseteq B_n(Rt).
	\end{equation}
	Then, for every $f\in L^p(\A)$, we have
	\begin{equation*}
		\begin{split}
			A_t^\mu(f)\overset{b.a.u.}{\longrightarrow} f, \qquad t\to0^+,
		\end{split}
	\end{equation*}
	and
	\begin{equation*}
		\begin{split}
			\mathcal{N}_t^\mu(f)\overset{b.a.u.}{\longrightarrow} f, \qquad t\to0^+.
		\end{split}
	\end{equation*}
\end{corollary}
\begin{proof}
	We first prove the corresponding convergence for $A_t^\mu$. Let $\mathscr{E}$ be the dense subspace defined in the proof of Proposition \ref{prop:variable ergodic maximal}. For $g=\varphi\otimes a\in\mathscr{E}$, (\ref{eq:variable measures normalization support}) gives
	\begin{equation*}
		\Norm{A_t^\mu(g)-g}_\infty
		\leq\Norm{a}_\infty\sup_{x\in\R^n,\,|z|\leq Rt}|\varphi(x+z)-\varphi(x)|\to0
		\quad \text{ as } t\to0.
	\end{equation*}
	By linearity, the same convergence holds for every $g\in\mathscr{E}$.
	\par We now extend the convergence to all of $L^p(\A)$.
	By (\ref{eq:variable A_t maximal hypothesis}) and the triangle inequality,
	\begin{equation*}
		\Big\|\sup_{0<t<1}\big(A_t^\mu(h)-h\big)\Big\|_{L^p(\A)}\leq(C_p+1)\Norm{h}_{L^p(\A)}.
	\end{equation*}
	Consequently, for every $h\in L^p(\A)$ and $\lambda>0$, there is a projection $q\in\A$ satisfying
	\begin{equation}
		\label{eq:variable averages projection estimate}
		\begin{split}
			\nu(1-q)&\leq \tilde{C}_p\lambda^{-p}\Norm{h}_{L^p(\A)}^p, \\
			\quad \Norm{q\big(A_t^\mu(h)-h\big)q}_\infty&\leq\lambda, \quad \forall \ 0<t<1,
		\end{split}
	\end{equation}
	where $\tilde{C}_p=2^{p+1}(C_p+1)^p$. Indeed, Definition \ref{def:NC maximal} provides a factorization $A_t^\mu(h)-h=ay_tb$ with $a,b\in L^{2p}(\A)$ and
	\begin{equation*}
		\sup_{0<t<1}\Norm{y_t}_\infty\leq1, \qquad
		\Norm{a}_{2p}^2=\Norm{b}_{2p}^2\leq2(C_p+1)\Norm{h}_{L^p(\A)},
	\end{equation*}
	after rescaling the factors. Set
	\begin{equation*}
		q=\chi_{[0,\sqrt{\lambda}]}(|a^*|)\wedge\chi_{[0,\sqrt{\lambda}]}(|b|).
	\end{equation*}
	Then $\max\{\Norm{qa}_\infty,\Norm{bq}_\infty\}\leq\sqrt{\lambda}$, which gives the second part of (\ref{eq:variable averages projection estimate}). The first part follows from Chebyshev's inequality:
	\begin{equation*}
		\nu(1-q)\leq\lambda^{-p}\left(\Norm{a}_{2p}^{2p}+\Norm{b}_{2p}^{2p}\right)
		\leq \tilde{C}_p\lambda^{-p}\Norm{h}_{L^p(\A)}^p.
	\end{equation*}
	\par Fix $g\in L^p(\A)$ and $\varepsilon>0$. By density, choose $g_k\in\mathscr{E}$, $k\geq1$, such that
	\begin{equation*}
		\tilde{C}_p2^{kp}\Norm{g-g_k}_{L^p(\A)}^p<\varepsilon2^{-k}.
	\end{equation*}
	Apply (\ref{eq:variable averages projection estimate}) to $h=g-g_k$ and $\lambda=2^{-k}$ to obtain projections $q_k$, and put $e=\bigwedge_{k\geq1}q_k$. Then $\nu(1-e)\leq\sum_{k\geq1}\nu(1-q_k)<\varepsilon$, and for every $k\geq1$,
	\begin{equation*}
		\limsup_{t\to0}\Norm{e(A_t^\mu(g)-g)e}_\infty
		\leq2^{-k}+\limsup_{t\to0}\Norm{A_t^\mu(g_k)-g_k}_\infty=2^{-k}.
	\end{equation*}
	Letting $k\to\infty$ proves that $A_t^\mu(g)\to g$ b.a.u. for the full continuous parameter $t\to0$.
	\par Finally, apply this conclusion to $g=Uf$ and use (\ref{eq:variable ergodic conjugacy}). If $e$ is a projection in the b.a.u. convergence of $A_t^\mu(Uf)$, then $U^{-1}(e)$ has a complement with the same trace, and
	\begin{equation*}
		\Norm{U^{-1}(e)(\mathcal{N}_t^\mu(f)-f)U^{-1}(e)}_\infty
		=\Norm{e(A_t^\mu(Uf)-Uf)e}_\infty\to0.
	\end{equation*}
	This proves the corollary.
\end{proof}
For the hypersurface averages $A_t$ and $\mathcal{N}_t$, we obtain the following consequence.
\begin{corollary}
	Under the assumptions of Theorem \ref{thm:variable hypersurface ergodic maximal}, we have
	\begin{equation*}
		\begin{split}
			A_t(f)\overset{b.a.u.}{\longrightarrow} m_0f, \qquad t\to0^+,
		\end{split}
	\end{equation*}
	and
	\begin{equation*}
		\begin{split}
			\mathcal{N}_t(f)\overset{b.a.u.}{\longrightarrow} m_0f, \qquad t\to0^+,
		\end{split}
	\end{equation*}
	where the function $m_0(x)=\lim_{t\to0}(A_t1)(x)$.
\end{corollary}
\begin{proof}
	We will apply Corollary~\ref{cor:variable ergodic bau at zero} after adding a point mass at the origin. We first prove that $m_t=A_t1$ has a uniform limit.
	\par Let $\tilde\psi^{(\ell)}$, $1\leq\ell\leq L$, be the finitely many graph amplitudes constructed in Section \ref{subsec:hypersurface-small-t}, viewed as functions of the original base variable $x$. For all sufficiently small $t>0$,
	\begin{equation*}
		m_t(x)=\sum_{\ell=1}^L\int_{\R^{n-1}}\tilde\psi^{(\ell)}(t,x,-z')dz'.
	\end{equation*}
	Each amplitude extends smoothly to $t=0$ and has compact support in $(t,x,z')$. Hence
	\begin{equation*}
		\Norm{m_t-m_0}_\infty\to0,\qquad
		m_0(x)=\sum_{\ell=1}^L\int_{\R^{n-1}}\tilde\psi^{(\ell)}(0,x,-z')dz'.
	\end{equation*}
	In particular, $m_0$ is bounded. The compact support of the original amplitude gives $\supp\mu_{t,x}\subset B_n(Rt)$ for a fixed $R>0$, and the total variation estimate at the end of Section \ref{sec:variable-hypersurfaces} gives $M:=\sup_{0<t<1,x}\mu_{t,x}(\R^n)<\infty$.
	\par Choose $C\geq\max\{1,M\}$ and define
	\begin{equation*}
		d\mu^\sharp_{t,x}=C^{-1}\big(d\mu_{t,x}+(C-m_t(x))d\delta_0\big)\,.
	\end{equation*}
	These measures are positive, have mass one and are supported in $B_n(Rt)$. The associated averages satisfy
	\begin{equation*}
		\begin{split}
			A_t^\sharp f&=C^{-1}\big(A_tf+(C-m_t)f\big),\\
			\mathcal N_t^\sharp f&=C^{-1}\big(\mathcal N_tf+(C-m_t)f\big)\,.
		\end{split}
	\end{equation*}
	The functions $C-m_t$ are central and bounded by $C$. For $h\in L^p(\A)$ with polar decomposition $h=v|h|$, write $(C-m_t)h=(v|h|^{1/2})(C-m_t)|h|^{1/2}$. By Definition~\ref{def:NC maximal}, the family $\{(C-m_t)h\}_{0<t<1}$ has maximal $L^p$ norm at most $C\Norm{h}_{L^p(\A)}$. Together with Theorem~\ref{thm:A_t bounded for p>n/(n-1)}, this yields
	\begin{equation*}
		\Big\|\sup_{0<t<1}A_t^\sharp h\Big\|_{L^p(\A)}\leq(C_p/C+1)\Norm{h}_{L^p(\A)}\,.
	\end{equation*}
	Corollary~\ref{cor:variable ergodic bau at zero} therefore applies to $\{\mu^\sharp_{t,x}\}$ and gives
	\begin{equation*}
		A_t^\sharp f\longrightarrow f,\qquad \mathcal N_t^\sharp f\longrightarrow f
		\quad\text{b.a.u. as }t\to0^+\,.
	\end{equation*}
	\par To recover the original averages, write
	\begin{equation*}
		\begin{split}
			A_tf-m_0f&=C(A_t^\sharp f-f)+(m_t-m_0)f,\\
			\mathcal N_tf-m_0f&=C(\mathcal N_t^\sharp f-f)+(m_t-m_0)f\,.
		\end{split}
	\end{equation*}
	For $K>0$, let $e_K=\chi_{[0,K]}(|f|)$. Then $\nu(1-e_K)\leq K^{-p}\Norm{f}_{L^p(\A)}^p$ and, since $m_t-m_0$ is central,
	\begin{equation*}
		\Norm{e_K(m_t-m_0)fe_K}_\infty\leq K\Norm{m_t-m_0}_\infty\longrightarrow0\,.
	\end{equation*}
	Thus $(m_t-m_0)f\to0$ b.a.u. Intersecting $e_K$ with the projections supplied by Corollary~\ref{cor:variable ergodic bau at zero} proves both assertions.
\end{proof}

\subsubsection{Large-scale ergodic theorems}
\label{subsubsec:large-scale-ergodic}
\par The b.a.u. convergence as $r\to\infty$ is more subtle.
To state the result, we define the fixed-point subspace as follows:
\begin{equation*}
	\F_p:=\{x\in L^p(\M):T_s(x)=x \text{ for all } s\in\R^n\}, \quad 1<p<\infty.
\end{equation*}

The group of operators $(T_s)_{s\in\R^n}$ has $n$ infinitesimal generators, defined by
\begin{equation}
	\label{eq:A_j generator def}
	A_j(x):=\lim_{\varepsilon\to0}\frac i{\varepsilon} (T_{\varepsilon e_j}-\text{Id})(x)= \Big(i\frac{\partial}{\partial s_j} T_s\Big) \Big|_{s=0}(x), \quad j=1,\dots,n.
\end{equation}
For $x\in L^2(\M)$, the above limit is understood in the $L^2$ sense.
Since $s\mapsto T_s$ is a unitary representation of $\R^n$ on the Hilbert space $L^2(\M)$, each $A_j$ is (possibly) unbounded and self-adjoint on its domain. They commute strongly with each other since the group is abelian.
The space $\F_2$ is a closed subspace of $L^2(\M)$ and we have the orthogonal decomposition
\begin{equation*}
	L^2(\M)=\F_2\oplus\F_2^\perp.
\end{equation*}
Denote $(A_1,\dots,A_n)$ by $A$. We have
\begin{equation*}
	T_s=\exp(-i \sum_j s_jA_j)=\exp(-i s\cdot A).
\end{equation*}
This implies
\begin{equation*}
	\F_2=\bigcap_{j=1}^n \ker(A_j) \quad \text{ and } \quad \F_2^\perp=\overline{\sum_j \operatorname{Ran}(A_j)}^{L^2}.
\end{equation*}

For general $L^p(\M)$, the limit (\ref{eq:A_j generator def}) is understood in the $L^p$ sense, and $A_j$ is an unbounded operator on $L^p(\M)$.
With the notation
\begin{equation*}
	\F_p^\perp:= \overline{\sum_j \operatorname{Ran}(A_j)}^{L^p},
\end{equation*}
(a notation that does not assert orthogonality in $L^p(\M)$), we have the canonical direct sum decomposition
\begin{equation*}
	L^p(\M)=\F_p \oplus \F_p^{\perp}.
\end{equation*}
For simplicity, we let $F$ be the projection from $L^p(\M)$ to the subspace $\F_p$  when $p$ is clear from context.

\medskip

\par Junge and Xu \cite{JungeXu2007} prove that ergodic averages associated with positive contractive operator semigroups on $L^p(\M)$ converge b.a.u. to the projection onto the fixed-point subspace.
This raises the question of whether the same conclusion holds for ergodic averages over lower-dimensional submanifolds in $\R^n$.
That is, does $N_r^P(x)$ converge b.a.u. to $F(x)$ for $x\in L^p(\M)$?

\begin{conj}
	Let $1<p<\infty$, and let $P:\R^k\to\R^n$ be a polynomial map of finite type at $t=0$. Then for all $x\in L^p(\M)$, the ergodic mean $N_r^P(x)$ converges to $F(x)$ bilaterally almost uniformly as $r\to\infty$.
\end{conj}

\par Junge and Xu \cite{JungeXu2007} use a telescoping argument to prove large-scale b.a.u. convergence for contractive operator semigroups with indices in $[0,\infty)$.
However, this method may not apply to averages over lower-dimensional sets.
Although b.a.u. convergence as $r\to\infty$ remains unproved here,
it remains true if we restrict the family of ergodic averages to suitable subsequences.

\begin{corollary}
	Let $1<p<\infty$, and let $P:\R^k\to\R^n$ be a polynomial map of finite type at $t=0$. Then for all $x\in L^p(\M)$, $N_{2^m}^P(x)\to F(x)$ bilaterally almost uniformly when $m\in\Z$ and $m\to+\infty$.
\end{corollary}
\begin{proof}
	Let $\mu_r^P$ be the normalized pushforward measure on $\R^n$ by $P$, i.e.,
	\begin{equation*}
		\int_{\R^n} \varphi(x)d\mu_r^P(x):=\frac{1}{|B_k(r)|}\int_{|t|<r}\varphi(P(t))dt, \quad r>0.
	\end{equation*}
	Its inverse Fourier transform is
	\begin{equation*}
		\F^{-1}[\mu_r^P](\xi)=\int_{\R^n}e^{2\pi i\inn{\xi}{x}}d\mu_r^P(x)=\frac1{|B_k(r)|}\int_{|t|<r}e^{2\pi i\inn{\xi}{P(t)}}dt.
	\end{equation*}
	We claim that, for any compact set $K\subset \R^n\setminus\{0\}$, there exist constants $C_K>0$ and $\delta_K>0$ such that
	\begin{equation}
		\label{eq:Fourer decay for compact K}
		\sup_{\xi\in K}\big|\F^{-1}[\mu_r^P](\xi)\big|\leq C_Kr^{-\delta_K}.
	\end{equation}
	Indeed, by assuming $P(0)=0$, we can write
	\begin{equation*}
		P(t)=\sum_{1\leq|\alpha|\leq d} a_\alpha t^\alpha,
	\end{equation*}
	where each coefficient $a_\alpha$ is in $\R^n$. For $\xi\in K$,
	\begin{equation*}
		\inn{\xi}{P(t)}=\sum_{1\leq|\alpha|\leq d} \inn{\xi}{a_\alpha} t^\alpha.
	\end{equation*}
	Since $P$ is a polynomial of finite type, its range is not contained in a hyperplane. For each $\xi\in K$, at least one coefficient $\inn{\xi}{a_\alpha}$ is non-zero.
	By compactness of $K$, after passing to finitely many subregions, we may assume $|\inn{\xi}{a_{\alpha_0}}|\geq c_K>0$ for some fixed multi-index $\alpha_0$.
	A change of variables gives
	\begin{equation*}
		\F^{-1}[\mu_r^P](\xi)=\frac1{|B_k(1)|}\int_{|t|<1}e^{2\pi i\inn{\xi}{P(rt)}}dt.
	\end{equation*}
	The phase polynomial $Q_r(t):=\inn{\xi}{P(rt)}$ has a coefficient of size at least $c_K r^{|\alpha_0|}$, where $c_K>0$ and $|\alpha_0|\geq1$.
	By a version of the van der Corput lemma (see, for example, \cite[Corollary 7.3]{Wright1999}), we obtain the estimate (\ref{eq:Fourer decay for compact K}).
	\par For $h\in\mathcal{S}(\R^n)$ and $y\in L^p(\M)$, define $\Phi_h(y)$ to be the Bochner integral
	\begin{equation*}
		\Phi_h(y):=\int_{\R^n}h(u)T_u(y)du \ \in L^p(\M).
	\end{equation*}
	Let $\mathcal{D}$ be the subspace of $L^1(\M)\cap\M$ defined by
	\begin{equation*}
		\mathcal{D}:=\operatorname{span}\left\{\Phi_h(y): y\in L^1(\M)\cap\M,\ h\in \mathcal{S}(\R^n),\ \supp\ \widehat{h} \Subset \R^n\setminus\{0\}\right\}.
	\end{equation*}
	The notation $A\Subset B$ means that $A$ is compactly contained in $B$.
	It is easy to see that $\mathcal{D}$ is contained in $\ker(F)=\F_p^\perp$. Indeed, since $F$ is the mean ergodic projection associated with the $\R^n$-action, it commutes with $T_u$ and satisfies $F(T_u(y))=T_u(F(y))=F(y)$. Thus
	\begin{equation*}
		F(\Phi_h(y))=\int_{\R^n}h(u)F(T_u(y))du=\left(\int_{\R^n}h(u)du\right) F(y)=\widehat{h}(0)F(y)=0.
	\end{equation*}
	Moreover, we claim that $\mathcal{D}$ is an $L^p$-dense subspace of $\F_p^\perp$. We will prove this claim at the end of the proof. From now, assume this claim and continue our proof.
	\par We now show that (\ref{eq:Fourer decay for compact K}) implies the following $L^2$-estimate: For each fixed $z\in\mathcal{D}$, there exist $C>0$ and $\delta>0$ depending only on $z$ such that
	\begin{equation*}
		\Norm{N_r^P(z)}_{L^2(\M)}\leq C r^{-\delta}\Norm{z}_{L^2(\M)}.
	\end{equation*}
	In fact, recall that $T_s$ can be extended to a bounded operator on $L^2(\M)$, and therefore $s\to T_s$ is a unitary representation of $\R^n$ on $L^2(\M)$.
	By Stone's theorem for unitary groups, there exists a spectral measure $dE$ on $\R^n$ such that
	\begin{equation*}
		T_s=\int_{\R^n} e^{2\pi i\inn{s}{\xi}} dE(\xi).
	\end{equation*}
	Then we have
	\begin{equation*}
		\Phi_h=\int_{\R^n}\widehat{h}(-\xi)dE(\xi).
	\end{equation*}
	Choose a compact set $K\subset\R^n\setminus\{0\}$ containing $-\supp(\widehat{h})$. Since the spectral multiplier is $\widehat{h}(-\xi)$, $z=\Phi_h(y)$ satisfies $z=E(K)z$.
	Then for $z=\Phi_h(y)$ and $y\in L^1(\M)\cap\M$, we have
	\begin{equation*}
		\begin{split}
			N_r^P(z)
			&= \frac1{|B_k(r)|}\int_{|t|<r} T_{P(t)}(z)dt = \int_{\R^n} T_s(z) d\mu_r^P(s) \\
			&=\left[\int_{\R^n}\int_{\R^n}e^{2\pi i\inn{s}{\xi}}dE(\xi) d\mu_r^P(s)\right](z) \\
			&=\left[\int_{\R^n}\F^{-1}[\mu_r^P](\xi)dE(\xi)\right](z).
		\end{split}
	\end{equation*}
	This implies
	\begin{equation}
		\label{eq:L2 estimate for N_r^P(z)}
		\begin{split}
		\Norm{N_r^P(z)}_2
		&\leq \Norm{N_r^P E(K)}_{B(L^2(\M))}\Norm{z}_2 \\
		&\leq \sup_{\xi\in K}\big|\F^{-1}[\mu_r^P](\xi)\big|\Norm{z}_2
		\lesssim_K r^{-\delta_K}\Norm{z}_2.
		\end{split}
	\end{equation}
	Here the compact set $K$ depends only on $h$, hence on $z\in\mathcal{D}$.
	\par We next prove that
	\begin{equation}
		\label{eq:N^P_2^m(z) in L^p(c_0)}
		\{N_{2^m}^P(z)\}_{m\geq0} \in L^p(\M;c_0),\quad \forall\ z\in \mathcal{D}.
	\end{equation}
	By (\ref{eq:L2 estimate for N_r^P(z)}) we have
	\begin{equation*}
		\Norm{N_{2^m}^P(z)}_2 \lesssim 2^{-m\delta}\Norm{z}_2.
	\end{equation*}
	Recall that we have the following square function bound for the noncommutative maximal norm:
    \begin{equation*}
	    \Norm{\{x_m\}_{m\geq M}}_{L^2(\M;\ell^\infty)}\leq \bigg(\sum_{m\geq M}\Norm{x_m}_{L^2(\M)}^2\bigg)^{1/2}.
    \end{equation*}
	Applying this to $x_m=N_{2^m}^P(z)$, we obtain
	\begin{equation*}
		\Norm{\{N_{2^m}^P(z)\}_{m\geq M}}_{L^2(\M;\ell^\infty)}\lesssim \bigg(\sum_{m\geq M}2^{-2m\delta}\bigg)^{1/2}\Norm{z}_2\to 0 \quad \text{ as } M\to \infty.
	\end{equation*}
	Hence $\{N_{2^m}^P(z)\}_{m\geq 0}$ is in $L^2(\M;c_0)$. This proves (\ref{eq:N^P_2^m(z) in L^p(c_0)}) with $p=2$. For general $1<p<\infty$, we choose $q$ such that $2<p<q<\infty$ if $p>2$ and $1<q<p<2$ if $p<2$.
	By Theorem \ref{thm:sup N_r^P(a) Lp bounded},
	\begin{equation*}
		\Norm{\{N_{2^m}^P(z)\}_{m\geq M}}_{L^q(\M;\ell^\infty)} \leq \Norm{\{N_{r}^P(z)\}_{r>0}}_{L^q(\M;\ell^\infty)} \lesssim_{p,d} \Norm{z}_q.
	\end{equation*}
	Interpolation between the $L^2(\M;\ell^\infty)$-norm, which tends to $0$, and the $L^q(\M;\ell^\infty)$-norm, which is bounded, gives
	\begin{equation*}
		\Norm{\{N_{2^m}^P(z)\}_{m\geq M}}_{L^p(\M;\ell^\infty)} \to 0 \quad \text{ as } M\to\infty.
	\end{equation*}
	This proves (\ref{eq:N^P_2^m(z) in L^p(c_0)}) for general $1<p<\infty$.
	\par We now complete the proof of the corollary. For general $x\in L^p(\M)$, write $x=F(x)+(x-F(x))$. Since $x-F(x) \in \F_p^\perp$, by the density of the subspace $\mathcal{D}$ in $\F_p^\perp$, we can choose $z_j\in\mathcal{D}$ such that
	\begin{equation*}
		\Norm{x-F(x)-z_j}_p \to 0, \quad j\to\infty.
	\end{equation*}
	For each $z_j$, (\ref{eq:N^P_2^m(z) in L^p(c_0)}) gives
	\begin{equation*}
		\{N_{2^m}^P(z_j)\}_{m\geq0} \in L^p(\M;c_0).
	\end{equation*}
	On the other hand, by the maximal ergodic theorem,
	\begin{equation*}
		\Norm{\{N_{2^m}^P(x-F(x))-N_{2^m}^P(z_j)\}_{m\geq0}}_{L^p(\M;\ell^\infty)} \lesssim_{p,d} \Norm{x-F(x)-z_j}_p \to 0.
	\end{equation*}
	Since $L^p(\M;c_0)$ is a closed subspace of $L^p(\M;\ell^\infty)$, we conclude that
	\begin{equation*}
		\{N_{2^m}^P(x-F(x))\}_{m\geq0} \in L^p(\M;c_0).
	\end{equation*}
	Then recall that $F(x)$ is fixed by every $T_s$, hence $N_{2^m}^P(F(x))=F(x)$. Thus
	\begin{equation*}
		\{N_{2^m}^P(x)-F(x)\}_{m\geq0} \in L^p(\M;c_0).
	\end{equation*}
	By the $c_0$-criterion of b.a.u. convergence (see Proposition \ref{prop:Lp(M;c0) prop} (\ref{item:c0-bau-convergence})), this implies that $N_{2^m}^P(x)\to F(x)$ bilaterally almost uniformly as $m\to +\infty$.
	\par Finally, it remains to show that $\mathcal{D}$ is dense in $\F_p^\perp$. Let $\psi\in C_c^\infty(\R^n)$ satisfy $0\leq\psi\leq1$ and $\psi(\xi)=1$ near $\xi=0$.
	Let $\eta_{\varepsilon,R}(\xi)=\psi(\xi/R)-\psi(\xi/\varepsilon)$. Then $\supp(\eta_{\varepsilon,R})\Subset\R^n\setminus\{0\}$, where $\varepsilon$ and $R$ are positive numbers.
	For $x\in L^p(\M)$, it suffices to approximate $x-F(x)$ by elements in $\mathcal{D}$. Note that
	\begin{equation*}
		\begin{cases}
			\Phi_{\psi(\frac{\cdot}{R})^\vee}(x) \to x, & \text{ as } R\to\infty, \\
			\Phi_{\psi(\frac{\cdot}{\varepsilon})^\vee}(x) \to F(x), & \text{ as } \varepsilon\to 0,
		\end{cases}
		\quad \text{ in } L^p(\M).
	\end{equation*}
	The first limit follows from an approximate identity argument, and the second from the mean ergodic theorem.
	Let $h_{\varepsilon,R}=\check{\eta_{\varepsilon,R}}$. Then for $y\in L^1(\M)\cap\M$,
	\begin{equation*}
		\Phi_{h_{\varepsilon,R}}(y)=\Phi_{\psi(\frac{\cdot}{R})^\vee}(y)-\Phi_{\psi(\frac{\cdot}{\varepsilon})^\vee}(y) \to y-F(y) \quad \text{ in } L^p(\M), \quad \text{ as } \varepsilon\to 0, \ R\to\infty.
	\end{equation*}
	Approximating $x$ further by elements in $L^1(\M)\cap\M$ gives $\overline{\mathcal{D}}^{L^p}=\F_p^\perp$.
\end{proof}

\appendix
\section{Preliminaries on Noncommutative Analysis}
\label{app:noncommutative-analysis}
\setcounter{theorem}{0}
\renewcommand{\thetheorem}{A\arabic{theorem}}
\par In the appendix, we recall some preliminaries on noncommutative analysis, which have been used in the paper. The material recalled here is standard. For further details, we refer to \cite{PisierXu2003, JungeXu2007, Xu07}.

\subsection{Noncommutative $L^p$ spaces}
\label{subsec:nc-lp-spaces}
\par Let $(\M;\tau)$ be a semifinite von Neumann algebra equipped with a normal semifinite faithful (n.s.f. for short) trace $\tau$. Let $\mathcal{S}_{\M}^+$ be all positive operators in $\M$ with $\tau$-finite supports and $\mathcal{S}_\M$ be the linear span of $\mathcal{S}_\M^+$. We define
\[ \Norm{x}_{L^p(\M)}=\Norm{x}_p=\tau(|x|^p)^{1/p}, \quad x\in \mathcal{S}_\M. \]
One can check that $\Norm{\cdot}_p$ is a norm on $\mathcal{S}_\M$ when $1\leq p<\infty$. The noncommutative space, denoted by $L^p(\M,\tau)$, is defined to be the completion of $\mathcal{S}_\M$ under the norm $\Norm{\cdot}_p$. We also set $L^{\infty}(\M,\tau)=\M$ equipped with the operator norm.
For simplicity, we often write $L^p(\M)$ and omit the trace $\tau$ if there is no confusion.
Let $L^0(\M)$ be the algebra of $\tau$-measurable operators affiliated with $\M$. For $p\geq0$, denote by $L^p(\M)_+$ the cone of positive operators in $L^p(\M)$.
There is a partial order on $L^0(\M)$, $x\leq y$ means $y-x\in L^0(\M)_+$.
Moreover, the set of all projections in $\M$ is denoted by $\PM$.
\par In the operator-valued (also called semi-commutative) case, we consider the tensor von Neumann algebra
\[ \mathcal{A}:=L^{\infty}(\Rn)\overline{\otimes}\M, \]
equipped with the tensor trace $\nu=\int_{\Rn}\otimes\tau$. Then $(\A,\nu)$ is also a semifinite von Neumann algebra with a n.s.f. trace. Then we have the isometry $L^p(\A)\simeq L^p(\Rn;L^p(\M))$. Indeed, let $f=\sum_i a_i\chi_{E_i}$ be a $\mathcal{S}_\M$-valued simple function on $\Rn$, if we define $T(f)=\sum_i\chi_{E_i}\otimes a_i\in L^p(\A)$, then it is easy to see
\[ \Norm{T(f)}_{L^p(\A)}=\left(\int_{\Rn}\tau(|f(x)|^p)dx\right)^{1/p}=\Norm{f}_{L^p(\Rn;L^p(\M))}. \]
Therefore we view operators in $L^p(\A)$ as $p$-integrable vector-valued functions in Bochner sense.
We have the following operator-valued H\"older inequality (see \cite{Mei2007}).
\begin{lemma}
	\label{lem:Holder for operator valued}
	Let $(X,\mu)$ be a $\sigma$-finite measure space. Let $\B\subset L^0(\M)$ be a Banach $\M$-bimodule. Let $f:X\to \B_+$ and $g:X\to\mathbb{R}_+$ be positive measurable functions. Then we have
	\begin{equation*}
		\int_X f(x)g(x)d\mu(x) \leq \left(\int_{X}f(x)^pd\mu(x)\right)^{1/p}\left(\int_{X}g(x)^{p'}d\mu(x)\right)^{1/p'},
	\end{equation*}
	where $1<p<\infty$, $1/p+1/p'=1$. The integral is in the Bochner sense, and the inequality holds with respect to the partial order on $L^0(\M)$.
\end{lemma}
The Fourier transform of $f\in L^1(\A)$ is defined by
\begin{equation*}
	\hat{f}(\xi)=\F[f](\xi)=\int_{\R^n}f(x)e^{-2\pi i\inn{x}{\xi}}dx,
\end{equation*}
For $f\in L^2(\A)$, we have the operator-valued Plancherel formula
\begin{equation*}
	\Norm{f}_{L^2(\A)}=\|\hat{f}\|_{L^2(\A)}.
\end{equation*}
For more details, we refer the reader to the survey article \cite{PisierXu2003}.

\subsection{Vector-valued noncommutative $L^p$ spaces}
\label{subsec:nc-vector-valued-spaces}
\par In this subsection, we introduce two vector-valued noncommutative $L^p$ spaces: $L^p(\M;\ell^\infty)$ and $L^p(\M;\ell^1)$.

\begin{definition}
	\label{def:NC maximal}
	Let $1\leq p\leq\infty$. For any index set $I$, define the space $L^p(\M;\ell^{\infty}(I))$ to be the space of all sequences $x=\{x_i\}_{i\in I}$, which admit a factorization of the following form: there exist $a,b\in L^{2p}(\M)$ and a sequence $y_i \in L^{\infty}(\M)$ such that $x_i=ay_ib$ for each $i\in I$. The norm of $x\in L^p(\M;\ell^{\infty}(I))$ is given by
	\[ \Norm{\{x_i\}_{i\in I}}_{L^p(\M;\ell^{\infty}(I))}=\inf\Big\{\Norm{a}_{L^{2p}(\M)}\sup_{j\in I}\Norm{y_j}_{L^{\infty}(\M)}\Norm{b}_{L^{2p}(\M)}\Big\}, \]
	where the infimum is taken over all factorizations of $x$ as above. For a sequence $\{x_i\}_{i\in I}\subset L^{1,\infty}(\M)$, we also define the following quasi-norm
	\[ \Norm{\{x_i\}_{i\in I}}_{\Lambda^{1,\infty}(\M;\ell^{\infty}(I))}=\sup_{\lambda>0}\lambda\inf_{e\in\PM}\Big\{\tau(e^{\bot}):\Norm{ex_ie}_{\infty}\leq\lambda \text{ for all } i\in I \Big\}. \]
\end{definition}

From the definition, it is easy to see that, for every subset $J\subset I$, we have
\begin{equation}
	\label{eq:J subset I maximal norm}
	\Norm{\{x_i\}_{i\in J}}_{L^p(\M;\ell^\infty(J))} \leq \Norm{\{x_i\}_{i\in I}}_{L^p(\M;\ell^\infty(I))}
\end{equation}
and
\begin{equation*}
	\Norm{\{x_i\}_{i\in J}}_{\Lambda^{1,\infty}(\M;\ell^\infty(J))} \leq \Norm{\{x_i\}_{i\in I}}_{\Lambda^{1,\infty}(\M;\ell^\infty(I))}.
\end{equation*}

\par If $x=\{x_i\}_{i\in I}$ is a sequence of self-adjoint elements, then $x\in L^p(\M;\ell^{\infty}(I))$ if and only if there exists a positive element $z\in L^p(\M)$ such that $-z\leq x_i\leq z$ for all $i\in I$, and
\[ \Norm{\{x_i\}_{i\in I}}_{L^p(\M;\ell^{\infty}(I))}=\inf\left\{\Norm{z}_{L^p(\M)}: -z\leq x_i\leq z \text{ for all } i\in I\right\}, \]
\[ \Norm{\{x_i\}_{i\in I}}_{\Lambda^{1,\infty}(\M;\ell^{\infty}(I))}=\sup_{\lambda>0}\lambda\inf_{e\in\PM}\left\{\tau(e^{\bot}):-\lambda\leq ex_ie\leq \lambda \text{ for all } i\in I\right\}. \]
In the following, for $\{x_i\}_{i\in I}\subset \M$, we also use the following notions
\begin{equation}
	\label{eq:NC maximal}
	\Big\|{\sup_{i\in I}x_i}\Big\|_{L^p(\M)}:=\Norm{\{x_i\}_{i\in I}}_{L^p(\M;\ell^{\infty}(I))}.
\end{equation}
and
\begin{equation*}
	\Big\|{\sup_{i\in I}x_i}\Big\|_{L^{1,\infty}(\M)}:=\Norm{\{x_i\}_{i\in I}}_{\Lambda^{1,\infty}(\M;\ell^{\infty}(I))}.
\end{equation*}
If the index set $I$ is countable, we simply write $L^p(\M;\ell^\infty)$ and $\Lambda^{1,\infty}(\M;\ell^\infty)$.
\begin{remark}
	In the noncommutative setting, one cannot define the maximal functions by taking pointwise supremum as in the classical cases, since the cone of positive operators is not totally ordered. Definition \ref{def:NC maximal} gives the noncommutative analogue of $L^p$-norm (and $L^{1,\infty}$-quasinorm) of classical maximal functions. That is the reason why we adopt the notation in (\ref{eq:NC maximal}).
\end{remark}

\begin{definition}
	Let $1\leq p\leq\infty$. Define the space $L^p(\M;\ell^1)$ as the space of all sequences $x=\{x_i\}$ in $L^p(\M)$ which can be decomposed as
	\begin{equation*}
		x_i=\sum_{j\geq1}u_{ji}^*v_{ji},\quad \text{ for all } i,
	\end{equation*}
	for two families $\{u_{ji}\}_{i,j}$ and $\{v_{ji}\}_{i,j}$ in $L^{2p}(\M)$ such that
	\begin{equation*}
		\sum_{j,i\geq1}u_{ji}^*u_{ji}\in L^p(\M) \quad \text{ and }\quad \sum_{j,i\geq1}v_{ji}^*v_{ji}\in L^p(\M).
	\end{equation*}
	Here all sums are required to be convergent in $L^p(\M)$ (relative to the $w^*$-topology in the case of $p=\infty$). We equip this space with the following norm for $x\in L^p(\M;\ell^1)$
	\begin{equation*}
		\Norm{x}_{L^p(\M;\ell^1)}=\inf\bigg\{\Big\|\sum_{j,i\geq1}u_{ji}^*u_{ji}\Big\|_{L^p(\M)}^{1/2}\Big\|\sum_{j,i\geq1}v_{ji}^*v_{ji}\Big\|_{L^p(\M)}^{1/2}\bigg\},
	\end{equation*}
	where the infimum runs over all decompositions of $x$ as above.
\end{definition}
\par When the sequence $x=\{x_i\}_i$ is positive (i.e. $x_i\geq0$ for all $i$), then it is easy to check that $\{x_i\}\in L^p(\M;\ell^1)$ iff $\sum_ix_i\in L^p(\M)$, and
\begin{equation*}
	\Norm{\{x_i\}_i}_{L^p(\M;\ell^1)}=\Big\|\sum_ix_i\Big\|_{L^p(\M)}.
\end{equation*}
\par We have the following duality proposition.
\begin{proposition}[{\cite[Theorem 4.1.4]{PisierXu2003}}]
	\label{prop:duality between L^p l^1 and L^p l^infty}
	The spaces $L^p(\M;\ell^\infty)$ and $L^p(\M;\ell^1)$, with their norms above, are Banach spaces. Let $1\leq p<\infty$ and let $p'$ be its conjugate exponent. Then we have
	\begin{equation*}
		L^p(\M;\ell^1)^*\simeq L^{p'}(\M;\ell^{\infty})
	\end{equation*}
	isometrically, via the following duality bracket
	\begin{equation*}
		\inn{x}{y}=\sum_i\tau(x_iy_i)
	\end{equation*}
	for $x\in L^p(\M;\ell^1)$ and $y\in L^{p'}(\M;\ell^\infty)$.
\end{proposition}

\par We end this subsection with a simple result on complex interpolation of these vector-valued noncommutative $L^p$-spaces.
\begin{prop}[{\cite[Proposition 2.5]{JungeXu2007}}]
	Let $1\leq p_0<p<p_1\leq \infty$ and $0<\theta<1$ satisfy $\frac 1p=\frac{1-\theta}{p_0}+\frac{\theta}{p_1}$. Then we have isometrically
	\begin{equation*}
		L^p(\M;\ell^1)=(L^{p_0}(\M;\ell^1),L^{p_1}(\M;\ell^1))_{\theta}
	\end{equation*}
	and
	\begin{equation*}
		L^p(\M;\ell^\infty)=(L^{p_0}(\M;\ell^\infty),L^{p_1}(\M;\ell^\infty))_{\theta}.
	\end{equation*}
\end{prop}

\par For more information about vector-valued $L^p$-spaces, see \cite{PisierXu2003}.

\subsection{Bilateral almost uniform convergence}
\label{subsec:bau convergence}
\par In classical harmonic analysis, maximal inequalities extend pointwise convergence from a dense class to the whole $L^p$ space, once convergence has been proved on that class. The same principle applies in the noncommutative setting.
To introduce the noncommutative analogue of pointwise convergence, we give the definition of almost uniform convergence, which was first introduced by Lance \cite{Lance1976}.

\begin{definition}
	\label{def:bau def}
	For a sequence $\{x_n\}_n \subset L^0(\M)$, we say that $x_n \to x$ bilaterally almost uniformly (b.a.u.) if and only if for every $\epsilon>0$, there is a projection $e\in\M$ such that
	\begin{enumerate}
		\item $\Norm{e(x_n-x)e}_{\infty}\to 0$ as $n\to\infty$;
		\item $\tau(1-e)\leq\epsilon$.
	\end{enumerate}
\end{definition}
Note that for any measure space with finite measure, Definition \ref{def:bau def} coincides with the classical almost everywhere convergence by Egorov's theorem.
\begin{definition}
	\label{def:nc-c0}
	Let $1\leq p<\infty$. The space $L^p(\M;c_0)$ is defined as the space of all sequences $\{x_n\}\subset L^p(\M)$ such that there are $a,b\in L^{2p}(\M)$ and $\{y_n\}\subset\M$ satisfying
	\begin{equation*}
		x_n=ay_nb \quad \text{ and }\quad \lim_{n\to\infty}\Norm{y_n}_\infty=0.
	\end{equation*}
\end{definition}
We recall some properties of the space $L^p(\M;c_0)$.
\begin{prop}
	\label{prop:Lp(M;c0) prop}
	Suppose $1\leq p<\infty$.
	\begin{enumerate}
		\item $L^p(\M;c_0)$ is the closure in $L^p(\M;\ell^\infty)$ of finite sequences in $L^p(\M)$.
		\item $L^p(\M;c_0)$ is a closed subspace of $L^p(\M;\ell^\infty)$, and for $\{x_n\}\in L^p(\M;c_0)$,
		      \begin{equation*}
				\Big\|{\sup_n x_n}\Big\|_p=\inf\{\Norm{a}_{2p}\sup_{n}\Norm{y_n}_\infty\Norm{b}_{2p}\},
			  \end{equation*}
			  where the infimum runs over all factorizations of $\{x_n\}$ as in the definition.
		\item\label{item:c0-bau-convergence} If $\{x_n\}_n\in L^p(\M;c_0)$, then $x_n$ converges b.a.u. to $0$.
	\end{enumerate}
\end{prop}
For proofs and more details, see \cite{JungeXu2007}.
\par For a family $\{x_r\}_{0<r\leq1}\subset L^0(\M)$, b.a.u. convergence to $x$ as $r\to0$ means that, for every $\epsilon>0$, there is a projection $e\in\M$ such that $\tau(1-e)\leq\epsilon$ and $\Norm{e(x_r-x)e}_\infty\to0$ as $r\to0$. We use the following continuous-parameter version of $c_0$, with the limit taken as $r\to0$.
\begin{definition}
	\label{def:continuous-c0}
	Let $1\leq p<\infty$. The space $L^p(\M;c_0((0,1]))$ consists of all families $\{x_r\}_{0<r\leq1}$ in $L^p(\M;\ell^\infty((0,1]))$ such that
	\begin{equation}
		\label{eq:continuous-c0-tail}
		\lim_{\varepsilon\downarrow0}\Big\|\sup_{0<r<\varepsilon}x_r\Big\|_p=0.
	\end{equation}
	It is equipped with the norm inherited from $L^p(\M;\ell^\infty((0,1]))$.
\end{definition}
The properties needed for continuous-parameter convergence are as follows.
\begin{prop}
	\label{prop:continuous-c0}
	Suppose $1\leq p<\infty$.
	\begin{enumerate}
		\item $L^p(\M;c_0((0,1]))$ is a closed subspace of $L^p(\M;\ell^\infty((0,1]))$.
		\item $L^p(\M;c_0((0,1]))$ is the closure in $L^p(\M;\ell^\infty((0,1]))$ of families $\{x_r\}_{0<r\leq1}$ satisfying $x_r=0$ for all sufficiently small $r$.
		\item If $\{x_r\}_{0<r\leq1}\in L^p(\M;c_0((0,1]))$, then $x_r\to0$ b.a.u. as $r\to0$.
	\end{enumerate}
\end{prop}
\begin{proof}
	To prove (1), let $\{x_r\}$ belong to the closure of $L^p(\M;c_0((0,1]))$ with respect to its inherited norm. Fix $\epsilon>0$. By the definition of closure, we can choose $\{y_r\}\in L^p(\M;c_0((0,1]))$ such that
	\begin{equation*}
		\Big\|\sup_{0<r\leq1}(x_r-y_r)\Big\|_p<\frac{\epsilon}{2}.
	\end{equation*}
	For this fixed family $\{y_r\}$, (\ref{eq:continuous-c0-tail}) gives, for all sufficiently small $\varepsilon>0$,
	\begin{equation*}
		\Big\|\sup_{0<r<\varepsilon}y_r\Big\|_p<\frac{\epsilon}{2}.
	\end{equation*}
	For every such $\varepsilon$, the triangle inequality and (\ref{eq:J subset I maximal norm}) give
	\begin{equation*}
		\Big\|\sup_{0<r<\varepsilon}x_r\Big\|_p
		\leq\Big\|\sup_{0<r\leq1}(x_r-y_r)\Big\|_p
		+\Big\|\sup_{0<r<\varepsilon}y_r\Big\|_p<\epsilon.
	\end{equation*}
	Since $\epsilon>0$ is arbitrary, $\{x_r\}$ satisfies (\ref{eq:continuous-c0-tail}). Thus $\{x_r\}\in L^p(\M;c_0((0,1]))$, which proves (1).
	\par Every family that vanishes for all sufficiently small $r$ satisfies (\ref{eq:continuous-c0-tail}). Conversely, if $\{x_r\}\in L^p(\M;c_0((0,1]))$, then
	\begin{equation*}
		\Norm{\{x_r-x_r\chi_{[\varepsilon,1]}(r)\}_{0<r\leq1}}_{L^p(\M;\ell^\infty((0,1]))}
		=\Big\|\sup_{0<r<\varepsilon}x_r\Big\|_p\to0.
	\end{equation*}
	This proves the second assertion.
	\par To prove b.a.u. convergence, fix $\epsilon>0$. For each $n\geq1$, (\ref{eq:continuous-c0-tail}) and Definition \ref{def:NC maximal} allow us to choose $\varepsilon>0$ and a factorization $x_r=ay_rb$ for $0<r<\varepsilon$, with $a,b\in L^{2p}(\M)$, such that
	\begin{equation*}
		\sup_{0<r<\varepsilon}\Norm{y_r}_\infty\leq1,
		\qquad \Norm{a}_{2p}^{2p}+\Norm{b}_{2p}^{2p}
		\leq\epsilon2^{-n(p+1)}.
	\end{equation*}
	Here we rescale the factors so that $\Norm{a}_{2p}=\Norm{b}_{2p}$. Set
	\begin{equation*}
		e_n=\chi_{[0,2^{-n/2}]}(|a^*|)\wedge\chi_{[0,2^{-n/2}]}(|b|).
	\end{equation*}
	Then $\Norm{e_na}_\infty\leq2^{-n/2}$ and $\Norm{be_n}_\infty\leq2^{-n/2}$, so Chebyshev's inequality gives
	\begin{equation*}
		\begin{split}
			\tau(1-e_n)&\leq2^{np}\left(\Norm{a}_{2p}^{2p}+\Norm{b}_{2p}^{2p}\right)\leq\epsilon2^{-n}, \\
			\sup_{0<r<\varepsilon}\Norm{e_nx_re_n}_\infty&\leq2^{-n}.
		\end{split}
	\end{equation*}
	Put $e=\bigwedge_{n\geq1}e_n$. Then $\tau(1-e)\leq\sum_{n\geq1}\tau(1-e_n)\leq\epsilon$, and for every $n\geq1$,
	\begin{equation*}
		\limsup_{r\to0}\Norm{ex_re}_\infty\leq2^{-n}.
	\end{equation*}
	Letting $n\to\infty$ proves that $x_r\to0$ b.a.u. as $r\to0$.
\end{proof}

\section{Bessel potential multipliers and kernels}
\label{app:Bessel kernels}
\par We collect the properties used in Sections \ref{sec:l2-boundedness}--\ref{sec:submanifolds}. Throughout this appendix the ambient space is $\R^n$, with $n=2$ in the parabola case and $n=N$ in Section \ref{sec:submanifolds}. We retain the multiplier
\[
 F_s(\xi)=(1+|\xi|^2)^{s/2}, \qquad G_s=\F^{-1}[F_s],
\]
and the convention $\hat f(\xi)=\int_{\R^n}f(x)e^{-2\pi i\inn{x}{\xi}}dx$.

\subsection{Integral representation and elementary properties}
\label{app:Bessel integral}
\par For $\Re(s)<0$, the Gamma-function identity gives
\[
 F_s(\xi)=\frac1{\Gamma(-s/2)}\int_0^\infty e^{-t}e^{-t|\xi|^2}t^{-s/2-1}dt.
\]
With our Fourier convention,
\begin{equation*}
 K_t(x):=\F^{-1}[e^{-t|\xi|^2}](x)
 =\left(\frac\pi t\right)^{n/2}e^{-\pi^2|x|^2/t}, \qquad t>0.
\end{equation*}
In particular, $\int_{\R^n}K_t(x)dx=1$ and $\partial_tK_t=(4\pi^2)^{-1}\Delta K_t$.
Writing $\sigma=\Re(s)<0$, we have
\[
 \frac1{|\Gamma(-s/2)|}\int_0^\infty e^{-t}t^{-\sigma/2-1}\Norm{K_t}_{L^1(\R^n)}dt
 =\frac{\Gamma(-\sigma/2)}{|\Gamma(-s/2)|}<\infty.
\]
Thus the following integral converges in $L^1(\R^n)$, and Fubini's theorem shows that its Fourier transform is $F_s$:
\begin{equation}
 \label{eq:G_s expression}
 G_s(x)=\frac1{\Gamma(-s/2)}\int_0^\infty K_t(x)e^{-t}t^{-s/2-1}dt, \qquad x\neq0.
\end{equation}
Differentiation under the integral on compact sets away from $0$ shows that $G_s$ is smooth there. It is radial for every $\Re(s)<0$. If $s<0$ is real, the integrand is positive and strictly decreasing as a function of $|x|>0$, so $G_s$ is positive and radially decreasing. Moreover, for every $\Re(s)<0$,
\begin{equation*}
 \int_{\R^n}G_s(x)dx
 =\frac1{\Gamma(-s/2)}\int_0^\infty e^{-t}t^{-s/2-1}dt=1.
\end{equation*}
In particular, $\Norm{G_s}_{L^1}=1$ for real $s<0$. Finally, if $s=\sigma+it$ and $\sigma<0$, taking absolute values in (\ref{eq:G_s expression}) gives
\[
 |G_s(x)|\leq\frac1{|\Gamma(-s/2)|}\int_0^\infty K_u(x)e^{-u}u^{-\sigma/2-1}du
 =\frac{\Gamma(-\sigma/2)}{|\Gamma(-s/2)|}G_\sigma(x),
\]
which proves (\ref{eq:G_s domination}).

\subsection{Pointwise behavior and derivatives}
\label{app:Bessel decay}
\par Let $s<0$ be real and put $r=|x|>0$. In terms of the modified Bessel function $K_\nu$, (\ref{eq:G_s expression}) reads
\begin{equation*}
 G_s(x)=\frac{2\pi^{n/2}}{\Gamma(-s/2)}(\pi r)^{-(n+s)/2}K_{(n+s)/2}(2\pi r).
\end{equation*}
The standard asymptotics give
\begin{equation*}
 G_s(x)\sim
 \begin{cases}
  c_{s,n}r^{-n-s}, & -n<s<0,\\
  c_{s,n}\log(1/r), & s=-n,\\
  c_{s,n}, & s<-n,
 \end{cases}
 \qquad r\to0,
\end{equation*}
and
\begin{equation*}
 G_s(x)\sim C_{s,n}r^{-(n+s+1)/2}e^{-2\pi r}, \qquad r\to\infty,
\end{equation*}
where the constants are positive. For the modified Bessel representation and these asymptotics, see \cite[Chapter II, Sections 3--4, especially (4,1)--(4,3)]{AronszajnSmith1961}. The normalization here is obtained from the kernel of order $-s$ in that reference by multiplying its value at $2\pi x$ by $(2\pi)^n$.

\par The first derivative estimate needed below follows directly from the heat representation. Indeed,
\[
 |\nabla K_t(x)|=\frac{2\pi^2|x|}{t}K_t(x).
\]
For $r\geq1$, split $e^{-t-\pi^2r^2/t}$ into two equal factors. The inequality $t+\pi^2r^2/t\geq2\pi r$ bounds one factor by $e^{-\pi r}$, while the other is at most $e^{-t/2-\pi^2/(2t)}$. The latter is integrable against every power of $t$. Hence (\ref{eq:G_s expression}) and its differentiated form give
\[
 |G_s(x)|+|\nabla G_s(x)|\lesssim_s (1+r)e^{-\pi r}, \qquad r\geq1.
\]
Consequently, for every $M>0$ and $|\alpha|\leq1$,
\begin{equation}
 \label{eq:Bessel kernel decay}
 |\partial^\alpha G_s(x)|\leq C_{s,\alpha,M}(1+|x|)^{-M}, \qquad |x|\geq1.
\end{equation}

\par We also record the integral translation estimate used in Section \ref{sec:weak-11-boundedness} and its higher-dimensional analogue in Section \ref{sec:submanifolds}. Since $\Norm{\nabla K_t}_{L^1}\lesssim t^{-1/2}$ and $\Norm{K_t}_{L^1}=1$,
\[
 \Norm{K_t(\cdot-y)-K_t}_{L^1}\lesssim\min\{1,|y|t^{-1/2}\}
 \leq |y|^\varepsilon t^{-\varepsilon/2}, \qquad 0<\varepsilon\leq1.
\]
For $0<\varepsilon<\min\{1,-s\}$, integration in (\ref{eq:G_s expression}) therefore yields
\begin{equation}
 \label{eq:Bessel translation estimate}
 \Norm{G_s(\cdot-y)-G_s}_{L^1(\R^n)}
 \lesssim \frac{|y|^\varepsilon}{\Gamma(-s/2)}\int_0^\infty e^{-t}t^{-(s+\varepsilon)/2-1}dt
 \lesssim_{s,\varepsilon}|y|^\varepsilon.
\end{equation}

\subsection{General real parameters}
\label{subsec:bessel-real-parameters}
\par For every real $s$, $F_s$ is a smooth function of at most polynomial growth, so $G_s$ is a tempered distribution. Since $F_s$ is radial and $F_s(-\xi)=\overline{F_s(\xi)}$, its inverse Fourier transform is radial and real. These distributional properties suffice for the self-adjointness argument in Section \ref{sec:l2-boundedness}.

\bibliographystyle{plain}
\bibliography{references}

\end{document}